\documentclass{amsart}

\usepackage[utf8]{inputenc}
\usepackage[T1]{fontenc}
\usepackage[pdfencoding=auto, psdextra]{hyperref}
\usepackage{amsmath}
\usepackage{amssymb}
\usepackage{amsthm}
\usepackage{mathtools}  
\usepackage{microtype}
\usepackage{graphicx}
\usepackage{subcaption}
\usepackage{clrscode3e}  
\usepackage{etoolbox}
\usepackage{marginnote}
\usepackage{bbm}
\usepackage{enumitem}
\usepackage[font=small,labelfont=bf]{caption} 
\usepackage[section]{placeins} 

\newcommand*{\mailto}[1]{\href{mailto:#1}{\nolinkurl{#1}}}
\newcommand{\Dx}{{\Delta x}}
\newcommand{\R}{\mathbb{R}}
\renewcommand{\id}{\text{\normalfont{id}}}

\newcommand{\arxiv}[1]{\href{https://arxiv.org/pdf/#1}{arXiv:#1}}

\newcommand{\D}{\mathcal{D}}
\newcommand{\F}{\mathcal{F}}
\newcommand{\M}{\mathcal{M}}

\DeclareMathOperator{\sgn}{sgn}
\DeclareMathOperator*{\argmin}{arg\,min}
\DeclareMathOperator{\arcosh}{arcosh}

\apptocmd{\lim}{\limits}{}{}

\numberwithin{equation}{section}
{
    \theoremstyle{plain}
    \newtheorem{definition}{Definition}[section]
    \newtheorem{remark}[definition]{Remark}
    \newtheorem{prop}[definition]{Proposition}
    \newtheorem{example}[definition]{Example}

    \newtheorem{theorem}[definition]{Theorem}
    \newtheorem{corollary}[definition]{Corollary}
    \newtheorem{lemma}[definition]{Lemma}   
}

\title[Numerical algorithm for the HS equation]{ A convergent numerical algorithm for \fontsize{15}{17}{$\alpha$}-dissipative solutions of the Hunter--Saxton equation}

\author[T. Christiansen]{Thomas Christiansen}
\address{Department of Mathematical Sciences\\ NTNU Norwegian University of Science and Technology\\ NO-7491 Trondheim\\ Norway}
\email{\mailto{thomas.christiansen@ntnu.no}}
\urladdr{\url{https://www.ntnu.edu/employees/thomachr}}

\author[K. Grunert]{Katrin Grunert}
\address{Department of Mathematical Sciences\\ NTNU Norwegian University of Science and Technology\\ NO-7491 Trondheim\\ Norway}
\email{\mailto{katrin.grunert@ntnu.no}}
\urladdr{\url{https://www.ntnu.edu/employees/katrin.grunert}}

\author[A. Nordli]{Anders Nordli}
\address{Department of Automation and Process Engineering\\ UiT - The Arctic University of Norway\\ Troms\o \\ Norway}
\email{\mailto{anders.s.nordli@uit.no}}

\author[S. Solem]{Susanne Solem}
\address{Department of Mathematics\\ NMBU Norwegian University of Life Sciences\\ NO-1432 {\AA}s \\ Norway}
\email{\mailto{susanne.solem@nmbu.no}}
\urladdr{\url{https://www.nmbu.no/med/susanne.solem}}

\thanks{Research supported by the grants {\it Waves and Nonlinear Phenomena (WaNP)} and {\it Wave Phenomena and Stability --- a Shocking Combination (WaPheS)} from the Research Council of Norway.}  

\allowdisplaybreaks
 
\begin{document}
\maketitle

\counterwithin{equation}{section}
\raggedbottom

\begin{abstract}
   A convergent numerical method for $\alpha$-dissipative solutions of the Hunter--Saxton equation is derived. The method is based on applying a tailor-made projection operator to the initial data, and then solving exactly using the generalized method of characteristics. The projection step is the only step that introduces any approximation error. It is therefore crucial that its design ensures not only a good approximation of the initial data, but also that errors due to the energy dissipation at later times remain small. 
 Furthermore, it is shown that the main quantity of interest, the wave profile, converges in $L^{\infty}$ for all $t \geq 0$, while a subsequence of the energy density converges weakly for almost every time. 
\end{abstract}
\enlargethispage{1.5pt}

\section{Introduction}
In this article, we present a numerical algorithm for $\alpha$-dissipative solutions of the Cauchy problem for the Hunter--Saxton (HS) equation 
\begin{align}
    u_t(t,x) + uu_x(t,x) &= \frac{1}{4}\int_{-\infty}^xu_x^2(t,z)dz - \frac{1}{4}\int_{x}^{\infty}u_x^2(t,z)dz, \quad u|_{t=0} = u_0. 
    \label{eq:Hunter-Saxton}
\end{align}
The equation was derived as an asymptotic model of the director field of a nematic liquid crystal \cite{DynamicsDirector}, and possesses a rich mathematical structure. We mention a few of its properties here: it is bi-Hamiltonian and admits a Lax pair \cite{PropertiesHS}, it can be interpreted as a geodesic flow \cite{GeodesicLenells1, GeodesicLenells2}, and it admits numerous extensions and generalizations \cite{generalizedHS, PeriodicMu2HS, Generalized2HS}. 

A lot of the interest for \eqref{eq:Hunter-Saxton} is generated by the fact that weak solutions in general will develop singularities in finite time, and, consequently, they are not unique, see \cite{UniquenessConservative, alphaHS,  DynamicsDirector, HunterHyperbolicEquation}. This phenomenon is known as wave breaking. In particular, $u_{x} \rightarrow -\infty$ pointwise while the energy $\|u_x (t, \cdot)\|_2$ remains uniformly bounded, and the solution $u$, itself, H{\"o}lder continuous. Furthermore, at wave breaking, energy concentrates on a set of measure zero. Thus the energy density is in general not absolutely continuous. To overcome this problem it is common to augment the solution with a finite, positive Radon measure $\mu$ describing the energy density, see \cite{UniquenessConservative, alphaHS, LagrangianViewCH}. In particular, $\mu$ coincides with the usual kinetic energy density $u_x^2$ in regions where there is no wave breaking, thus, $d\mu_{\text{ac}}=u_x^2dx$. The energy is then described by $F(x) = \mu((-\infty, x))$. 

Weak solutions can be extended past wave breaking in various ways, see for instance \cite{BressanDissipative, UniquenessConservative, alphaHS, HunterHyperbolicEquation}. The two most prominent notions in the literature are that of a dissipative solution, where one removes all  the concentrated energy from the system, and that of a conservative solution, where one reinserts the concentrated energy. In this work, we consider the concept of $\alpha$-dissipative solutions, first introduced in \cite{alphaCH} for the related Camassa--Holm equation, and in \cite{alphaHS} for \eqref{eq:Hunter-Saxton}. Instead of removing all the concentrated energy or none of it, an $\alpha$-fraction, where $\alpha \in [0, 1]$, is removed. This way, the notion of $\alpha$-dissipative solutions acts as a continuous interpolant between the two extreme cases, $\alpha=0$ corresponding to conservative solutions, and $\alpha=1$ corresponding to dissipative solutions. Thus, the notion of $\alpha$-dissipative solutions allows for a uniform treatment of weak solutions with nonincreasing energy. The existence of $\alpha$-dissipative solutions was established in \cite{alphaHS} for the two-component Hunter--Saxton system which generalizes \eqref{eq:Hunter-Saxton}, and in the more general setting where $\alpha \in W^{1, \infty}(\R, [0, 1)) \cup \{1 \}$.

A common approach for solving the Cauchy problem \eqref{eq:Hunter-Saxton} is to use a generalized method of characteristics, see for instance \cite{UniquenessConservative, alphaHS}. 
This approach is followed here. In particular, given an $\alpha$-dissipative solution, the corresponding Lagrangian coordinates are governed by a linear system of differential equations, whose right hand side admits discontinuities at fixed times. These times can be computed a priori based on the initial data. Whence, if the initial data in Lagrangian coordinates is piecewise linear, the Lagrangian coordinates will remain piecewise linear at all later times, an important property which we take advantage of in this work.

Despite receiving a considerable amount of attention from a theoretical perspective, only a few number of numerical schemes have been proposed for the HS equation. In \cite{NumericalFDSchemes} several finite difference schemes for dissipative solutions were proved to converge. Also for dissipative solutions, \cite{NumericalDGDissipative} introduces a convergent, discontinuous Galerkin method. Furthermore, in \cite{SatoNumerical} a finite difference scheme on a periodic domain for a modified HS equation was derived and proven to converge towards conservative solutions. More recently, a Godunov-inspired scheme \cite[Chp. 12.1]{FiniteVolumeMethodsBook} for conservative solutions based on tracking the solution along characteristics, was introduced in \cite{NumericalConservative}. The scheme was proved to converge and a convergence rate prior to wave breaking was derived. 

We contribute to this line of research by introducing a numerical algorithm well-suited for $\alpha$-dissipative solutions. This algorithm is based on applying a tailor-made projection operator to the initial data, and then solving exactly along characteristics. Thus, the projection step is the only step that introduces any approximation error and it is therefore of particular importance that it not only yields a good approximation of the initial data, but also ensures that additional prospective errors, due to the energy dissipation, remain small and hence do not prevent convergence. 

To highlight the importance of the correct choice of the projection operator, we compare our projection operator with the one introduced in \cite{NumericalConservative}. 
Motivated by the fact that a piecewise linear structure of the initial data is preserved at all later times, see \cite{InverseScatteringHS, PropertiesHS}, the numerical scheme in \cite{NumericalConservative} uses a standard piecewise linear projection operator in Eulerian coordinates at every timestep. In between timesteps the numerical solution is evolved exactly along characteristics. 
The projection operator used treats $u$ and $F$ completely independently, although $u$ and $\mu$ are strongly connected through the absolutely continuous part. Consequently, a deviation is introduced in the sense that $d\mu_{\Delta x, \mathrm{ac}} \neq u_{\Delta x, x}^2dx$, where $\Delta x$ denotes the spatial discretization parameter. Thus, as pointed out in \cite{NumericalConservative}, one is no longer dealing with the Hunter--Saxton equation, but rather a reformulated version of the two-component Hunter--Saxton system \cite{Conservative2HS}, which is accompanied by a density $\rho$. 
Or, in other words, the projection operator maps into the Eulerian set for the two-component Hunter-Saxton equation, see \cite[Def. 2.2]{Conservative2HS}, for which it has been established that $\rho$ has a regularizing effect. In particular, if $\rho > 0$ in some interval $[x_1, x_2]$, then the solution will never experience wave breaking in the interval spanned by the characteristics emanating from $[x_1, x_2]$, see \cite[Thm 7.1]{GlobalDissipative2CH}. While this works out neatly in the setting of conservative solutions, since no energy is lost, for $\alpha$-dissipative solutions we are highly dependent on wave breaking actually occurring in order to remove concentrated energy. Thus, more care is needed to derive a suitable projection operator. That is why we introduce a piecewise linear projection operator $P_{\Delta x}$, which in contrast to the the one in \cite{NumericalConservative}, ensures that the projected initial data satisfies $d\mu_{\Delta x, \mathrm{ac}} = u_{\Delta x, x}^2dx$. 

We accompany the numerical algorithm with a convergence analysis, 
which shows that the limit is a weak $\alpha$-dissipative solution. In particular, we obtain that $u_{\Delta x} \rightarrow u$ in $L^{\infty}(\R)$ for all $t \geq 0$, which is the main quantity of interest in practice, while a subsequence of the energy measure $\mu_\Dx(t)$ converges weakly for almost every time. 

The outline of the paper is as follows. In Section~\ref{sec:prelim} we set the stage for the Cauchy problem of \eqref{eq:Hunter-Saxton} by defining the Eulerian set $\D$, the Lagrangian set $\F$, the mappings $L$ and $M$ between them and the notion of an $\alpha$-dissipative solution. Then in Section \ref{sec:numAlgorithm} we focus on deriving and motivating our choice of the projection operator $P_{\Delta x}$, before we discuss the practical implementation of our numerical algorithm. In Section \ref{sec:convAnalysis} we conduct the convergence analysis. In the last section, Section~\ref{sec:Numex}, we provide a few numerical experiments to illustrate the theoretical results and to go beyond the theory established here by investigating the convergence rate. 

\section{Preliminaries}\label{sec:prelim}
In this section, we set the stage for the numerical algorithm by defining the sets we are working in, as well as recalling the construction of $\alpha$-dissipative solutions by using a generalized method of characteristics, as introduced in \cite{alphaHS} and used in \cite{tandy_lipschitz}.

Assume that $\alpha\in [0,1]$ and denote by $\M^+(\R)$ the space of positive, finite Radon measures on $\R$. To define the set of Eulerian coordinates, we first need to recall some important spaces from \cite{alphaHS} and \cite{tandy_lipschitz}.

Introduce 
\begin{align*}
    E &:= \left \{ f \in L^{\infty}(\R): f' \in L^2(\R) \right\},
\end{align*}
which is a Banach space when equipped with the norm $$\|f \|_{E} = \|f \|_{\infty} + \|f' \|_{2}.$$ 
Furthermore, define
\begin{align*}
    H_d^1(\R) &:= H^1(\R) \times \R^d, \quad d = 1, 2 
\end{align*}
and introduce a partition of unity $\chi^+$ and $\chi^-$ on $\R$, i.e., a pair of functions $\chi^+, \chi^- \in C^{\infty}(\R)$ such that
\begin{itemize}
    \item $\chi^+ + \chi^- = 1$,
    \item $0 \leq \chi^{\pm} \leq 1$,
    \item $\text{supp}(\chi^+) \subset (-1, \infty)$ and $\text{supp}(\chi^-) \subset (-\infty, 1)$.
\end{itemize}
Then, we can define the following mappings 
\begin{align*}
    R_1: H_1^1(\R) \rightarrow E, & \qquad  (\bar{f}, a) \mapsto f= \bar{f} + a \chi^+, \\
    R_2: H_2^1(\R) \rightarrow E, & \qquad  (\bar{f}, a, b) \mapsto f= \bar{f} + a \chi^+ + b \chi^-,
\end{align*} 
which are linear, continuous, and injective, see \cite{HSconser}. Based on those, we introduce the Banach spaces $E_1$ and $E_2$ as the images of $H_1^1(\R)$ and $H_2^1(\R)$, respectively, that is, 
\begin{align*}
    E_1 := R_1 \left (H_1^1(\R) \right)  \quad \text{and} \quad 
    E_2 := R_2 \left(H_2^1(\R) \right) 
\end{align*}
and endow them with the norms
\begin{align*}
    \|f\|_{E_1} &:= \|\bar{f} + a \chi^+ \|_{E_1} = \left( \|\bar{f}\|_{H^1(\R)}^2 + a^2 \right )^{\frac{1}{2}}, \\
    \|f\|_{E_2} &:= \|\bar{f} + a \chi^+ + b\chi^- \|_{E_2} = \left ( \|\bar{f}\|_{H^1(\R)}^2 + a^2 + b^2\right )^{\frac{1}{2}}.
\end{align*}
 By construction the spaces $E_1$ and $E_2$ do not rely on the particular choice of $\chi^+$ and $\chi^-$, cf. \cite{nonvanCH}. Furthermore, observe that the mapping $R_1$ is also well-defined for functions in $L_1^2(\R) = L^2(\R) \times \R$ and therefore, we set
\begin{align*}
    E_1^0 := R_1 \left(L_1^2(\R) \right),
\end{align*}
and equip this space with the norm
\begin{align*}
    \|f \|_{E_1^0} &:= \|\bar{f} + a \chi^+ \|_{E_1^0} = \left( \|\bar{f} \|_2^2 + a^2 \right)^{\frac{1}{2}}.
\end{align*}

At last, we can define the set of Eulerian coordinates, $\D$, in which we also seek numerical solutions.

\begin{definition}\label{def:EulerianSet}
    The space $\D$ consists of all triplets $(u, \mu, \nu)$ such that
    \begin{enumerate}[label=(\roman*)]
        \item $u \in E_2$, 
        \item $\mu \leq \nu \in \M^{+}(\R)$,
       \item $\mu_{\mathrm{ac}}\leq \nu_{\mathrm{ac}}$,
        \item $d\mu_{\mathrm{ac}} = u_x^2dx$,\label{def:EulerianSetlist:condition3}
        \item $\mu\left( (-\infty, \cdot) \right) \in E_1^0$,
        \item $\nu \left((-\infty, \cdot) \right)  \in E_1^0$,
    \label{def:EulerianSetlist:condition5}
        \item If $\alpha = 1$, then $d\nu_{\mathrm{ac}}=d\mu = u_x^2dx$,
        \item If $\alpha \in [0, 1)$, then $\frac{d\mu}{d\nu}(x) \in \{1 - \alpha, 1 \}$, and $\frac{d\mu_{\mathrm{ac}}}{d\nu_{\mathrm{ac}}}(x) = 1$ if $u_x(x) < 0$. 
    \end{enumerate}
\end{definition}

As $\mu$, $\nu \in \M^+(\R)$, we can define the primitive functions $F(x) = \mu((-\infty, x))$ and $G(x) = \nu((- \infty, x))$. These are bounded, increasing, left-continuous and satisfy
\begin{align*}
    \lim_{x \rightarrow - \infty}F(x) &= \lim_{x \rightarrow - \infty} G(x) = 0. 
\end{align*}
We will interchangeably use the notation $(u, F, G)$ and $(u, \mu, \nu)$ to refer to the same triplet in $\D$, since by \cite[Thm. 1.16]{RealAnalysisFolland}, there is a one-to-one correspondence between $(F, G)$ and $(\mu, \nu)$.

Moreover, for practical purposes, we will often restrict the initial data to belong to 
\begin{align*}
    \D_0 :&= \{ (u, \mu, \nu) \in \D: \mu = \nu \} \subset \D. 
\end{align*}
 
Let $B=E_2 \times E_2 \times E_1 \times E_1$ endowed with the norm 
\begin{align*}
   \|(f_1, f_2, f_3, f_4) \|_B := \|f_1 \|_{E_2} + \|f_2 \|_{E_2} + \|f_3 \|_{E_1} + \|f_4 \|_{E_1},
\end{align*}
then the set of Lagrangian coordinates, $\F$, is defined as follows.

\begin{definition}\label{def:DefinitionLagrangian} The set $\F$ consists of all quadruplets $X = (y, U, V, H)$ with $(y-\id, U, V, H) \in B$ such that
\begin{enumerate}[label=(\roman*)]
  \item $(y-\id, U, V, H) \in \big[W^{1, \infty}(\R) \big]^4$,
  \item $ y_{\xi}, H_{\xi} \geq 0$ and there exists $c > 0$ such that $y_{\xi} + H_{\xi} \geq c $ holds a.e.,
  \item \label{def:DefinitionLagrangian:ImporantRel} $y_{\xi}V_{\xi}= U_{\xi}^2 $ a.e., 
  \item $0 \leq V_{\xi} \leq H_{\xi}$ a.e., 
  \item If $\alpha = 1$, then $y_{\xi}(\xi)=0$ implies that $V_{\xi}(\xi)=0$, and $y_{\xi}(\xi) > 0$ implies that $V_{\xi}(\xi) = H_{\xi}(\xi)$ a.e.,
  \item If $\alpha \in [0, 1)$, then there exists a function $\kappa: \R \rightarrow \{1 - \alpha, 1 \}$ such that $V_{\xi}(\xi) = \kappa(\xi)H_{\xi}(\xi)$ a.e., and $\kappa(\xi)=1$ whenever $U_{\xi}(\xi) < 0$. 
\end{enumerate}
\end{definition}
In the convergence analysis, the following subsets of $\F$ will play an important role
\begin{align*}
    \F^0 &= \{X \in \F: H(\xi) = V(\xi) \text{ for all }\xi \in \R \}, \\
    \F_0 &= \{ X \in \F: y + H = \id \},
\end{align*}
 and  $$\F_{0}^0 = \F^0 \cap \F_0.$$ 

To construct the $\alpha$-dissipative solution using a generalized method of characteristics means to study the time evolution in Lagrangian rather than in Eulerian coordinates, and therefore the mappings between $\D$ and $\F$ are an essential part.

\begin{definition} \label{def:MappingL}Let $L: \D \rightarrow \F_0$ be defined by $L \left((u, \mu, \nu) \right)= (y, U, V, H)$, where
\begin{subequations}
\begin{alignat}{3}
             y(\xi) &= \sup \{x \in \R: x + \nu((-\infty, x)) < \xi \}, \label{eq:L_eq1}\\
    U(\xi) &= u (y(\xi)), \label{eq:L_eq2} \\
    H(\xi) &= \xi - y(\xi), \label{eq:L_eq3} \\
    V(\xi) &= \int_{-\infty}^{\xi} \frac{d\mu}{d\nu} (y(\eta)) H_{\xi}(\eta) d\eta 
    \nonumber.
        \end{alignat}
\end{subequations}
\end{definition}
Next we introduce the mapping taking us from Lagrangian to Eulerian coordinates. To this end, recall that the pushforward of a Borel measure $\lambda$ by a measurable function $f$ is the measure $f_{\#}\lambda$ defined for all Borel sets $A\subset \R$ by 
\begin{equation*}
	(f_{\#}\lambda)(A) = \lambda(f^{-1}(A)).
\end{equation*}

\begin{definition}\label{def:MappingM} Define $M: \F \rightarrow \D$ by $M((y, U, V, H)) = (u, \mu, \nu)$, where 
\begin{subequations}
\begin{align*}
    u(x) &= U(\xi) \quad \text{ for all } \xi \in \mathbb{R} \text{ such that } x = y(\xi),  \\
    \mu &= y_{\#} \big(V_{\xi}d\xi \big), \\
    \nu &= y_{\#} \big(H_{\xi} d\xi \big). 
    \end{align*}
\end{subequations}
\end{definition}

For proofs that these mappings are well-defined we refer to \cite[Prop. 2.1.5 and 2.1.7]{phdthesisNordli}. Furthermore, it should be noted that the triplets $(u,\mu,\nu)$ are mapped to quadruplets $(y,U,V,H)$ and hence there cannot be a one-to-one correspondence between Eulerian and Lagrangian coordinates. However, as pointed out in \cite{phdthesisNordli}, one can identify equivalence classes in Lagrangian coordinates, so that each equivalence class corresponds to exactly one element in Eulerian coordinates. 

Moreover, it should be pointed out that all the important information in Eulerian coordinates is encoded in the pair $(u,\mu)$, and hence contained in the triplet $(y,U,V)$ in Lagrangian coordinates. In contrast, the mapping $L$ relies heavily on $\nu$ and hence changing $\nu$ changes not only $H$, but also $(y,U,V)$. 

Finally, we can turn our attention to the time evolution. The $\alpha$-dissipative solution in Lagrangian coordinates, $X(t) = (y, U, V, H)(t)$, with initial data $X(0)=X_0\in \F$, is the unique solution to the following system of differential equations
\begin{subequations}
\label{eq:integrated_ODE}
\begin{align}
    y_t(t, \xi) &= U(t, \xi), \label{eq:integrated_ODE1} \\
    U_t(t, \xi) &= \frac{1}{2}V(t, \xi) - \frac{1}{4}V_{\infty}(t), \label{eq:integrated_ODE2} \\
    V(t, \xi) &= \int_{-\infty}^{\xi}V_{0, \xi}(\eta) \big(1 - \alpha \chi_{\{t \geq \tau(\eta) > 0\}}(\eta) \big) d\eta, \label{eq:integrated_ODE3} \\
    H_t(t, \xi) &= 0. \label{eq:integrated_ODE4}
\end{align}
Here $\tau:\R \rightarrow [0, \infty]$ is the wave breaking function given by
\end{subequations}
\begin{align}
    \tau(\xi) &= \begin{cases}- \frac{2y_{0, \xi}(\xi)}{U_{0, \xi}(\xi)}, & \text{if } U_{0, \xi}(\xi) < 0,\\
    0, & \text{if }y_{0, \xi}(\xi) = U_{0, \xi}(\xi) = 0, \\ 
    \infty, & \text{otherwise},
    \end{cases}
\label{eq:waveBreakingFunction_Lagrangian}
\end{align}
and $V_{\infty}(t) = \lim_{\xi \rightarrow \infty}V(t, \xi)$ denotes the total Lagrangian energy at time $t$. 

For a proof of the uniqueness of the solution to \eqref{eq:integrated_ODE}, we refer to \cite[Lem. 2.3]{alphaHS}. Furthermore, it should be pointed out that the solution operator respects equivalence classes in the following sense: If $X_{A,0}$ and $X_{B,0}$ belong to the same equivalence class, then also $X_A(t)$ and $X_B(t)$ belong to the same equivalence class for all $t\geq 0$, see \cite[Prop. 3.7]{alphaHS}. 

Based on \eqref{eq:integrated_ODE} we now define the solution operator $S_t$ in Lagrangian coordinates.

\begin{definition}
For any $t\geq 0$ and $X_0\in \F$ define $S_t(X_0)=X(t)$, where $X(t)$ denotes the unique $\alpha$-dissipative solution to \eqref{eq:integrated_ODE} with initial data $X(0)=X_0$.
  \end{definition}
  
 To finally obtain the $\alpha$-dissipative solution in Eulerian coordinates, we combine the solution operator $S_t$ with the mappings $L$ and $M$ as follows. 
 
 \begin{definition}\label{sol:Euler}
 For any $t\geq 0$ and $(u_0,\mu_0,\nu_0)\in \D$ the $\alpha$-dissipative solution at time $t$ is given by 
 \begin{equation*}
 (u, \mu, \nu)(t)= T_t((u_0,\mu_0,\nu_0))=M\circ S_t\circ L((u_0, \mu_0, \nu_0)).
 \end{equation*}
 \end{definition}
 
 As we mentioned earlier $L((u,\mu,\nu))$ is heavily influenced by the choice of $\nu$ in the triplet $(u,\mu,\nu)$. Nevertheless, it has been shown in \cite[Lem. 2.13]{NewstLipschitzMetrichAlpha}, that the choice of $\nu$ has no influence on the solution, in the following sense. Given any two triplets of initial data $(u_{A,0}, \mu_{A,0}, \nu_{A,0})$ and $(u_{B,0}, \mu_{B,0}, \nu_{B,0})$ in $\D$ such that 
 \begin{equation*}
 u_{A,0}=u_{B,0} \quad \text{ and }\quad \mu_{A,0}=\mu_{B,0},
 \end{equation*}
 then
 \begin{equation*}
 u_A(t)=u_B(t)\quad \text{ and }\quad \mu_A(t)=\mu_B(t) \quad \text{ for all } t\geq 0.
 \end{equation*}
 As a consequence, we restrict ourselves from now on to consider only initial data $(u_0,\mu_0,\nu_0)\in \D_0$, which means in particular that $\mu_0=\nu_0$.
 
 Furthermore, in the case of conservative and dissipative solutions the uniqueness of weak solutions has been established in \cite{UniquenessConservative} and \cite{UniquenessDissipative}, respectively, by showing that if a weak solution of the Hunter--Saxton equation satisfies certain properties, which are heavily dependent on the class of solutions one is interested in, then it can be computed using Definition~\ref{sol:Euler}. For the remaining values of $\alpha$, i.e., $\alpha\in (0,1)$, this is yet an open question.

\section{The numerical algorithm}\label{sec:numAlgorithm}

This section is devoted to presenting our numerical algorithm, which combines a projection operator $P_{\Delta x}$ with the solution operator $T_t$ as follows. 

\begin{definition}\label{def:numericalSolution}
    We define the numerical solution $(u_{\Delta x}, F_{\Delta x}, G_{\Delta x})$ for $t \in [0,T]$ by 
    \begin{align*}
        (u_{\Delta x}, F_{\Delta x}, G_{\Delta x})(t) &= T_t\circ P_{\Delta x}\left((u_0,F_0,G_0) \right)=M \circ S_{t} \circ L \circ P_{\Delta x} \left((u_0, F_0, G_0) \right),
    \end{align*}
    for any $(u_0, F_0, G_0) \in \D_0$. 
\end{definition}

As the projection operator $P_{\Delta x}$ is the only part that actually introduces any error, its construction is the crucial step. In particular, $P_{\Delta x}$ must preserve key properties of $\alpha$-dissipative solutions such as the total energy $\mu(\mathbb{R})$ and Definition \ref{def:EulerianSet}~\ref{def:EulerianSetlist:condition3}. Beside the construction of $P_{\Delta x}$, we will, at the end of this section, discuss the implementation of our algorithm.

\subsection{The projection operator}  
Any $\mu \in \M^+(\R)$ can be split into an absolutely continuous part $\mu_{\text{ac}}$, and a singular part $\mu_{\text{sing}}$ with respect to the Lebesgue measure, see \cite[Prop. 9.8]{McDonalRealAnalysis}, i.e.,
\begin{align*}
    d\mu &= d\mu_{\mathrm{ac}} +  d\mu_{\mathrm{sing}}.
\end{align*}
Thus $F(x)=\mu((-\infty,x))$ can be written as 
\begin{equation}\label{split:F}
F(x)=F_{\mathrm{ac}}(x)+ F_{\mathrm{sing}}(x),
\end{equation}
where  $F_{\mathrm{ac}}(x) = \mu_{\mathrm{ac}}((-\infty, x))$ and $F_{\mathrm{sing}}(x) = \mu_{\mathrm{sing}}((- \infty, x))$. Thus a projection operator acting on $F(x)$ can be a combination of two projections, one for $F_{\mathrm{ac}}(x)$ and another one for $F_{\mathrm{sing}}(x)$. As we will see, this is how we define  $P_{\Delta x}$.

To derive $P_{\Delta x}$, let $\{x_j \}_{j \in \mathbb{Z}}$ be a uniform discretization of $\R$, where $x_j = j \Delta x$ for $j \in \mathbb{Z}$ and $\Delta x > 0$ is fixed. Furthermore, set $P_{\Delta x}\left((u, F, G)\right) = (u_{\Delta x}, F_{\Delta x}, G_{\Delta x})$. To ensure that $(u_{\Delta x}, F_{\Delta x}, G_{\Delta x})$ satisfies Definition \ref{def:EulerianSet}~\ref{def:EulerianSetlist:condition3}, the projection operator is defined over two grid cells, i.e., over $[x_{2j}, x_{2j+1}] \cup [ x_{2j+1}, x_{2j+2}]$. Furthermore, to preserve the continuity of $(u, F_{\mathrm{ac}})$ and the total energy, we require that  $\left (u_{\Delta x}, F_{\Delta x, \text{ac}} \right)$ coincides with $(u, F_{\mathrm{ac}})$ in every other gridpoint, i.e., 
\begin{align}
    u_{\Delta x}(x_{2j}) &= u(x_{2j}), \nonumber \\
    F_{\Delta x, \mathrm{ac}}(x_{2j}) &= F_{\text{ac}}(x_{2j}).
    \label{eq:Projec_no_accum}
\end{align}

To interpolate $u_{\Delta x}$ and $F_{\Delta x}$ between the gridpoints $\{x_{2j} \}_{j \in \mathbb{Z}}$, we fit two lines, $p_1$ and $p_2$,
such that the resulting wave profile
\begin{align*}
    u_{\Delta x}(x) &= \begin{cases}
    p_1(x), & x_{2j} \leq x < x_{2j+1}, \\
    p_2(x), & x_{2j+1} \leq x \leq x_{2j+2}, \end{cases}
\end{align*}
is continuous, $(u_{\Delta x}, F_{\Delta x, \mathrm{ac}})$ satisfies \eqref{eq:Projec_no_accum}, and 
\begin{align*}
    F_{\Delta x, \mathrm{ac}}(x) &= \int_{- \infty}^x u_{\Delta x, x}^2(y)dy.
\end{align*} 
For an arbitrary $j \in \mathbb{Z}$ these constraints then read
\begin{subequations}
\begin{align}
    p_1(x_{2j}) &= u(x_{2j}), \label{eq:Projec_constraint1}\\
    p_2(x_{2j+2}) &= u(x_{2j+2}), \label{eq:Projec_constraint2}\\
    p_1(x_{2j+1}) &= p_2(x_{2j+1}), \label{eq:Projec_constraint3} \\
    \int_{x_{2j}}^{x_{2j+1}} \left(p_1'(y) \right)^2dy + \int_{x_{2j+1}}^{x_{2j+2}} \left(p_2'(y) \right)^2dy &= F_{\mathrm{ac}}(x_{2j+2}) - F_{\mathrm{ac}}(x_{2j}) . \label{eq:Projec_constraint4}
\end{align}
\end{subequations}

Introducing the operators 
\begin{align*}
    D_+ f_j &:= \frac{f_{j+1} - f_j}{\Delta x}, \\  
    Df_{2j} &:= \frac{D_+ f_{2j+1} + D_+ f_{2j}}{2} = \frac{f_{2j+2} - f_{2j}}{2\Delta x}, 
\end{align*}
for any sequence $\{f_j \}_{j \in \mathbb{Z}}$
and solving 
\eqref{eq:Projec_constraint1}--\eqref{eq:Projec_constraint4}, we end up with 
 \begin{align*}
     u_{\Delta x}(x) &= \begin{cases}
    u(x_{2j}) + \left(Du_{2j} \mp  q_{2j}\right)(x-x_{2j}), & x_{2j} \leq x \leq x_{2j+1}, \\
    u(x_{2j+2}) + \left(Du_{2j} \pm  q_{2j}\right)(x-x_{2j+2}), & x_{2j+ 1} \leq x \leq x_{2j+2}, 
    \end{cases} \\ F_{\Delta x, \mathrm{ac}}(x) &= \begin{cases} F_{\mathrm{ac}}(x_{2j}) + \left(Du_{2j} \mp q_{2j}\right)^2 \left(x - x_{2j} \right), & x_{2j} < x \leq x_{2j+1}, \\ 
    \frac{F_{\mathrm{ac}}(x_{2j+2}) + F_{\mathrm{ac}}(x_{2j})}{2} \mp 2Du_{2j} q_{2j}\Delta x \\ \quad + \left(Du_{2j} \pm q_{2j} \right)^2(x-x_{2j+1}), & x_{2j+1} < x \leq x_{2j+2},  \end{cases} 
 \end{align*}
 where 
 \begin{align}
     q_{2j} &:= \sqrt{DF_{\mathrm{ac}, 2j} - \left(Du_{2j} \right)^2}.\label{eq:short_hand_notation} 
 \end{align} 
 Note that \eqref{eq:short_hand_notation} is well-defined, as we have by the Cauchy--Schwarz inequality 
 \begin{align*}
     DF_{\mathrm{ac}, 2j} - \left(Du_{2j} \right)^2 & = \frac{1}{2\Delta x} \int_{x_{2j}}^{x_{2j+2}}u_x^2(z)dz - \left(\frac{1}{2\Delta x} \int_{x_{2j}}^{x_{2j+2}}u_x(z)dz \right)^2 \\ & \geq \frac{1}{2\Delta x} \int_{x_{2j}}^{x_{2j+2}}u_x^2(z)dz - \frac{1}{2\Delta x} \int_{x_{2j}}^{x_{2j+2}}u_x^2(z)dz = 0.
 \end{align*}
  
 For the singular part, we set
\begin{align*}
   F_{\Delta x, \mathrm{sing}}(x) &:=  F(x_{2j+2}) - F_{ \mathrm{ac}}(x_{2j+2})= F_{\mathrm{sing}}(x_{2j+2}), \quad \text{ for } x \in (x_{2j}, x_{2j+2}],
\end{align*}
where the last equality follows from \eqref{split:F}. Hence, all the discontinuities of $F_{\Delta x, \text{sing}}(x)$ are located within the countable set $\{x_{2j} \}_{j \in \mathbb{Z}}$. Finally, introduce 
\begin{equation*}
F_{\Delta x}(x) = F_{\Delta x, \mathrm{ac}}(x) + F_{\Delta x, \mathrm{sing}}(x),
\end{equation*}
which is left-continuous. Then we can associate to $F_{\Delta x}$ a positive and finite Radon measure $\mu_{\Delta x}$ through 
\begin{equation*}
 \mu_{\Delta x}((- \infty, x)) = F_{\Delta x}(x), 
 \end{equation*}
 and, by construction, 
 \begin{equation*}
  \mu_{\Delta x,\mathrm{ac}}((-\infty, x))=F_{\Delta x,\mathrm{ac}}(x), \quad \text{ and }\quad  \mu_{\Delta x,\mathrm{sing}}((- \infty, x))=F_{\Delta x,\mathrm{sing}}(x).
  \end{equation*} 

To summarize, the projection operator $P_{\Delta x}$ is defined as follows.

\begin{definition}[Projection operator]\label{def:ProjectionOperator}
    We define the projection operator $P_{\Delta x}:\D_0\to \D_0$ by $P_{\Delta x} \left((u,F,G) \right)=(u_{\Delta x},F_{\Delta x}, G_{\Delta x})$, where
    \begin{align}
    u_{\Delta x}(x) &= \begin{cases}
    u(x_{2j}) + \left(Du_{2j} \mp  q_{2j}\right)(x-x_{2j}), & x_{2j} < x \leq x_{2j+1}, \\
    u(x_{2j+2}) + \left(Du_{2j} \pm  q_{2j}\right)(x-x_{2j+2}), & x_{2j+ 1} < x \leq x_{2j+2}, 
    \end{cases}
    \label{eq:Projected_u} \end{align}
    and 
    \begin{align*}
       G_{\Delta x}(x) &= F_{\Delta x}(x) = F_{\Delta x, \mathrm{ac}}(x) + F_{\Delta x, \mathrm{sing}}(x),  
    \end{align*}
    for all $x \in \R$. The absolutely continuous part of $F_{\Delta x}$ is given by 
    \begin{align}
    F_{\Delta x, \mathrm{ac}}(x) &= \begin{cases} F_{\mathrm{ac}}(x_{2j}) + \left(Du_{2j} \mp q_{2j}\right)^2 \left(x - x_{2j} \right), & x_{2j} < x \leq x_{2j+1}, \\ 
    \frac{F_{\mathrm{ac}}(x_{2j+2}) + F_{\mathrm{ac}}(x_{2j})}{2} \mp 2q_{2j}Du_{2j} \Delta x \\ \quad + \left(Du_{2j} \pm q_{2j} \right)^2(x-x_{2j+1}), & x_{2j+1} < x \leq x_{2j+2},
    \end{cases}
    \label{eq:Projected_F_ac} \end{align} 
    and the singular part is defined by 
    \begin{align}
    F_{\Delta x, \mathrm{sing}}(x) &= F(x_{2j+2}) - F_{ \mathrm{ac}}(x_{2j+2})=F_{\mathrm{sing}}(x_{2j+2}), \qquad  x_{2j} < x \leq x_{2j+2}. \label{eq:singular_F} 
    \end{align}
    Finally, let $\mu_{\Delta x} ((- \infty, x))=F_{\Delta x}(x)$ and $\nu_{\Delta x}((- \infty, x)) = G_{\Delta x}(x)$ be the unique, finite and positive Radon measures associated with $F_{\Delta x}$ and $G_{\Delta x}$, respectively. 
\end{definition} 

\begin{remark}
The projection operator $P_{\Delta x}$ in Definition~\ref{def:ProjectionOperator} is only defined for $(u,\mu,\nu)\in \D_0$ due to Definition~\ref{def:numericalSolution}. In this case $\mu=\nu$ and by construction $\mu_{\Delta x}= \nu_{\Delta x}$. Nevertheless, $P_{\Delta x}$ can  be extended to a projection operator $\tilde P_{\Delta x}$, which is well-defined for any $(u,\mu,\nu) \in\D$ as follows.  
We keep $u_{\Delta x}$ and $\mu_{\Delta x}((-\infty,x))=F_{\Delta x}(x) $ given by \eqref{eq:Projected_u}--\eqref{eq:singular_F}, but $\nu_{\Delta x}((-\infty,x))=G_{\Delta x}(x)$ needs to be adjusted: 

Following the same idea as for $F_{\Delta x, \mathrm{sing}}$, we set
\begin{align*}
    G_{\Delta x, \mathrm{sing}}(x) &= G(x_{2j+2}) - G_{ \mathrm{ac}}(x_{2j+2})=G_{\mathrm{sing}}(x_{2j+2}) , \quad \text{ for } x \in (x_{2j}, x_{2j+2}],
\end{align*}
      Since we now have $d\mu_{\mathrm{ac}} = u_{x}^2dx$ and $\mu_{\mathrm{ac}} \leq \nu_{\mathrm{ac}}$, which implies that $d\nu_{\mathrm{ac}} = f dx $ for some $f \geq u_x^2$, we approximate the deviation $(f - u_x^2)$ with the integral average over two grid cells. Following these lines, we end up with 
\begin{align*}
    d\nu_{\Delta x, \mathrm{ac}} &= \left(u_{\Delta x, x}^2 + \frac{1}{2\Delta x}\int_{x_{2j}}^{x_{2j+2}} \left(f(y) - u_x^2(y) \right)dy \right)dx \\ &= \left(u_{\Delta x, x}^2 + \left(DG_{\mathrm{ac}, 2j} - DF_{\mathrm{ac}, 2j} \right) \right)dx,
\end{align*}
for $x\in (x_{2j}, x_{2j+1})\cup(x_{2j+1}, x_{2j+2})$. Computing $G_{\Delta x, \mathrm{ac}}(x)=\nu_{\Delta x, \mathrm{ac}}((-\infty,x))$, yields $G_{\Delta x, \mathrm{ac}}(x_{2j}) = G_{\mathrm{ac}}(x_{2j})$, and
\begin{align*}
     G_{\Delta x, \mathrm{ac}}(x) &= 
    G_{\mathrm{ac}}(x_{2j})+F_{\Delta x, \mathrm{ac}}(x) - F_{\mathrm{ac}}(x_{2j})  +  \left(DG_{\mathrm{ac}, 2j} - DF_{\mathrm{ac}, 2j} \right) (x-x_{2j}),
\end{align*}
for $x \in (x_{2j}, x_{2j+2}]$. Finally, we define $\nu_{\Delta x}$ implicitly by
\begin{equation*}
\nu_{\Delta x}((-\infty,x))= G_{\Delta x}(x)=G_{\Delta x, \mathrm{ac}}(x)+G_{\Delta x, \mathrm{sing}}(x).
\end{equation*}

Last but not least, note that there is one drawback with $\tilde P_{\Delta x}$. It preserves all properties in Definition~\ref{def:EulerianSet} when $d\mu_{\mathrm{ac}} = d\nu_{\mathrm{ac}}$. However, in the case $d\mu_{\mathrm{ac}} \neq d\nu_{\mathrm{ac}}$, the Radon--Nikodym derivative $\frac{d\mu_{\Delta x}}{d\nu_{\Delta x}}$ belongs to the interval $[1-\alpha, 1]$ rather than the set $\{1-\alpha, 1 \}$, and hence the relation $\frac{d\mu_{\mathrm{ac}}}{d\nu_{\mathrm{ac}}}(x) =1$ whenever $u_x(x) < 0$ in Definition \ref{def:EulerianSet}~\ref{def:EulerianSetlist:condition5} will not be obeyed in general. Therefore, $\tilde P_{\Delta x}$ maps into a larger space $\hat{\D} \supset \D$. Nevertheless, $\tilde P_{\Delta x}$ is relevant for numerical algorithms for nonconservative solutions ($\alpha\not =0$), which in a similar spirit to the one in \cite{NumericalConservative}, are based on applying the projection operator after each time step $\Delta t$. 
\end{remark}

A closer look at \eqref{eq:Projected_u} and \eqref{eq:Projected_F_ac} reveals that there are two possible sign choices on each interval $[x_{2j}, x_{2j+2}]$. Especially on a coarse grid, it is vital to make the right choice, as the following example shows.

\begin{example}\label{ex:Example1}
Consider the tuple $(u, F, G)$, where 
\begin{align*}
    u(x) &= \begin{cases}
    0, & x \leq 0, \\
    x, & 0 < x \leq 1, \\
    2-x, & 1 < x \leq 2, \\
    0, & 2 < x, \end{cases} \\
    F(x) &= \begin{cases}
    0, & x\leq 0, \\
    x + 1 -\alpha, & 0 < x \leq 2, \\
    3 - \alpha, & 2 < x, \end{cases} \\
    G(x) &= \begin{cases}
    0, & x\leq 0, \\
    x + 1 , & 0 < x \leq 2, \\
    3 , & 2 < x. \end{cases}
\end{align*}
This models a scenario where wave breaking takes place at $x=0$. 

We discretize the domain $[-\frac{1}{2}, \frac{5}{2}]$ with $\Delta x = \frac{1}{2}$, but for convenience we use $x_j = - \frac{1}{2} + j \Delta x$ as our gridpoints. Using \eqref{eq:Projected_u} and \eqref{eq:Projected_F_ac} we find 
\begin{equation}
\begin{aligned}
    u_{\Delta x}(x) &= \begin{cases}
    \frac{1}{2}\left(1 \mp 1\right)(x+\frac{1}{2}), & -\frac{1}{2} < x \leq 0, \\
    \frac{1}{2} + \frac{1}{2}\left(1 \pm 1\right)(x-\frac{1}{2}) , & 0 < x \leq \frac{1}{2}, \\
    \frac{1}{2} \mp (x- \frac{1}{2}), & \frac{1}{2} < x \leq 1, \\
    \frac{1}{2} \pm (x - \frac{3}{2}), & 1 < x \leq \frac{3}{2}, \\
    \frac{1}{2} + \frac{1}{2}\left(-1 \mp 1\right)(x - \frac{3}{2}), & \frac{3}{2} < x \leq 2, \\
    \frac{1}{2}\left(-1 \pm 1\right)(x-\frac{5}{2}), & 2 < x \leq \frac{5}{2},
    \end{cases} 
    \label{eq:example_approx_full_1}
\end{aligned}
\end{equation}

\begin{equation}
\begin{aligned}
    F_{\Delta x, \mathrm{ac}}(x) &= \begin{cases}
    \frac{1}{4}\left(1 \mp 1\right)^2(x + \frac{1}{2}), & - \frac{1}{2} < x \leq 0, \\
    \frac{1}{4} \mp\frac{1}{4} + \frac{1}{4}\left(1 \pm 1 \right)^2x, & 0 < x \leq \frac{1}{2}, \\ 
    x, & \frac{1}{2}< x \leq 1, \\
   x, & 1 < x \leq \frac{3}{2}, \\
   \frac{3}{2} + \frac{1}{4}\left(- 1 \mp 1 \right)^2 (x - \frac{3}{2}), & \frac{3}{2} < x \leq 2, \\
   \frac{7}{4} \pm \frac{1}{4} + \frac{1}{4}\left(-1 \pm 1 \right)^2(x-2), & 2 < x \leq \frac{5}{2}. \end{cases}
\label{eq:example_approx_full_2}
\end{aligned}
\end{equation}

\begin{figure}
    \centering
\includegraphics[width=0.8\textwidth, trim={0cm 0.6cm 0cm 1cm},clip]{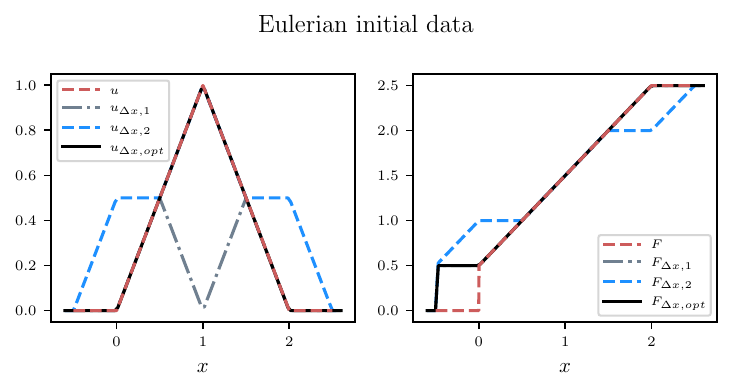}
    \caption{A comparison of the tuple $(u, F, G)$ from Example \ref{ex:Example1} with the projected data $(u_{\Delta x}, F_{\Delta x}, G_{\Delta x})$ for different sign-choices for $\alpha=\frac12$ on a grid with $\Delta x = \frac{1}{2}$ ($x$ on the horizontal axis).}
    \label{fig:ex1_eulerian_initial}
\end{figure}
There are several possible ways to choose the signs in \eqref{eq:example_approx_full_1}--\eqref{eq:example_approx_full_2} and some of those are shown in Figure \ref{fig:ex1_eulerian_initial}: $u_{\Delta x, 1}$ is based on using minus sign over $[x_{2j}, x_{2j+1}]$ and plus over $[x_{2j+1}, x_{2j+2}]$, while $u_{\Delta x, 2}$ is based on using plus over $[x_{2j}, x_{2j+1}]$ and minus  over $[x_{2j+1}, x_{2j+2}]$. By carefully choosing the signs, $u_{\Delta x}$ fully overlaps with the exact solution. This choice which we denote as $u_{\Delta x, \text{opt}}$, is given by  
\begin{align*}
    u_{\Delta x, \mathrm{opt}}(x) &= \begin{cases} 
    0, & -\frac{1}{2} < x \leq 0, \\
    x, & 0 < x \leq 1, \\
    2 - x, & 1 < x \leq 2, \\
    0, & 2 < x \leq \frac{5}{2}, \end{cases} \nonumber \\
    F_{\Delta x, \mathrm{ac}, \mathrm{opt}}(x) &= \begin{cases}
    0, &  - \frac{1}{2} < x \leq 0, \\
    x, & 0 < x \leq 2, \\
    2, & 2<  x \leq \frac{5}{2}, \end{cases}
\end{align*}
and equals $(u, F_{\mathrm{ac}})$. The singular parts are approximated by 
\begin{align*}
    F_{\Delta x, \mathrm{sing}}(x) &= \begin{cases}
    0, & x \leq -\frac{1}{2}, \\
    1 - \alpha, & -\frac{1}{2} < x, 
    \end{cases} \nonumber \\
    G_{\Delta x, \mathrm{sing}}(x) &= \begin{cases}
    0, & x \leq -\frac{1}{2}, \\
    1, &-\frac{1}{2} < x. 
    \end{cases}
\end{align*}
\end{example}

As the above example shows, choosing the best sign on each of the two grid cells in $[x_{2j}, x_{2j+2}]$  is vital, since it significantly affects the accuracy in the projection step and hence the whole algorithm, especially for a coarse grid. Instead of trial and error, which is never a good choice except for illustrative purposes, usually a selection criterion is imposed. In the implementation we decided to choose the sign over 
     $[x_{2j}, x_{2j+1}]$ that minimizes the distance between $u_{\Delta x}(x_{2j+1})$ and  $u(x_{2j+1})$, which can be formalized as follows. Introduce $m \in \{0, 1\}$ such that for $x \in [x_{2j}, x_{2j+1}]$ we may write \eqref{eq:Projected_u} as 
\begin{align*}
    u_{\Delta x}(x) &= u(x_{2j}) + \left(Du_{2j} + \left(-1 \right)^{m} q_{2j}\right)(x-x_{2j}).
\end{align*}
Subtracting $u_{\Delta x}(x_{2j+1})$ from $u(x_{2j+1})$, and finding the $m$ that minimizes the distance can then be expressed as  
\begin{align*}
    k_{2j} :&= \argmin_{m \in \{0, 1\}} \left \{ \left|\left(D_{+}u_{2j} - Du_{2j} \right)\Delta x  + \left(-1\right)^{m+1}q_{2j}\Delta x \right| \right \}. 
\end{align*}
Note that the sign over $[x_{2j+1}, x_{2j+2}]$ will be $(-1)^{k_{2j}+1}$.

\subsection{Numerical implementation of \texorpdfstring{$T_t$}{the solution operator from Definition~\ref{sol:Euler}}}\label{subsec:Implementation}
Since $T_t$ associates to each piecewise linear initial data, the corresponding solution at time $t$, which again is piecewise linear, our numerical implementation of $T_t$ will yield the exact solution with the projected initial data. 

Fix a discretization parameter $\Delta x > 0$ and an initial datum $(u, F, G)(0) \in \D_0$. 
Following Definition~\ref{def:numericalSolution}, the numerical Lagrangian initial data is given by 
\begin{align}
    X_{\Delta x}(0) &= (y_{\Delta x}, U_{\Delta x}, V_{\Delta x}, H_{\Delta x})(0) = L\circ P_{\Delta x} \left( (u, F, G)(0) \right),
    \label{eq:numerical_Lagrangian_initial data}
\end{align}
and we therefore first focus on the numerical implementation of $L$. We will observe that each component of $X_{\Delta x}(0)$ is again a piecewise linear function. Moreover, the associated Lagrangian grid is non-uniform and has possibly a larger number of gridpoints than the original Eulerian grid. 

\subsubsection{Implementation of $L$}
To avoid any ambiguity between breakpoints of a function and points of wave breaking, we denote the former as nodes in the following. 

By construction, $x+G_{\Delta x}(0, x)$ is an increasing, piecewise linear function with nodes situated at the points $\{x_j\}_{j\in \mathbb{Z}}$, and hence $y_{\Delta x}(0,\xi)$ is again an increasing and piecewise linear function due to Definition~\ref{def:MappingL}. Furthermore, $y_{\Delta x}(0,\xi)$ is continuous and its nodes can be identified by finding all $\xi$ which satisfy $y_{\Delta x}(0,\xi)=x_j$ for some $j\in \mathbb{Z}$ as follows. By Definition~\ref{def:MappingL}  
\begin{align*}
y_{\Delta x}(0,\xi)+ G_{\Delta x}(0, y_{\Delta x}(0,\xi)) &\leq y_{\Delta x}(0,\xi)+ H_{\Delta x}(0,\xi) \\ &\leq y_{\Delta x}(0,\xi)+ G_{\Delta x}(0, y_{\Delta x}(0,\xi)+),
\end{align*}
and, due to $X_{\Delta x}(0)\in \F_0$, we have
\begin{equation}\label{def:xxi}
y_{\Delta x}(0,\xi)+ G_{\Delta x}(0, y_{\Delta x}(0,\xi)) \leq\xi \leq y_{\Delta x}(0,\xi)+ G_{\Delta x}(0, y_{\Delta x}(0,\xi)+).
\end{equation}
Since $G_{\Delta x}(0, x)$ is continuous except possibly at the points $\{x_{2j} \}_{j \in \mathbb{Z}}$, 
to every $x_{2j+1}$ with $j\in \mathbb{Z}$, there exists a unique $\xi_{2j+1}$ such that $y_{\Delta x}(0, \xi_{2j+1})=x_{2j+1}$ and, using \eqref{def:xxi}, $\xi_{2j+1}$ is given by
\begin{equation*}
x_{2j+1}+G_{\Delta x}(0, x_{2j+1})=\xi_{2j+1}.
\end{equation*}
At the points $\{x_{2j}\}_{j\in \mathbb{Z}}$, the function $G_{\Delta x}(0, x)$ might have a jump. Therefore, there exists a maximal interval $I_{2j}=[\xi_{2j}^{l}, \xi_{2j}^{r}]$ such that $y_{\Delta x}(0,\xi)=x_{2j}$ for all $\xi \in I_{2j}$, and using once more \eqref{def:xxi}, $\xi_{2j}^{l}$ and $\xi_{2j}^{r}$ are given by 
\begin{equation}
x_{2j}+G_{\Delta x}(0, x_{2j})= \xi_{2j}^{l}\leq \xi_{2j}^{r}= x_{2j}+G_{\Delta x}(0, x_{2j}+).
\label{eq:jump_point}
\end{equation}
Note that $\xi_{2j}^{l}=\xi_{2j}^{r}$ if and only if $G_{\Delta x}(0, x)$ has no jump at $x=x_{2j}$.

Set 
\begin{equation}
\hat\xi_{3j}= \xi_{2j}^{l}, \quad \hat\xi_{3j+1}=\xi_{2j}^{r}, \quad \text{ and } \quad \hat \xi_{3j+2}=\xi_{2j+1}\quad \text{ for }j\in \mathbb{Z}.
\label{eq:LagrangianGridpoints}
\end{equation} 
Then the nodes of $y_{\Delta x} (0,\xi)$ are situated at the points $\{\hat\xi_{j}\}_{j\in \mathbb{Z}}$, as shown in Figure~\ref{fig:jumps_in_G_constant_in_y}. As mentioned earlier $y_{\Delta x}(0)$ is piecewise linear and continuous, and hence also $U_{\Delta x}(0)$, $V_{\Delta x}(0)$, and $H_{\Delta x}(0)$ are piecewise linear and continuous. Moreover, the nodes of $X_{\Delta x}(0)$ are located at $\{\hat\xi_{j}\}_{j\in \mathbb{Z}}$, since the nodes of $(u_{\Delta x}, F_{\Delta x}, G_{\Delta x})(0)$ are situated at $x=x_{j}$, $j\in \mathbb{Z}$. Thus, once the nodes are identified and the values of $X_{\Delta x}(0)$ at these nodes are determined using Definition~\ref{def:MappingL}, the value of $X_{\Delta x}(0,\xi)$ can be computed at any point $\xi$. 

Furthermore, each two grid cells $[x_{2j}, x_{2j+1}],[x_{2j+1}, x_{2j+2}]$ are mapped to 3 grid cells $[\hat \xi_{3j}, \hat\xi_{3j+1}], [\hat\xi_{3j+1}, \hat\xi_{3j+2}], [\hat\xi_{3j+2}, \hat\xi_{3j+3}]$  by $L$.
In addition, while the Eulerian discretization is uniform with size  
$\Delta x$, we now have 
\begin{align*}
\vert \hat \xi_{3j+1}-\hat\xi_{3j}\vert& =\vert  \xi_{2j}^{r}-\xi_{2j}^{l}\vert = G_{\Delta x}(0, x_{2j}+)-G_{\Delta x}(0, x_{2j}), \\
\vert \hat \xi_{3j+2}-\hat\xi_{3j+1}\vert&=\vert \xi_{2j+1}-\xi_{2j}^{r}\vert=\Delta x+ G_{\Delta x}(0, x_{2j+1})-G_{\Delta x}(0, x_{2j}+),\\
\vert \hat\xi_{3j+3}-\hat\xi_{3j+2}\vert&= \vert \xi_{2j+2}-\xi_{2j+1}\vert = \Delta x+ G_{\Delta x}(0, x_{2j+2})-G_{\Delta x}(0, x_{2j+1}),
\end{align*}
and hence the Lagrangian discretization is non-uniform.

\begin{figure}
    \centering
    \includegraphics[width=0.8\textwidth]{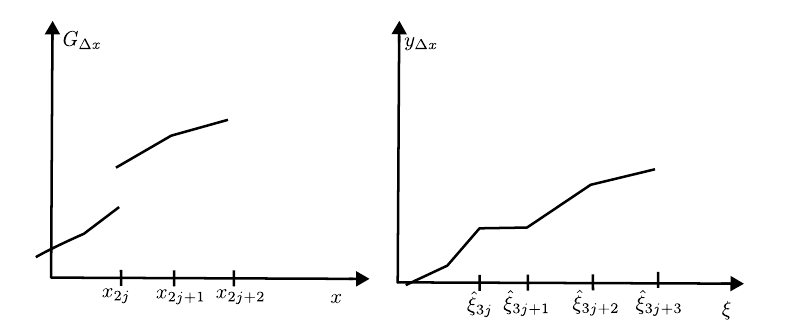}
     \captionsetup{width=0.88\textwidth}
    \caption{An illustration of how jumps in $G_{\Delta x}$ are mapped into intervals where  $y_{\Delta x}$ is constant.}
    \label{fig:jumps_in_G_constant_in_y}
\end{figure}

\subsubsection{The wave breaking function}
 The next step is to compute the numerical wave breaking function $\tau_{\Delta x}: \R \rightarrow [0, \infty]$, using \eqref{eq:waveBreakingFunction_Lagrangian}. For $\xi\in [\hat\xi_{3j}, \hat\xi_{3j+1}]$, we have $y_{\Delta x, \xi}(0,\xi) = U_{\Delta x, \xi} (0,\xi)= 0$ and hence $\tau_{\Delta x}(\xi)=0$ by \eqref{eq:waveBreakingFunction_Lagrangian}. For $\xi\in (\hat\xi_{3j+1}, \hat\xi_{3j+2})$, note that $U_{\Delta x,\xi}(0,\xi)$ has the same sign as $u_{\Delta x,x}(0, x)= Du_{2j}\mp q_{2j}$ in $(x_{2j}, x_{2j+1})$. Likewise, for $\xi \in (\hat\xi_{3j+2}, \hat\xi_{3j+3})$, $U_{\Delta x,\xi}(0,\xi)$ has the same sign as $u_{\Delta x,x}(0, x)= Du_{2j}\pm q_{2j}$ in $(x_{2j+1}, x_{2j+2})$. Furthermore, $U_{\Delta x,\xi}(0,\xi)=u_{\Delta x,x}(0, y_{\Delta x}(0,\xi))y_{\Delta x,\xi}(0,\xi)$ for all $\xi \in (\hat\xi_{3j+1}, \hat\xi_{3j+3})\backslash\{ \hat\xi_{3j+2}\}$. Therefore, introducing $\tau_{3j+ \frac{3}{2}}$ and $\tau_{3j+ \frac{5}{2}}$ as
\begin{align*}
    \tau_{3j+ \frac{3}{2}} &= \begin{cases} -\frac{2}{Du_{2j} \mp q_{2j}}, & \text{if } Du_{2j} \mp q_{2j} < 0, \\ 
    \infty, & \text{otherwise}, \end{cases} \\
    \tau_{3j+\frac{5}{2}} &=\begin{cases} -\frac{2}{Du_{2j} \pm q_{2j}}, & \text{if } Du_{2j} \pm q_{2j} < 0, \\ 
    \infty, & \text{otherwise}, \end{cases}
\end{align*} 
the wave breaking function $\tau_{\Delta x}$ for $\xi \in [\hat\xi_{3j}, \hat\xi_{3j+3})$ is given by 
\begin{align*}
    \tau_{\Delta x}(\xi) &= \begin{cases} 
    0, & \xi \in [\hat\xi_{3j}, \hat\xi_{3j+1}], \\
    \tau_{3j+\frac{3}{2}}, & \xi \in (\hat\xi_{3j+1}, \hat\xi_{3j+2}], \\
    \tau_{3j + \frac{5}{2}}, & \xi \in (\hat\xi_{3j+2}, \hat\xi_{3j+3}). \end{cases}
\end{align*} 
Figure \ref{fig:waveBreaking} illustrates the relation between the slopes of $U_{\Dx}(0)$ and the value attained by $\tau_{\Dx}$. Note that an interval where $U_{\Dx}(0)$ is strictly increasing leads to an interval where $\tau_{\Dx}$ is unbounded.
\begin{figure}
	 \includegraphics[width=0.8\textwidth]{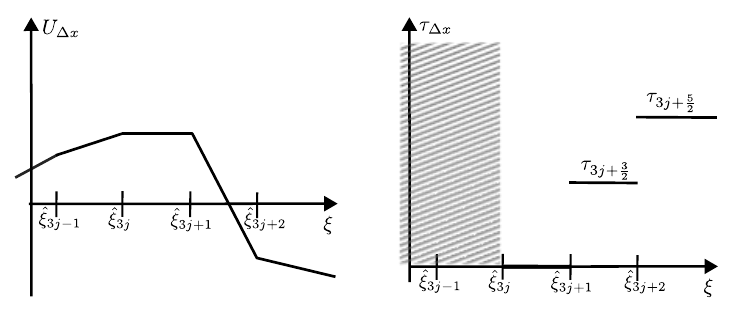}
	  \captionsetup{width=0.88\textwidth}
	 \caption{The projected Lagrangian wave profile $U_{\Dx}$ consists of increasing, decreasing and constant segments. The right plot visualizes the corresponding wave breaking function $\tau_{\Dx}$, with the shaded region representing an interval where $\tau_{\Dx}(\xi)=\infty$.}
	 \label{fig:waveBreaking}
\end{figure}

\subsubsection{Implementation of \texorpdfstring{$S_t$}{the Lagrangian solution operator}}
We proceed by considering $X_{\Delta x}(t)=S_t(X_{\Delta x}(0))$.  Introduce the function 
\begin{equation*}
\zeta_{\Delta x} = y_{\Delta x} - \id, 
\end{equation*}
which we prefer to work with  as $\zeta_{\Delta x}\in L^{\infty}(\R)$ in contrast to $y_{\Delta x}$, see Definition~\ref{def:DefinitionLagrangian}. Furthermore, let 
\begin{equation*}
\hat{X}_{\Delta x}=(\zeta_{\Delta x}, U_{\Delta x}, V_{\Delta x}, H_{\Delta x}).
\end{equation*}
It then turns out to be advantageous to compute the time evolution of $\hat{X}_{\Delta x,\xi}$ rather than the one of $\hat{X}_{\Delta x}$, 
due to possible drops in $V_{\Delta x,\xi}$, cf. \eqref{eq:integrated_ODE3}. However, this forces us to slightly change our point of view, since by differentiating \eqref{eq:numerical_Lagrangian_initial data} we find that $\hat{X}_{\Delta x, \xi}(0)$ is a piecewise constant function, whose discontinuities are situated at the nodes $\hat\xi_j$, where $j\in \mathbb{Z}$. Therefore, we associate to each $j \in \mathbb{Z}$, 
\begin{equation*}
\hat{X}_{j, \xi}(0)= \begin{cases}\hat{X}_{\Delta x, \xi} \left(0, \tfrac12 (\hat \xi_{j}+\hat\xi_{j+1}) \right), & \text{ if }\hat\xi_{j}\not = \hat\xi_{j+1}, \\
(0,0,0,0), & \text{ otherwise.}
\end{cases}
\end{equation*}
This sequence is evolved numerically according to, 
\begin{align}
    \zeta_{j, t\xi}(t) &= U_{j, \xi}(t), \nonumber \\ 
    U_{j, t\xi}(t) &= \frac{1}{2}V_{j, \xi}(t), \nonumber \\ 
    V_{j, \xi}(t) &= \left(1 - \alpha \chi_{\{t \geq \tau_{\Delta x}(\xi) > 0\}} \left(\tfrac12(\hat \xi_j+\hat\xi_{j+1}) \right) \right)V_{j, \xi}(0), \nonumber \\ 
    H_{j, t \xi}(t) &= 0. 
\label{eq:LagrangianSystem}
\end{align}
The above system is obtained by differentiating \eqref{eq:integrated_ODE} with respect to $\xi$. Since computing $\hat{X}_{\Delta x, \xi}(t,\cdot)$  exactly, yields a piecewise constant function whose discontinuities are again located at the nodes $\{\hat \xi_j\}_{j \in \mathbb{Z}}$, considering the above sequence $\{\hat {X}_{j, \xi}(t)\}_{j \in \mathbb{Z}}$ does not yield any additional error. Furthermore, the exact $\hat{X}_{\Delta x, \xi}(t,\cdot)$ can be read off from $\{\hat {X}_{j, \xi}(t)\}_{j \in \mathbb{Z}}$.

Finally, we can exactly recover $\hat{X}_{\Delta x}(t)$, which is continuous with respect to $\xi$, from $\{\hat{X}_{j,\xi}(t)\}_{j \in \mathbb{Z}}$. Since the asymptotic behavior of $\hat{X}_{\Delta x}(t,\xi)$ as $\xi\to \pm\infty$ changes in accordance with \eqref{eq:integrated_ODE}, this must also be taken into account by our algorithm. The fact that the initial data is in the space $\mathcal{D}$ combined with \eqref{eq:integrated_ODE3} implies
\begin{equation*}
 V_{\Delta x, -\infty} (t)=0 \quad \text{ for all } t\geq 0,
\end{equation*}
where the abbreviation
\begin{equation*}
f_{\pm \infty}(t)=\lim_{\xi\to\pm\infty} f(t,\xi),
\end{equation*}
has been introduced to ease the notation. Hence, for all $j\in \mathbb{Z}$,
\begin{equation*}
V_{\Delta x}(t, \hat \xi_j) = \sum_{k=-\infty}^{j-1} V_{k, \xi}(t) \left(\hat \xi_{k+1}-\hat\xi_{k} \right)= V_{\Delta x}(t, \hat \xi_{j-1})+ V_{j-1, \xi}(t) \left(\hat \xi_{j}-\hat \xi_{j-1} \right).
\end{equation*}
From \eqref{eq:integrated_ODE4} it follows that 
\begin{equation*}
H_{\Delta x}(t, \hat \xi_j)= H_{\Delta x}(0, \hat \xi_j), \quad \text{ for all } t\geq 0.
\end{equation*}
For $U_{\Delta x}(t,\xi)$, we have, combining \eqref{eq:integrated_ODE2} and Fubini's theorem, 
 \begin{align*}
    U_{\Delta x,-\infty}(t) &= U_{\Delta x,-\infty}(0) - \frac{1}{4} \int_0^t V_{\Delta x, \infty}(s)ds \nonumber \\ 
    &= U_{\Delta x, -\infty}(0) -\frac{1}{4}H_{\Delta x, \infty}(0)t \nonumber \\ \nonumber
     & \qquad + \frac{1}{4}\alpha \int_{\R} H_{\Delta x, \xi}(0, \xi)\int_{\tau_{\Delta x}(\xi)}^t \chi_{\{s \geq \tau_{\Delta x}(\xi)>0 \}}(\xi)ds d\xi\\ \nonumber
    & =  U_{\Delta x, -\infty}(0) -\frac{1}{4}H_{\Delta x, \infty}(0)t \nonumber \\ \nonumber
    & \qquad + \frac{1}{4}\alpha\sum_{k=-\infty}^\infty H_{k, \xi}(0) \left(\hat\xi_{k+1}-\hat\xi_k \right) \left(t-\tau_{\Delta x} \left(\tfrac12 (\hat \xi_k+\hat \xi_{k+1}) \right) \right) \\
    & \qquad \qquad \qquad \times \chi_{ \left\{s \geq \tau_{\Delta x}\left(\tfrac12(\hat\xi_k+\hat\xi_{k+1}) \right) \right\}}(t)\chi_{ \left\{j: \tau_{\Delta x}\left(\tfrac12(\hat\xi_j+\hat\xi_{j+1}) \right)>0 \right\}}(k),
\end{align*}
and for all $j\in \mathbb{Z}$
\begin{align*}
 U_{\Delta x}(t, \hat\xi_j) &= U_{\Delta x,-\infty}(t) + \sum_{k=-\infty}^{j-1} U_{k, \xi}(t)\left(\hat \xi_{k+1}-\hat\xi_{k} \right) \nonumber \\ 
 & =U_{\Delta x}(t, \hat \xi_{j-1}) + U_{j-1, \xi}(t) \left(\hat\xi_{j} - \hat{\xi}_{j-1} \right).
\end{align*}
For $\zeta_{\Delta x}(t,\xi)$, we  find, using \eqref{eq:integrated_ODE1} and $\zeta_{\Delta x,-\infty}(0) = 0$, 
\begin{align*}
    \zeta_{\Delta x,-\infty}(t) &= \int_0^t U_{\Delta x,-\infty}(s)ds \\ 
    &= U_{\Delta x,-\infty}(0)t - \frac{1}{8}\alpha H_{\Delta x, \infty}(0)t^2 \\ & \qquad + \frac{1}{8}\alpha \sum_{k=-\infty}^\infty H_{k, \xi}(0) \left(\hat \xi_{k+1}-\hat \xi_{k} \right)\left(t-\tau_{\Dx} \left(\tfrac12 (\hat \xi_k+\hat \xi_{k+1}) \right) \right)^2\\
    & \qquad \qquad \qquad \times \chi_{ \left\{s \geq \tau_{\Delta x}\left(\tfrac12(\hat\xi_k+\hat\xi_{k+1}) \right) \right\}}(t)\chi_{ \left\{j: \tau_{\Delta x}\left(\tfrac12(\hat\xi_j+\hat\xi_{j+1})\right)>0 \right\}}(k),
\end{align*}
and for all $j\in \mathbb{Z}$, 
\begin{align*}
 \zeta_{\Delta x}(t, \hat{\xi}_j) &= \zeta_{\Delta x,-\infty}(t) + \sum_{k=-\infty}^{j-1} \zeta_{k, \xi}(t) \left(\hat \xi_{k+1}-\hat\xi_{k} \right) \nonumber \\ 
 & = \zeta_{\Delta x}(t, \hat \xi_{j-1})+ \zeta_{j-1, \xi}(t) \left(\hat\xi_j - \hat{\xi}_{j-1} \right).
\end{align*}
Recalling that $\hat X_{\Delta x}(t)$ is piecewise linear and continuous, with nodes situated at $\hat \xi_j$ with $j\in \mathbb{Z}$, we can now recover $\hat X_{\Delta x}(t,\xi)$ for all $\xi\in \mathbb{R}$, and, using that $y_{\Delta x}(t)= \zeta_{\Delta x}(t)+ \id$, we end up with $X_{\Delta x}(t,\xi)$ for all $\xi \in \mathbb{R}$.

In practice, we have to limit the numerical approximation of the Cauchy problem \eqref{eq:Hunter-Saxton} to initial data where $u_{x}(0)$ and $\mu_{\mathrm{sing}}(0)$ have compact support. For such initial data there exists $N \in \mathbb{N}$ such that 
\begin{equation*}
\zeta_{j,\xi}(0)=U_{j,\xi}(0)=V_{j, \xi}(0) = H_{j, \xi}(0) = 0 \quad \text{for all }\vert j\vert \geq N,
\end{equation*} and by \eqref{eq:LagrangianSystem} we have
\begin{equation*}
\zeta_{j,\xi}(t)=U_{j,\xi}(t)=V_{j, \xi}(t) = H_{j, \xi}(t) = 0 \quad \text{for all }\vert j\vert \geq N \text{ and } t\geq 0.
\end{equation*}

\subsubsection{Implementation of $M$} Finally, to recover the solution in Eulerian coordinates we apply the mapping $M$ to $X_{\Delta x}(t)$, i.e., 
\begin{align*}
    \left (u_{\Delta x}, \mu_{\Delta x}, \nu_{\Delta x}\right)(t) &= M\left(X_{\Delta x}(t) \right).
\end{align*}
Here it is important to note that $X_{\Delta x}(t)$ is piecewise linear and continuous, thus  $u_{\Delta x}(t,\cdot)$ is also piecewise linear and continuous, while $F_{\Delta x}(t, \cdot)$ and $G_{\Delta x}(t, \cdot)$ are piecewise linear, increasing and in general only left-continuous. Furthermore, their nodes are situated at the points $\{y_{\Delta x}(t, \hat \xi_j)\}_{j \in \mathbb{Z}}$. Numerically, we therefore apply a piecewise linear reconstruction. 

Given $x \in \R$, there exists $j \in \mathbb{Z}$ such that $x \in (y_{\Delta x}(t, \hat{\xi}_j), y_{\Delta x}(t, \hat{\xi}_{j+1})]$.

If $x=y_{\Delta x}(t, \hat{\xi}_{j+1})$, we have, by the definition of $M$,
\begin{equation*}
u_{\Delta x}(t,x)=U_{\Delta x}(t, \hat \xi_{j+1}), \hspace{0.16cm}  F_{\Delta x}(t,x)=V_{\Delta x}(t, \hat \xi_{j+1}), \hspace{0.16cm} \text{and}\hspace{0.16cm} G_{\Delta x}(t,x)= H_{\Delta x} (t, \hat \xi_{j+1}).
\end{equation*}
This covers also the case where $\hat{\xi}_j = \hat{\xi}_{j+1}$, and the case where wave breaking occurs, i.e., $y_{\Delta x}(t, \hat{\xi}_j) = y_{\Delta x}(t, \hat{\xi}_{j+1})$.

If $x\in  (y_{\Delta x}(t, \hat{\xi}_j), y_{\Delta x}(t, \hat{\xi}_{j+1}))$, observe that  
\begin{equation*}
F_{\Delta x}(t, y_{\Delta x}(t, \hat{\xi}_j)+)= V_{\Delta x}(t, \hat \xi_j), \quad \text{ and }\quad V_{\Delta x}(t, \hat \xi_{j+1})=F_{\Delta x}(t, y_{\Delta x}(t, \hat{\xi}_{j+1})),
\end{equation*}
and
\begin{equation*}
G_{\Delta x}(t, y_{\Delta x}(t, \hat{\xi}_j)+)= H_{\Delta x}(t, \hat \xi_j), \quad \text{ and }\quad H_{\Delta x}(t, \hat \xi_{j+1})=G_{\Delta x}(t, y_{\Delta x}(t, \hat{\xi}_{j+1})). 
\end{equation*}
Therefore, for $x\in (y_{\Delta x}(t, \hat{\xi}_j), y_{\Delta x}(t, \hat{\xi}_{j+1}))$ the linear interpolation, which coincides with  $M\left(X_{\Delta x}(t) \right)$, is given by
\begin{align}
    u_{\Delta x}(t,x) &= \frac{y_{\Delta x}(t, \hat{\xi}_{j+1}) - x}{D_{+}^{\xi}y_{\Delta x}(t, \hat{\xi}_{j})}U_{\Delta x}(t, \hat{\xi}_{j}) + \frac{x - y_{\Delta x}(t, \hat{\xi}_{j})}{D_{+}^{\xi}y_{\Delta x}(t, \hat{\xi}_{j})}U_{\Delta x}(t, \hat{\xi}_{j+1}), \nonumber \\
    F_{\Delta x}(t, x) &= \frac{y_{\Delta x}(t, \hat{\xi}_{j+1}) - x}{D_{+}^{\xi}y_{\Delta x}(t, \hat{\xi}_{j})}V_{\Delta x}(t, \hat{\xi}_{j}) + \frac{x - y_{\Delta x}(t, \hat{\xi}_{j})}{D_{+}^{\xi}y_{\Delta x}(t, \hat{\xi}_{j})}V_{\Delta x}(t, \hat{\xi}_{j+1}), \nonumber \\
    G_{\Delta x}(t, x) &= \frac{y_{\Delta x}(t, \hat{\xi}_{j+1}) - x}{D_{+}^{\xi}y_{\Delta x}(t, \hat{\xi}_{j})}H_{\Delta x}(t, \hat{\xi}_{j}) + \frac{x - y_{\Delta x}(t, \hat{\xi}_{j})}{D_{+}^{\xi}y_{\Delta x}(t, \hat{\xi}_{j})}H_{\Delta x}(t, \hat{\xi}_{j+1}),
\label{eq:piecewise_linear_recon}
\end{align}
where $D_+^{\xi}y_{\Delta x}(t, \hat{\xi}_j) = y_{\Delta x}(t, \hat{\xi}_{j +1 }) - y_{\Delta x}(t, \hat{\xi}_{j})$. Finally, note that the Eulerian grid changes significantly with respect to time. First of all the number of grid cells is changing over time and secondly, the grid does not remain uniform. 

\section{Convergence of the numerical method}\label{sec:convAnalysis}

In this section we prove that our family of approximations $\{ \left(u_{\Delta x}, F_{\Delta x}, G_{\Delta x} \right) \}_{\Delta x > 0}$ converges to the $\alpha$-dissipative solution of \eqref{eq:Hunter-Saxton}. We start by proving convergence of the projected initial data in Eulerian coordinates. Thereafter, we show that this induces convergence, initially and at later times, towards the unique $\alpha$-dissipative solution in Lagrangian coordinates. Finally, we investigate in what sense convergence in Lagrangian coordinates carries over to Eulerian coordinates. 

\subsection{Convergence of the initial data in Eulerian coordinates}
Let $P_{\Delta x}$ be the projection operator given by Definition \ref{def:ProjectionOperator}. We have the following result.

\begin{prop}\label{prop:prop_projection}
    For $(u,\mu, \nu) \in \D_0$, let $(u_{\Delta x}, F_{\Delta x}, G_{\Delta x})=P_{\Delta x} \left((u,F, G) \right)$. Then 
 \begin{subequations}
    \begin{align}
        \|u - u_{\Delta x}\|_{\infty} & \leq \left(1 + \sqrt{2} \right)\sqrt{F_{\mathrm{ac}, \infty}} \Delta x^{\frac{1}{2}}, \label{eq:projection_u_infty}\\
        \|u - u_{\Delta x}\|_{2} & \leq \sqrt{2}\left(1+\sqrt{2}\right)\sqrt{F_{\mathrm{ac}, \infty}} \Delta x, \label{eq:projection_u_L2}  \\ 
        \| F - F_{\Delta x}\|_p &\leq 2 F_{\infty}\Delta x^{\frac{1}{p}},  \quad \text { for } p=1,2,\label{eq:projection_F_Lp}\\
        \|G - G_{\Delta x} \|_p & \leq 2 G_{\infty} \Delta x^{\frac{1}{p}}, \quad \text{ for } p=1,2, \label{eq:projection_G_Lp} \\
         F_{\Delta x}(x) &\rightarrow F(x) \text{ for every } x \text{ at which } F \text{ is continuous}, \label{eq:vague_conv_initially_F} \\
         G_{\Delta x}(x) &\rightarrow G(x) \text{ for every } x \text{ at which } G \text{ is continuous}. \label{eq:vague_conv_initially_G}
    \end{align}
    \end{subequations}
\end{prop}

\begin{proof} 
Let $x \in [x_{2j}, x_{2j+1}]$. By \eqref{eq:Projected_u}, we have   
\begin{align*}\nonumber
    \left |u(x) - u_{\Delta x}(x) \right| &\leq \left|  \frac{x_{2j+2} - x}{2\Delta x} \left(u(x) - u(x_{2j}) \right) + \frac{x- x_{2j}}{2\Delta x}\left(u(x) - u(x_{2j+2}) \right) \right| \\   & \qquad + q_{2j}(x- x_{2j}),
\end{align*}
where $q_{2j}$ is defined by \eqref{eq:short_hand_notation}. The second term is bounded by 
\begin{align}
    0\leq q_{2j}(x-x_{2j}) 
    & \leq \sqrt{DF_{\mathrm{ac}, 2j}}\Delta x \leq \sqrt{\frac{F_{\mathrm{ac}}(x_{2j+2}) - F_{\mathrm{ac}}(x_{2j})}{2}} \Delta x^{\frac{1}{2}}. 
    \label{eq:estimate_square_root}
\end{align}
The first term can be estimated as follows: 
\begin{align}
    \biggl | \frac{x_{2j+2}-x}{2\Delta x}&\left(u(x) - u(x_{2j}) \right) + \frac{x-x_{2j}}{2\Delta x}\left(u(x) - u(x_{2j+2}) \right) \biggr | \nonumber \\
    & \leq \int_{x_{2j}}^{x}\left|u_x(z)\right|dz + \frac12\int_{x}^{x_{2j+2}}\left|u_x(z)\right|dz \nonumber \\
    &\leq \sqrt{x-x_{2j}}\sqrt{F_{\mathrm{ac}}(x)-F_{\mathrm{ac}}(x_{2j})} + \frac{1}{2}\sqrt{x_{2j+2}-x}\sqrt{F_{\mathrm{ac}}(x_{2j+2})-F_{\mathrm{ac}}(x)} \nonumber 
    \\ &\leq \left(1 + \sqrt{2} \right)\sqrt{\frac{F_{\mathrm{ac}}(x_{2j+2}) - F_{\mathrm{ac}}(x_{2j})}{2}} \Delta x^{\frac{1}{2}},  
\label{eq:estimate_u_polation}
\end{align}
where we applied the Cauchy--Schwarz inequality. Combining \eqref{eq:estimate_square_root} and \eqref{eq:estimate_u_polation} yields 
\begin{equation}\label{est:differenceu}
  \left |u(x) - u_{\Delta x}(x) \right| \leq (1+\sqrt{2})\sqrt{F_{\mathrm{ac}}(x_{2j+2}) - F_{\mathrm{ac}}(x_{2j})}\Delta x^{\frac{1}{2}}.
  \end{equation}
  By a similar argument for $x \in [x_{2j+1}, x_{2j+2}]$, \eqref{eq:projection_u_infty} is established. 

 Using \eqref{est:differenceu} we get
\begin{align*}
    \|u-u_{\Delta x} \|_2^2 & = \sum_{j \in \mathbb{Z}} \left( \int_{x_{2j}}^{x_{2j+1}}\left(u(x) - u_{\Delta x}(x)\right)^2dx + \int_{x_{2j+1}}^{x_{2j+2}}\left(u(x) - u_{\Delta x}(x)\right)^2dx \right) \\
     & \leq \sum_{j \in \mathbb{Z}} 2\left(1+\sqrt{2}\right)^2  \left(F_{\mathrm{ac}}(x_{2j+2}) - F_{\mathrm{ac}}(x_{2j}) \right)  \Delta x^2 \\
     & = 2 \left(1 + \sqrt{2} \right)^2 F_{\mathrm{ac}, \infty}\Delta x^2,
\end{align*}
which yields \eqref{eq:projection_u_L2}.

Next, we show \eqref{eq:projection_F_Lp}. For the $L^1(\R)$-estimate, we have
\begin{align*}
    \|F - F_{\Delta x} \|_{1} &= \sum_{j \in \mathbb{Z}} \int_{x_{2j}}^{x_{2j+2}}|F(x) - F_{\Delta x}(x)|dx \leq \sum_{j \in \mathbb{Z}} \int_{x_{2j}}^{x_{2j+2}} \left(F(x_{2j+2}) - F(x_{2j}) \right)dx \nonumber \\  & = 2\sum_{j \in \mathbb{Z}} \left (F(x_{2j+2}) - F(x_{2j}) \right)\Delta x = 2F_{\infty}\Delta x,
\end{align*}
and for the $L^2(\R)$-estimate,
\begin{align*}
    \|F - F_{\Delta x}\|_2^2 &\leq 2 \sum_{j \in \mathbb{Z}}\left(F(x_{2j+2}) - F(x_{2j}) \right)^2\Delta x  \nonumber \\
    & \leq 2 \sum_{j \in \mathbb{Z}}F_{\infty} \left(F(x_{2j+2}) - F(x_{2j}) \right)\Delta x = 2 F_{\infty}^2 \Delta x. 
\end{align*}

To prove \eqref{eq:vague_conv_initially_F}, it suffices to show that $\mu_{\Delta x} \rightarrow \mu$ vaguely as $\Delta x \to 0$, see \cite[Prop. 7.19]{RealAnalysisFolland}. That is, 
\begin{align}\label{est:vague}
    \lim_{\Delta x \rightarrow 0} \int_{\R} \phi(x)d\mu_{\Delta x} &= \int_{\R}\phi(x)d\mu, 
\end{align}
for all $\phi \in C_0(\R)$, where $C_0(\R)$ denotes the space of continuous functions vanishing at infinity. Furthermore, since $C_c^{\infty}(\R)$ is dense in $C_0(\R)$, cf. \cite[Prop 8.17]{RealAnalysisFolland}, it suffices to show that \eqref{est:vague} holds for all $\phi \in C_c^{\infty}(\R)$. 

Let $\phi \in C_c^{\infty}(\R)$. Combining integration by parts and \eqref{eq:projection_F_Lp} with $p=1$, we find 
\begin{align*}
    \left| \int_{\R} \phi(x)d\mu_{\Delta x} - \int_{\R} \phi(x)d\mu \right| &= \left| \int_{\R}\phi'(x)F_{\Delta x}dx - \int_{\R}\phi'(x)Fdx \right|\\ &\leq \|\phi ' \|_{\infty} \|F_{\Delta x} - F \|_1  \\ & \leq 2F_{\infty}\|\phi ' \|_{\infty}\Delta x,
\end{align*}
and consequently $\mu_{\Delta x} \to \mu$ vaguely as $\Delta x\to 0$.

Since $F=G$ and $F_{\Delta x}=G_{\Delta x}$ by assumption, \eqref{eq:projection_G_Lp} and \eqref{eq:vague_conv_initially_G} hold.
\end{proof}

Next we establish convergence of the spatial derivative $u_{\Delta x, x}$. 

\begin{lemma}\label{lem:convergence_ux}
    Let $(u,F,G)\in \D_0$ and $(u_{\Delta x}, F_{\Delta x}, G_{\Delta x})=P_{\Delta x} \left((u,F, G) \right)$, then  
    \begin{align*}
        \lim_{\Delta x \rightarrow 0} \|u_{x} - u_{\Delta x, x} \|_{2} &= 0.
    \end{align*}
 \end{lemma}
 
  \begin{proof} We apply the Radon--Riesz theorem. Thus, we have to show that $\|u_{\Delta x, x} \|_2 \rightarrow \|u_x \|_2$ and $u_{\Delta x, x} \rightharpoonup u_x$ in $L^2(\R)$. A direct calculation, using \eqref{eq:Projected_u}, yields 
 \begin{align*}
     \|u_{\Delta x, x} \|_{2}^2 &= \sum_{j\in \mathbb{Z}} \bigg(\int_{x_{2j}}^{x_{2j+1}} \left(Du_{2j} \mp q_{2j} \right)^2dy  + \int_{x_{2j+1}}^{x_{2j+2}} \left( Du_{2j} \pm q_{2j} \right)^2dy \bigg) \nonumber \\ 
     &= \sum_{j \in \mathbb{Z}}\int_{x_{2j}}^{x_{2j+2}}DF_{\mathrm{ac}, 2j}dy = \sum_{j\in \mathbb{Z}} \int_{x_{2j}}^{x_{2j+2}}u_{x}^2(y)dy = \|u_{ x} \|_{2}^2.
 \end{align*}
For the weak convergence, it suffices to consider test functions $\phi\in C_c^\infty(\mathbb{R})$, as $C_c^{\infty}(\R)$ is dense in $L^2(\R)$. Let $\phi \in C_c^{\infty}(\R)$. Integration by parts combined with the Cauchy--Schwarz inequality and \eqref{eq:projection_u_L2} yields
 \begin{align*}
     \left |\int_{\R} \phi(y) \big(u_{x}(y) - u_{\Delta x, x}(y) \big)dy \right|&= \left | \int_{\R} \phi'(y)\left(u(y) - u_{\Delta x}(y) \right)  dy\right| \\
     &\leq \|\phi' \|_{2} \|u - u_{\Delta x}\|_{2}  \\
     &\leq \sqrt{2}\left(1+ \sqrt{2}\right) \sqrt{F_{\mathrm{ac}, \infty}} \|\phi' \|_{2}\Delta x.
 \end{align*}
 Thus $u_{\Delta x, x} \rightharpoonup u_{x}$ in $L^2(\R)$ as $\Delta x\to 0$.  
\end{proof}

The auxiliary function $f(Z)$ will be an essential part in the upcoming convergence analysis, when we want to relate convergence in Eulerian coordinates to convergence in Lagrangian coordinates. It is defined as
\begin{align}
        f(Z)(x) &= \begin{cases}
        (1 - \alpha)u_x^2(x), & x \in \tilde \Omega_{d}(Z), \\
        u_x^2(x), & x \in \tilde \Omega_{c}(Z), \end{cases} 
        \label{eq:stability_function_eulerian}
\end{align}
where \vspace{-0.1cm}
\begin{align*}
\tilde \Omega_{d}(Z) & = \{x \in \R: u_x(x) < 0 \},\\ 
\tilde \Omega_{c}(Z) &= \{x \in \R: u_x(x) \geq 0 \},
\end{align*}
 and \vspace{-0.1cm}
 \begin{equation*}
        Z := (\id, u, 1, u_{x}, \mu_{\mathrm{ac}}, \nu_{\mathrm{ac}}).
\end{equation*} 
Recalling that $d\mu_{\mathrm{ac}}=u_x^2dx$ by Definition~\ref{def:EulerianSet}~\ref{def:EulerianSetlist:condition3}, and $\mu_{\mathrm{ac}}=\nu_{\mathrm{ac}}$ for $(u, \mu, \nu) \in \D_0$, we establish the following $L^1(\R)$-estimate. 

\begin{lemma}\label{lem:convergence_g} 
Let $(u,F,G)\in \D_0$ and $(u_{\Delta x}, F_{\Delta x}, G_{\Delta x})=P_{\Delta x} \left((u,F, G) \right)$. Then $f(Z)$, $f(Z_{\Delta x}) \in L^1(\mathbb{R})$, and 
   \begin{align}
        \|f(Z) - f(Z_{\Delta x}) \|_{1} &\leq 8\sqrt{F_{\mathrm{ac}, \infty}} \|u_{x} - u_{\Delta x, x} \|_{2}. \label{eq:estimate_stability_g}
    \end{align} 
\end{lemma}

\begin{proof}
We only show that $f(Z)\in L^1(\R)$, since the argument for $f(Z_{\Delta x})$ is exactly the same. By construction $f(Z)(x)\leq u_x^2(x)$ for all $x\in \R$ and, since $u_x\in L^2(\R)$, it follows that 
\begin{equation}\label{est:f1}
\|f(Z)\|_{1}\leq\|u_x\|_{2}^2 \leq F_{\mathrm{ac},\infty},
\end{equation}
and hence $f(Z)\in L^1(\R)$. 

For \eqref{eq:estimate_stability_g}, we use a splitting based on $\tilde\Omega_d(Z)$ and $\tilde\Omega_c(X)$.
If $x \in \tilde \Omega_{d}(Z)\cap \tilde \Omega_{d}(Z_{\Delta x})$, then 
    \begin{align*}
        \left |f(Z)(x)  - f(Z_{\Delta x})(x) \right | &= \left(1- \alpha \right) \left|(u_{x} + u_{\Delta x, x})(u_{x} - u_{\Delta x, x})(x) \right|, 
    \end{align*}
    and, by the Cauchy--Schwarz inequality and \eqref{est:f1}, we obtain
    \begin{align}
        \int_{\tilde \Omega_{d}(Z)\cap \tilde \Omega_{d}(Z_{\Delta x})} \left|f (Z)(x)  - f (Z_{\Delta x})(x)  \right|dx &\leq (1 - \alpha) \|u_{x} + u_{\Delta x, x}\|_2 \|u_{x} - u_{\Delta x, x}\|_2 \nonumber \\
        &\leq 2 \sqrt{F_{\mathrm{ac}, \infty}} \|u_{ x}-u_{\Delta x,  x}\|_2. \label{eq:same_cases}
    \end{align}
    We can proceed similarly for $x \in \tilde \Omega_c(Z) \cap \tilde \Omega_c(Z_{\Delta x})$. 
    
    Finally, consider $x \in \left(\tilde \Omega_c(Z)\cap \tilde \Omega_d(Z_{\Delta x})\right) \cup \left(\tilde \Omega_d(Z)\cap \tilde\Omega_c(Z_{\Delta x})\right)$. By symmetry, it is sufficient to consider $x \in \tilde\Omega_c(Z)\cap \tilde \Omega_d(Z_{\Delta x})$, for which 
    \begin{align*}
        \left |f (Z)(x) - f (Z_{\Delta x})(x) \right| &\leq (1 - \alpha) \left|(u_{x}^2 - u_{\Delta x, x}^2)(x) \right| + \alpha u_{x}^2(x) \\
        &\leq (1 - \alpha)\left|(u_{x}^2 - u_{\Delta x,x}^2)(x)\right| + \alpha u_{x}(x) \left(u_{x}(x) - u_{\Delta x, x}(x) \right), 
    \end{align*}
    since $u_{x}(x) \geq 0$ and $u_{\Delta x, x}(x) <0$. Therefore, by the Cauchy--Schwarz inequality,
    \begin{equation}
        \int_{\tilde \Omega_c(Z) \cap \tilde \Omega_d(Z_{\Delta x})} |f(Z)(x) - f(Z_{\Delta x})(x)|dx \leq 2 \sqrt{F_{\mathrm{ac}, \infty}}\|u_{1, x} - u_{2, x} \|_{2}. \label{eq:mixed_cases}
    \end{equation}
    Combining \eqref{eq:same_cases} and \eqref{eq:mixed_cases} with analogous estimates for the other cases yields \eqref{eq:estimate_stability_g}.
\end{proof}

\subsection{Convergence in Lagrangian coordinates}
 We start by showing convergence of the initial data in Lagrangian coordinates. 

\begin{lemma}\label{lem:initial_conv_lagrangian}
    Given $(u, \mu, \nu) \in \D_0$, let $X = \left(y, U, V, H \right) = L \left( (u, \mu, \nu) \right)$ and $X_{\Delta x}= (y_{\Delta x}, U_{\Delta x}, V_{\Delta x}, H_{\Delta x}) = L \circ P_{\Delta x} \left( (u, \mu, \nu) \right)$, then 
    \begin{subequations}
    \begin{align}
        \|y - y_{\Delta x} \|_{\infty} &\leq 2\Delta x, \label{eq:initial_approx_y}\\
        \|U - U_{\Delta x} \|_{\infty} & \leq \left(1 + 2\sqrt{2} \right)\sqrt{F_{\mathrm{ac}, \infty}}\Delta x^{\frac{1}{2}},\label{eq:initial_approx_U} \\
        \|H - H_{\Delta x} \|_{\infty} & \leq 2 \Delta x. \label{eq:initial_approx_H}
    \end{align}
    \end{subequations}
\end{lemma}

\begin{proof}
Given $\xi \in \R$, there exists $j \in \mathbb{Z}$ such that $\xi \in [\hat{\xi}_{3j}, \hat{\xi}_{3j+3})$. Moreover, by construction, $y$ and $y_{\Delta x}$ are continuous, increasing and satisfy 
\begin{equation}
 y(\hat{\xi}_{3j})=x_{2j}=y_{\Delta x} (\hat{\xi}_{3j}) \quad \text{ for all }j \in \mathbb{Z},
 \label{eq:char_coincide_even}
\end{equation}
as discussed in Section~\ref{subsec:Implementation}. Thus   
\begin{equation*}
        \left | y(\xi) - y_{\Delta x}(\xi) \right|  \leq \left | y( \hat{\xi}_{3j+3}) - y(\hat{\xi}_{3j}) \right| = x_{2j+2} - x_{2j} = 2 \Delta x 
    \end{equation*}
    and, as $X$ and $X_{\Delta x}$ belong to $\F_0$, we have shown both \eqref{eq:initial_approx_y} and \eqref{eq:initial_approx_H}.
        
        To prove \eqref{eq:initial_approx_U}, note that by \eqref{eq:L_eq1}, to any $x \in \mathbb{R}$ there exist characteristic variables $\xi_{\Delta x}$ and $\xi$, such that $y_{\Delta x}( \xi_{\Delta x}) = x=y(\xi)$, and hence 
    \begin{align}
        u(x) - u_{\Delta x}(x) &= u(y(\xi)) - u_{\Delta x}(y_{\Delta x}( \xi_{\Delta x})) \nonumber \\ &= u(y(\xi)) - u_{\Delta x}( y_{\Delta x}(\xi)) + u_{\Delta x}( y_{\Delta x}(\xi)) - u_{\Delta x}( y_{\Delta x}(\xi_{\Delta x})).
    \label{eq:proving_U_initial}
    \end{align}
    By \eqref{eq:L_eq2}, we have, $u(y(\xi)) = U(\xi)$ and $u_{\Delta x}(y_{\Delta x}(\xi)) = U_{\Delta x}(\xi)$. Therefore, by rearranging \eqref{eq:proving_U_initial}, we obtain 
    \begin{align*}
        \left |U(\xi) - U_{\Delta x}(\xi) \right| & \leq \left |u(x) - u_{\Delta x}(x) \right| + \left |u_{\Delta x}( y_{\Delta x}(\xi_{\Delta x})) - u_{\Delta x}(y_{\Delta x}(\xi)) \right|.
    \end{align*}
    The first term is bounded by \eqref{eq:projection_u_infty}, while for the second term we apply the Cauchy--Schwarz inequality, yielding
    \begin{align*}
        \left |u_{\Delta x} \left(y_{\Delta x}(\xi_{\Delta x}) \right) - u_{\Delta x} \left(y_{\Delta x}(\xi) \right) \right| &= \left| \int_{y_{\Delta x}( \xi)}^{y_{\Delta x}(\xi_{\Delta x})}u_{\Delta x, x}(z)dz \right| \\ &\leq \|u_{\Delta x, x}\|_2 \sqrt{ \left |y_{\Delta x}(\xi_{\Delta x}) - y_{\Delta x}(\xi) \right| }  \\ & \leq \sqrt{F_{\mathrm{ac}, \infty}} \sqrt{ \left |y(\xi) - y_{\Delta x}(\xi) \right| } \\ &\leq \sqrt{F_{\mathrm{ac}, \infty}} \sqrt{\|y - y_{\Delta x} \|_{\infty}}.
    \end{align*}
    Thus,
    \begin{align*}
        \|U - U_{\Delta x}\|_{\infty} & \leq \|u - u_{\Delta x} \|_{\infty} + \sqrt{F_{\mathrm{ac}, \infty}} \sqrt{\|y- y_{\Delta x} \|_{\infty}} \\ &\leq \left(1 + \sqrt{2} \right)\sqrt{ F_{\mathrm{ac},\infty}}\Delta x^{\frac{1}{2}} + \sqrt{2F_{\text{ac}, \infty}} \Delta x^{\frac{1}{2}} \\
        &= \left(1 + 2\sqrt{2} \right)\sqrt{F_{\mathrm{ac}, \infty}}\Delta x^{\frac{1}{2}}. \qedhere
    \end{align*}
\end{proof}

For the initial derivatives we have the following convergence result.

\begin{lemma}\label{lem:Convergences_derivatives_without_rates}
        Given $(u, \mu, \nu) \in \D_0$, let $X= (y, U, V, H) = L \left((u, \mu, \nu) \right)$ and $X_{\Delta x} = \left(y_{\Delta x}, U_{\Delta x}, V_{\Delta x}, H_{\Delta x} \right) = L \circ P_{\Delta x} \left((u, \mu, \nu ) \right)$, then as $\Dx \rightarrow 0$, 
        \begin{subequations}
        \begin{align}
            y_{\Delta x, \xi} & \rightarrow y_{\xi} \text{ in } L^1(\R) \cap L^2(\R),  \label{eq:convergence_y_xi} \\
             U_{\Delta x, \xi} &\rightarrow U_{\xi} \text{ in } L^2(\R), \label{eq:convergence_U_xi} \\ 
            H_{\Delta x, \xi} &\rightarrow H_{\xi} \text{ in } L^1(\R) \cap L^2(\R). \label{eq:convergence_H_xi}
        \end{align}
        \end{subequations}
\end{lemma}

\begin{proof} For a proof of $H_{\Delta x, \xi} \rightarrow H_{\xi}$ and $U_{\Delta x, \xi} \rightarrow U_{\xi}$ in $L^2(\R)$ we refer to \cite[Sec. 5]{EquivalenceEulerLagrange}. 
As $X\in \F_0^0$ and, in particular, $U_{\xi}^2 = y_{\xi}H_{\xi}$, we have
\begin{equation}
    H_{\xi}^2 = H_{\xi} - U_{\xi}^2.
\label{eq:useful_relations}
\end{equation}
Since \eqref{eq:useful_relations} also holds for $X_{\Delta x}$, we obtain, using the Cauchy--Schwarz inequality, 
\begin{align*}
    \|H_{\Delta x, \xi} - H_{\xi}\|_1 & \leq \|H_{\Delta x, \xi}^2 - H_{\xi}^2 \|_1 + \|U_{\Delta x, \xi}^2 - U_{\xi}^2 \|_1 \\ & \leq \left( \|H_{\Delta x, \xi}\|_2 + \|H_{\xi}\|_2 \right) \|H_{\Delta x, \xi} - H_{\xi} \|_2 \\ & \qquad + \left(\|U_{\Delta x, \xi}\|_2 + \|U_{\xi}\|_2 \right)\|U_{\Delta x, \xi} - U_{\xi}\|_2,
\end{align*}
and $H_{\Delta x, \xi} \to H_\xi$ in $L^1(\R)$. Lastly, \eqref{eq:convergence_y_xi} follows from \eqref{eq:convergence_H_xi}, as $X_{\Delta x}, X \in \F_0^0$. 
\end{proof}

We proceed by introducing the stability function $g(X)$, which is the key element when showing convergence at later times in Lagrangian coordinates. In particular, it describes the loss of energy at wave breaking in a continuous way, in contrast to the actual energy density $V_{\xi}$, which drops abruptly at wave breaking.

Let 
\begin{equation*}
\R=\Omega_d(X) \cup \Omega_c(X),
\end{equation*}
 where 
\begin{align}\label{def.omegas}
    \Omega_d \left(X \right) :&= \left \{ \xi \in \R: U_{\xi}(\xi) < 0 \right\}, \notag 
    \\
    \Omega_c(X) :&= \{\xi \in \R: U_{\xi}(\xi) \geq 0 \}. 
\end{align}

Note that for each $\xi\in\Omega_c$ no wave breaking occurs in the future and hence $V_\xi(t,\xi)$ is continuous forward in time, while for $\xi \in \Omega_d$ wave breaking occurs in the future and hence $V_\xi(t,\xi)$ might be discontinuous forward in time.

The function $g(X)$ is then defined as  
\begin{align}
    g(X)(\xi) &:= \begin{cases}
   (1 - \alpha)V_{\xi}(\xi), & \xi \in \Omega_d(X), \\
     V_{\xi}(\xi), & \xi \in \Omega_c(X). \end{cases}
     \label{eq:stability_function_g}
\end{align}

Note that the Eulerian counterpart of $g(X)$ is the previously defined function $f(Z)$, cf. \eqref{eq:stability_function_eulerian}. The relation between the two functions is clarified in the following remark.

\begin{remark}
If $(u,\mu,\nu)\in \D_0$, $\nu$ is purely absolutely continuous, and $X=L((u,\mu,\nu))$, the functions \eqref{eq:stability_function_g} and \eqref{eq:stability_function_eulerian} are related through 
    \begin{align*}
        g(X)(\xi) = f(Z) (y(\xi))y_{\xi}(\xi), \text{ for a.e. } \xi \in \R.
    \end{align*}
\end{remark}

The following result is the final one concerning convergence of the initial data in Lagrangian coordinates.

\begin{prop}\label{prop:convergence_g_initially} 
Given $(u, \mu, \nu) \in \D_0$, let $X= (y, U, V, H) = L \left((u, \mu, \nu) \right)$ and $X_{\Delta x} = \left(y_{\Delta x}, U_{\Delta x}, V_{\Delta x}, H_{\Delta x} \right) = L \circ P_{\Delta x} \left((u, \mu, \nu ) \right)$, then as $\Dx \rightarrow 0$, 
\begin{align}
    g(X_{\Delta x}) &\rightarrow g(X) \text{ in } L^1(\R) \cap L^2(\R). \label{eq:convergence_g}
\end{align}
\end{prop}

\begin{proof}
To begin with observe that for any $X\in \F_0^0$
\begin{equation*}
\|g(X)\|_2^2\leq \|V_\xi\|_2^2\leq \|V_\xi\|_1\leq F_\infty, 
\end{equation*}
since $0\leq V_\xi(\xi)=H_\xi(\xi)\leq 1$ for all $\xi\in \R$, which follows from \eqref{eq:L_eq3} and the fact that both $y$ and $H$ are increasing functions. We can therefore apply the Radon--Riesz theorem to establish $g(X_{\Delta x})\to g(X)$ in $L^2(\R)$. Accordingly, we split the proof of \eqref{eq:convergence_g} into three parts, \emph{i)} $L^2$-norm convergence, \emph{ii)} weak $L^2$-convergence, and \emph{iii)} $L^1$-convergence. \medskip 

\noindent \emph{i)} Verification of $\|g(X_{\Delta x}) \|_2 \rightarrow \|g(X) \|_{2}$. Note that since $X_{\Delta x}$ and $X$  belong to $\F_0^0$, we have
\begin{subequations}\label{eq:norm_estimate}
\begin{align}\nonumber
    \bigl | \|g(X) \|_{2}^2 &- \|g(X_{\Delta x})\|_2^2 \bigr | \\ &= \bigg|\int_{ \Omega_c(X)} H^2_{\xi}(\xi)d\xi + \int_{\Omega_d(X)} \left(1-\alpha \right)^2 H_{\xi}^2(\xi)d\xi  \nonumber \\ & \qquad - \int_{ \Omega_c(X_{\Delta x})}H_{\Delta x, \xi}^2(\xi)d\xi  - \int_{ \Omega_d(X_{\Delta x})} \left(1-\alpha \right)^2 H_{\Delta x, \xi}^2(\xi)d\xi \bigg| \nonumber \\ 
    & \leq \left| \int_{\R} H_{\xi}^2(\xi)d\xi - \int_{\R} H_{\Delta x, \xi}^2(\xi)d\xi \right| \nonumber \\ 
    & \qquad + \alpha(2-\alpha) \left| \int_{ \Omega_d(X)}H_{\xi}^2(\xi)d\xi - \int_{ \Omega_d(X_{\Delta x})}H_{\Delta x, \xi}^2(\xi)d\xi \right| \nonumber \\
    & \leq \left | \|H_{\xi} \|_2^2 - \|H_{\Delta x, \xi} \|_2^2 \right| \nonumber\\ 
    & \qquad + \alpha(2-\alpha) \left| \int_{ \Omega_d(X)}H_{\xi}(\xi)d\xi - \int_{ \Omega_d(X_{\Delta x})}H_{\Delta x, \xi}(\xi)d\xi \right| \label{eq:norm_estimate_2}\\ 
    & \qquad + \alpha(2-\alpha) \left| \int_{\Omega_d(X)}U_{\xi}^2(\xi)d\xi - \int_{\Omega_d(X_{\Delta x})}U_{\Delta x, \xi}^2(\xi)d\xi \right|, \label{eq:norm_estimate_3}
 \end{align}
 \end{subequations}
where we used \eqref{eq:useful_relations} in the last step.

To estimate \eqref{eq:norm_estimate_2} and \eqref{eq:norm_estimate_3}, introduce
\begin{align}
        \mathcal{S} &:= \{ \xi \in \R: y_{\xi}(\xi) =0 \}=\{\xi\in \R : V_\xi(\xi)=1\},\label{eq:wave_breaking_set_temp} 
    \end{align}
    and $B=y(\mathcal{S})$. Then it has been shown in the proof of \cite[Thm. 27]{alphaCH} that 
     \begin{equation}\label{relation:acs}
    \mu_{\mathrm{ac}}=\mu\vert_{B^c} \quad \text{ and } \quad \mu_{\mathrm{sing}}=\mu\vert_B,
    \end{equation}
    where $\mu\vert_B$ denotes the restriction of $\mu$ to $B$, that is $\mu\vert_B(E)=\mu(E\cap B)$ for any Borel set $E$. Thus
    \begin{equation*}
    \mathrm{meas}(\mathcal{S})= \mu_{\mathrm{sing}}(\R)=F_{\mathrm{sing},\infty}.
    \end{equation*}
    Along the same lines, by defining
     \begin{align}
         \mathcal{S}_{\Delta x} :&=\{\xi\in \R:y_{\Delta x, \xi}(\xi)=0\}= \{\xi \in \R: V_{\Delta x, \xi}(\xi) = 1 \} = \bigcup_{j \in \mathbb{Z}} \left[\hat\xi_{3j}, \hat\xi_{3j+1}\right], \label{eq:numerical_wave_breaking_set}
    \end{align}
    we have
    \begin{equation}\label{size:sd}
    \mathrm{meas}(\mathcal{S}_{\Delta x})= \mu_{\Delta x,\mathrm{sing}}(\R)=F_{\Delta x, \mathrm{sing}, \infty}=F_{\mathrm{sing},\infty}.
    \end{equation}
 Since $\Omega_d(X)=\Omega_d(X)\cap \mathcal{S}^c$ and $\Omega_d(X_{\Delta x})=\Omega_d(X_{\Delta x})\cap \mathcal{S}^c_{\Delta x}$, due to Definition~\ref{def:DefinitionLagrangian}~\ref{def:DefinitionLagrangian:ImporantRel}, \eqref{relation:acs} implies that 
 \begin{align}\label{hilfsabsch}
    \bigg| \int_{\Omega_d(X)}&H_{\xi}(\xi)d\xi - \int_{\Omega_d(X_{\Delta x})}H_{\Delta x, \xi}(\xi)d\xi \bigg| \nonumber \\ 
    &\hspace{-0.18cm} = \left|\int_{y\left(\Omega_d(X) \right)}u_x^2(z)dz - \int_{ y_{\Delta x}\left(\Omega_d(X_{\Delta x}) \right)}u_{\Delta x, x}^2(z)dz\right|. 
\end{align}
Furthermore, $U_{\xi}(\xi) = u_x (y(\xi))y_{\xi}(\xi)$ for almost every $\xi \in  \Omega_d(X)$ and likewise for $U_{\Delta x,\xi}$. Hence,  we can replace in the above integrals $y\left(\Omega_d(X) \right)$ and $y_{\Delta x}\left(\Omega_d(X_{\Delta x}) \right)$, by $\tilde \Omega_d(Z)$ and $\tilde \Omega_{d}(Z_{\Delta x})$, respectively, and end up with   
    \begin{align*}
     \bigg| \int_{ \Omega_d(X)}H_{\xi}(\xi)d\xi &- \int_{ \Omega_d(X_{\Delta x})}H_{\Delta x, \xi}(\xi)d\xi \bigg| \nonumber
     \\ 
     & \leq \int_{\tilde \Omega_d(Z) \cap \tilde \Omega_{c}(Z_{\Delta x})} u_x^2(z)dz + \int_{\tilde\Omega_c(Z) \cap \tilde \Omega_{d}(Z_{\Delta x})}u_{\Delta x, x}^2(z)dz  \\ 
     & \qquad + \int_{\tilde \Omega_d(Z) \cap \tilde \Omega_{d}(Z_{\Delta x})} \left| u_x^2(z) - u_{\Delta x, x}^2(z) \right|dz. 
     \end{align*}
The terms on the right hand side can be estimated using the same approach as in the proof of Lemma~\ref{lem:convergence_g}, which yields 
 \begin{align}
     \left| \int_{ \Omega_d(X)}H_{\xi}(\xi)d\xi - \int_{ \Omega_d(X_{\Delta x})}H_{\Delta x, \xi}(\xi)d\xi \right| \leq 4 \sqrt{F_{\mathrm{ac}, \infty}}\|u_x - u_{\Delta x, x} \|_2.
     \label{eq:estimate_H_xi_negative}
\end{align}  

For \eqref{eq:norm_estimate_3}, we get
\begin{subequations}
\begin{align}
    \bigg| \int_{\Omega_d(X)}U_{\xi}^2(\xi)d\xi &- \int_{\Omega_d(X_{\Delta x})}U_{\Delta x, \xi}^2(\xi)d\xi \bigg| \nonumber \\ &\leq \int_{\Omega_d(X) \cap \Omega_c(X_{\Delta x})} U_{\xi}^2(\xi)d\xi + \int_{\Omega_c(X) \cap \Omega_d(X_{\Delta x})} U_{\Delta x, \xi}^2(\xi)d\xi \label{eq:norm_estimate_3_2}\\ &\quad + \int_{\Omega_d(X) \cap \Omega_{d}(X_{\Delta x})} \left| U_{\xi}^2(\xi) - U_{\Delta x, \xi}^2(\xi) \right|d\xi. \label{eq:norm_estimate_3_3}
\end{align}
\end{subequations}
The two terms in \eqref{eq:norm_estimate_3_2} have a similar structure and we therefore only consider the first one. We have 
\begin{align*}
    \int_{\Omega_d(X) \cap \Omega_c(X_{\Delta x})} U_{\xi}^2(\xi)d\xi &\leq \int_{\Omega_d(X) \cap \Omega_c(X_{\Delta x})} U_{\xi}(\xi) \left(U_{\xi}(\xi) - U_{\Delta x, \xi}(\xi) \right)d\xi \\
    &\leq \left(\int_{\R} y_{\xi}V_{\xi}(\xi)d\xi \right)^{\frac{1}{2}} \|U_{\xi} - U_{\Delta x, \xi} \|_2 \\ 
    & \leq \left(\int_{\R} V_{\xi}(\xi)d\xi \right)^{\frac{1}{2}} \|U_{\xi} - U_{\Delta x, \xi} \|_2\\
    & = \sqrt{F_{\infty}}\|U_{\xi} - U_{\Delta x, \xi} \|_2, 
\end{align*}
since $U_{\xi}^2 = y_{\xi}V_{\xi}$ and $0 \leq y_{\xi} \leq 1$.
Estimating \eqref{eq:norm_estimate_3_3} in much the same way, yields 
\begin{equation}\label{eq:estimate_U_xi_negative}
 \bigg| \int_{\Omega_d(X)}U_{\xi}^2(\xi)d\xi - \int_{\Omega_d(X_{\Delta x})}U_{\Delta x, \xi}^2(\xi)d\xi \bigg| \leq 4\sqrt{F_\infty}\|U_{\xi} - U_{\Delta x, \xi} \|_2.
 \end{equation}
 
Finally, combining \eqref{eq:norm_estimate}, \eqref{eq:estimate_H_xi_negative} and \eqref{eq:estimate_U_xi_negative}, we have shown that 
\begin{align*}
    \left| \|g(X) \|_{2}^2 - \|g(X_{\Delta x})\|_2^2 \right| &\leq \left| \|H_{\xi} \|_2^2 - \|H_{\Delta x, \xi} \|_2^2 \right| + 4\alpha (2-\alpha)\sqrt{F_{\mathrm{ac}, \infty}} \|u_x - u_{\Delta x, x} \|_2 \\ & \quad + 4\alpha (2-\alpha)\sqrt{F_{\infty}} \|U_{\xi} - U_{\Delta x, \xi} \|_2. 
\end{align*}
As $H_{\Delta x, \xi} \rightarrow H_{\xi}$ and $U_{\Delta x, \xi} \rightarrow U_{\xi}$ in $L^2(\R)$ by Lemma \ref{lem:Convergences_derivatives_without_rates} and $u_{\Delta x, x} \rightarrow u_{x}$ in $L^2(\R)$ by Lemma \ref{lem:convergence_ux}, it holds that $\|g(X_{\Delta x})\|_2 \rightarrow \|g(X) \|_2$. \medskip
   
\noindent \emph{ii)} We show that $g(X_{\Delta x}) \rightharpoonup g(X)$ in $L^2(\R)$. To that end, we interpret $g(X_{\Delta x})$ and $g(X)$ as positive Radon measures with the associated functions
\begin{align*}
    \hat{G}(\xi) = \int_{-\infty}^{\xi} g(X)(\eta)d\eta \quad \text{and} \quad
    \hat{G}_{\Delta x}(\xi) = \int_{-\infty}^{\xi}g(X_{\Delta x})(\eta)d\eta. 
\end{align*}
If we show that $\hat{G}_{\Delta x} \rightarrow \hat{G}$ pointwise, then \cite[Prop. 7.19]{RealAnalysisFolland} implies that $g(X_{\Delta x}) \rightarrow g(X)$ vaguely, and as a consequence, $g(X_{\Delta x}) \rightharpoonup g(X)$ in $L^2(\R)$. 

Let $\xi \in [\hat{\xi}_{3j}, \hat{\xi}_{3j+3})$, and note that due to Definition~\ref{def:ProjectionOperator}, Definition~\ref{def:MappingL}, \eqref{eq:jump_point}--\eqref{eq:LagrangianGridpoints}, and \eqref{relation:acs},
\begin{align*}
\mu_{\mathrm{sing}}((-\infty,x_{2j}))&=\int_{(-\infty, \hat\xi_{3j}]\cap \mathcal{S}}V_\xi(\eta) d\eta = \int_{(-\infty, \hat\xi_{3j}]\cap \mathcal{S}_{\Delta x}} V_{\Delta x, \xi}(\eta) d\eta\\ & =\mu_{\Delta x,\mathrm{sing}}((-\infty, x_{2j})).
\end{align*}
Thus, 
\begin{align*}
    \bigl| \hat{G}(\xi) - \hat{G}_{\Delta x}(\xi) \bigr|& \leq\left| \int_{(-\infty, \hat\xi_{3j}]\cap \mathcal{S}^c} g(X) (\eta) d\eta- \int_{(-\infty, \hat \xi_{3j}]\cap \mathcal{S}_{\Delta x}^c}g(X_{\Delta x}) (\eta) d\eta\right|\\
    & \quad + \left| \int_{\hat \xi_{3j}}^\xi (g(X)-g(X_{\Delta x}))(\eta) d\eta\right|\\
    & = I_1+I_2.
    \end{align*} 
  For $I_1$ note that after using the same argument as the one leading to \eqref{hilfsabsch}, we obtain  
  \begin{equation*}
   \int_{(-\infty, \hat\xi_{3j}]\cap \mathcal{S}^c} g(X) (\eta) d\eta- \int_{(-\infty, \hat \xi_{3j}]\cap \mathcal{S}_{\Delta x}^c}g(X_{\Delta x}) (\eta) d\eta
   = \int_{-\infty}^{x_{2j}} (f(Z)-f(Z_{\Delta x})) (x) dx,
   \end{equation*}
   and due to Lemma~\ref{lem:convergence_g},
   \begin{equation*}
   \vert I_1\vert \leq \| f(Z)-f(Z_{\Delta x})\|_1\leq 8\sqrt{F_{\mathrm{ac}, \infty}} \|u_x-u_{\Delta x,x}\|_2.
   \end{equation*}
The term $I_2$ can be estimated as follows,  
\begin{align*}
    |I_2| &\leq \left| \int_{(\hat{\xi}_{3j}, \xi]}H_{\xi}(\eta)d\eta - \int_{(\hat{\xi}_{3j}, \xi]}H_{\Delta x, \xi}(\eta)d\eta \right| \\ 
    & \qquad + \alpha \left| \int_{(\hat{\xi}_{3j}, \xi] \cap \Omega_d(X)}H_{\xi}(\eta)d\eta - \int_{(\hat{\xi}_{3j}, \xi] \cap \Omega_d(X_{\Delta x})}H_{\Delta x, \xi}(\eta)d\eta \right|.
\end{align*}
The first term is bounded from above by $\|H - H_{\Delta x} \|_{\infty} \leq 2\Delta x$, due to \eqref{eq:initial_approx_H}. For the second term, let $x= y(\xi)$ and $x_{\Delta x}=y_{\Delta x}(\xi)$, which both belong to $[x_{2j}, x_{2j+2}]$. Using once more the same argument as the one leading to \eqref{hilfsabsch}, $\nu_{\textrm{ac}}=u_x^2$, and $\nu_{\Delta x, \textrm{ac}}=u_{\Delta x, x}^2$, we end up with  
\begin{align*}
      \alpha \biggl| \int_{(\hat{\xi}_{3j}, \xi] \cap \Omega_d(X)}H_{\xi}(\eta)d\eta &- \int_{(\hat{\xi}_{3j},\xi] \cap \Omega_d(X_{\Delta x})}H_{\Delta x, \xi}(\eta)d\eta \biggr| \nonumber \\ 
      & \leq \alpha \left( \int_{x_{2j}}^xu_x^2(z)dz + \int_{x_{2j}}^{x_{\Delta x}}u_x^2(z)dz \right)\\
      & \leq 2(F_{\textrm{ac}}(x_{2j+2})- F_{\textrm{ac}}(x_{2}))\\
      & \leq 2(F_{\textrm{ac}}(x+2\Delta x)-F_{\textrm{ac}}(x-2\Delta x)), 
\end{align*}
and therefore 
\begin{equation*}
|I_2|\leq 2\Delta x+ 2(F_{\textrm{ac}}(x+2\Delta x)-F_{\textrm{ac}}(x-2\Delta x)).
\end{equation*}
Thus, 
\begin{align}\nonumber
    |\hat{G}(\xi) - \hat{G}_{\Delta x}(\xi)|& \leq 2 \Delta x + 8 \sqrt{F_{\text{ac}, \infty}}\|u_x - u_{\Delta x, x} \|_2\\
    & \quad  + 2(F_{\textrm{ac}}(x+2\Delta x)-F_{\textrm{ac}}(x-2\Delta x)). 
    \label{eq:estimate_hat_G}
\end{align}
By Lemma \ref{lem:convergence_ux}, $u_{\Delta x, x} \rightarrow u_x$ in $L^2(\R)$, and, since $F_{\textrm{ac}}$ is continuous, also the third term tends to zero as $\Delta x \rightarrow 0$. Therefore, \eqref{eq:estimate_hat_G} implies that $\hat{G}_{\Delta x} \rightarrow \hat{G}$ pointwise, which gives, see \cite[Prop. 7.19]{RealAnalysisFolland}, that 
\begin{equation}\label{vaguetol2}
\lim_{\Delta x\to 0} \int_{\R} g(X_{\Delta x})\phi(\eta) d\eta = \int_{\R} g(X) \phi(\eta) d\eta, \quad \text{ for all } \phi\in C_c^\infty(\R).
\end{equation}
Since $g(X)$ and $g(X_{\Delta x})$ belong to $L^2(\R)$ and $C_c^\infty(\R)$ is a dense subset of $L^2(\R)$, \eqref{vaguetol2} holds for all $\phi\in L^2(\R)$ and $g(X_{\Delta x})\rightharpoonup g(X)$ in $L^2(\R)$.

Consequently, combining \emph{i)} and \emph{ii)}, the conditions of the Radon--Riesz theorem are met and $g(X_{\Delta x})\to g(X)$ in $L^2(\R)$. \medskip

\noindent \emph{iii)} It is left to show that $g(X_{\Delta x}) \rightarrow g(X)$ in $L^1(\R)$. Again we will use a splitting based on the sets $\mathcal{S}$ and $\mathcal{S}_{\Delta x}$ defined in \eqref{eq:wave_breaking_set_temp} and \eqref{eq:numerical_wave_breaking_set}, respectively.  

Since $g(X_{\Delta x})(\xi) = g(X) (\xi)= 1$ for a.e. $\xi\in \mathcal{S}\cap \mathcal{S}_{\Delta x}$, we have 
\begin{subequations}\label{eq:convergence_gg}
\begin{align}
    \|g(X) - g(X_{\Delta x}) \|_1 &= \int_{\mathcal{S} \cap \mathcal{S}_{\Delta x}^c} \left|g(X)(\xi)- g(X_{\Delta x})(\xi) \right|d\xi \nonumber\\ & \qquad +  \int_{\mathcal{S}^c \cap \mathcal{S}_{\Delta x}} \left|g(X)(\xi) - g(X_{\Delta x})(\xi) \right|d\xi \label{eq:convergence_g_2}
    \\ & \qquad +  \int_{\mathcal{S}^c \cap \mathcal{S}_{\Delta x}^c} \left|g(X)(\xi) - g(X_{\Delta x})(\xi) \right|d\xi. \label{eq:convergence_g_3}
\end{align}
\end{subequations}
By the Cauchy--Schwarz inequality and \eqref{size:sd}, \eqref{eq:convergence_g_2} satisfies
\begin{align}\label{eq:convergence_ggg}
    \int_{\mathcal{S}^c \cap \mathcal{S}_{\Delta x}} \left|g(X)(\xi) - g(X_{\Delta x})(\xi) \right|d\xi \nonumber & \leq  \sqrt{\mathrm{meas}\! \left( \mathcal{S}_{\Delta x} \right)}\|g(X) - g(X_{\Delta x}) \|_2 \\ &\leq \sqrt{F_{\infty}}\|g(X) - g(X_{\Delta x}) \|_2,
\end{align}
and we treat the first term in \eqref{eq:convergence_gg} the same way. 

The term \eqref{eq:convergence_g_3}, on the other hand, requires a bit more work. Introduce the set $A = \{\xi : H_{\xi}(\xi) \geq \frac{1}{2}\}$, which satisfies, due to Chebyshev's inequality with $p=1$,
\begin{equation}\label{measA}
    \mathrm{meas}(A) \leq 2 \|H_{\xi}\|_1 = 2 G_{\infty}. 
\end{equation}
Furthermore, let $E_{\Delta x} = \mathcal{S}^c \cap \mathcal{S}_{\Delta x}^c$, so that \eqref{eq:convergence_g_3} can be written as

\begin{align}
 \int_{E_{\Delta x}} \left|g(X)(\xi) - g(X_{\Delta x})(\xi) \right|d\xi& =\!\int_{E_{\Delta x}\cap A} \left|g(X)(\xi) - g(X_{\Delta x})(\xi) \right|d\xi \notag \\ 
 & \quad +\int_{E_{\Delta x}\cap A^c} \left|g(X)(\xi) - g(X_{\Delta x})(\xi) \right|d\xi \label{eq:convergence_gh}.
 \end{align}
 For \eqref{eq:convergence_gh}, the Cauchy--Schwarz inequality and \eqref{measA} imply
 \begin{equation}\label{eq:convergence_ggggg}
 \int_{E_{\Delta x}\cap A} \left|g(X)(\xi) - g(X_{\Delta x})(\xi) \right|d\xi\leq \!\sqrt{2G_\infty} \|g(X)-g(X_{\Delta x})\|_2.
 \end{equation}
 
 Combining \eqref{eq:convergence_gg}, \eqref{eq:convergence_ggg} \eqref{eq:convergence_gh}, and \eqref{eq:convergence_ggggg} and recalling that $g(X_{\Delta x})\to g(X) $ in $L^2(\R)$, it remains to show that \eqref{eq:convergence_gh} tends to zero as $\Delta x\to 0$. We follow the proof of \cite[Lem. 7.3]{GlobalDissipative2CH} closely. For almost every $\xi \in E_{\Delta x}$, we have 
\begin{equation}
    y_{\xi}(\xi) = \frac{1}{1 + u_{x}^2 (y(\xi))} \quad \text{ and }\quad  y_{\Delta x, \xi}(\xi) = \frac{1}{1 + u_{\Delta x, x}^2 (y_{\Delta x}(\xi))},
\label{eq:explicit_representatives_characteristics}
\end{equation}
and
\begin{equation}\label{splittingg}
 g(X)(\xi) - g(X_{\Delta x})(\xi) = f(Z) (y(\xi)) y_{\xi}(\xi) - f(Z_{\Delta x}) (y_{\Delta x}(\xi)) y_{\Delta x, \xi} (\xi).
 \end{equation}
 Thus combining \eqref{eq:explicit_representatives_characteristics} and \eqref{splittingg}, 
 \begin{subequations}
\begin{align}
    g(X) - g(X_{\Delta x})
        &= \left( f(Z) (y) - f(Z) (y_{\Delta x}) \right) y_{\xi} y_{\Delta x, \xi} \label{eq:first_term_g_1}\\ 
       & \qquad + \left( f(Z) (y_{\Delta x}) - f(Z_{\Delta x})( y_{\Delta x}) \right) y_{\xi} y_{\Delta x, \xi} \label{eq:first_term_g_2}\\
    &\qquad + f(Z) (y) \left(u_{\Delta x, x}^2 (y_{\Delta x}) - u_{ x}^2(y) \right)y_{\xi} y_{\Delta x, \xi} \label{eq:second_part_g}\\ 
    & \qquad + u_{x}^2 (y) \left(f(Z) (y) - f(Z_{\Delta x}) (y_{\Delta x}) \right) y_{\xi} y_{\Delta x, \xi}\label{eq:third_part_g}.
\end{align}
\end{subequations}

Concerning the integral of \eqref{eq:first_term_g_1}, note that $f(Z) \in L^1(\R)$ by Lemma \ref{lem:convergence_g}. Thus, given $\epsilon > 0$ there exists $\psi \in C_c(\R)$ such that $\|f(Z) - \psi \|_{L^1(\R)} \leq \epsilon $, since $C_c(\R)$ is dense in $L^1(\R)$, and  
\begin{align}
   \int_{E_{\Delta x}\cap A^c} \left|f(Z) (y) - f(Z) (y_{\Delta x}) \right| y_{\xi}y_{\Delta x, \xi} d\xi \nonumber  &\leq \int_{\R} \left|f(Z) (y) - \psi (y) \right|y_{\Delta x, \xi}y_{\xi} d\xi \nonumber \\ 
   & \quad + \int_{\R} \left| \psi (y) - \psi (y_{\Delta x}) \right|y_{\Delta x, \xi}y_{\xi}d\xi \nonumber \\
   & \quad + \int_{\R} \left|\psi (y_{\Delta x}) - f(Z) (y_{\Delta x}) \right|y_{\Delta x, \xi}y_{\xi}d\xi \nonumber  \\ 
   & \leq 2 \epsilon + \int_{\R} \left| \psi (y) - \psi (y_{\Delta x}) \right|y_{\Delta x, \xi}y_{\xi}d\xi,
   \label{eq:density_argument_g}
\end{align}
since $0\leq y_\xi$, $y_{\Delta x,\xi}\leq 1$. 
As $y_{\Delta x} \rightarrow y$ in $L^{\infty}(\R)$ by \eqref{eq:initial_approx_y}, the support of $\psi (y_{\Delta x})$ is contained inside some compact set which can be chosen independently of $\Delta x$. Therefore, by the Lebesgue dominated convergence theorem, $\psi(y_{\Delta x}) \rightarrow \psi (y)$ in $L^1(\R)$. Consequently, the left hand side of \eqref{eq:density_argument_g} vanishes as $\epsilon \rightarrow 0$ and $\Delta x \rightarrow 0$.

For the term \eqref{eq:first_term_g_2} we use the change of variables $x=y_{\Delta x}(\xi)$, $0 \leq y_{\xi} \leq 1$, and Lemma \ref{lem:convergence_g} to deduce that 
\begin{align*}
    \int_{E_{\Delta x} \cap A^c} |f(Z) (&y_{\Delta x}) - f(Z_{\Delta x}) (y_{\Delta x}) |y_{\xi}y_{\Delta x, \xi} d\xi \\ 
    &\leq \int_{y_{\Delta x}(E_{\Delta x} \cap A^c)} \left|f(Z) - f(Z_{\Delta x})\right|dx \leq 8 \sqrt{F_{\mathrm{ac}, \infty}} \|u_{x} - u_{\Delta x, x} \|_{2}. 
\end{align*}

For \eqref{eq:second_part_g} note that  
\begin{equation*}
f(Z)(y(\xi))y_\xi(\xi)=g(X)(\xi)\leq V_\xi(\xi)=H_\xi(\xi)\leq y_\xi(\xi)  \text{ for a.e. } \xi\in E_{\Delta x} \cap A^c,
\end{equation*}
which implies 
\begin{align*}
    \int_{E_{\Delta x} \cap A^c} \left| u_{\Delta x, x}^2 ( y_{\Delta x}) - u_x^2 (y) \right|& f(Z) (y) y_{\xi}y_{\Delta x, \xi}d\xi \\ &\leq \int_{E_{\Delta x} \cap A^c} \left|u_{\Delta x,x}^2 (y_{\Delta x}) - u_x^2 (y) \right|y_{\xi}y_{\Delta x, \xi}d\xi .
\end{align*}
Since  $u_x^2 \in L^1(\R)$, an argument similar to the one for the integrals of \eqref{eq:first_term_g_1} and \eqref{eq:first_term_g_2} shows that the term on the left hand side tends to zero as $\Delta x\to 0$. 

 Finally, for \eqref{eq:third_part_g}, we observe that 
 \begin{equation*}
u_x^2(y(\xi))y_\xi= V_\xi(\xi)=H_\xi(\xi)\leq y_\xi(\xi)  \text{ for a.e. } \xi\in E_{\Delta x} \cap A^c,
\end{equation*}
and hence \eqref{eq:third_part_g} is bounded from above by 
\begin{equation*}
    \int_{E_{\Delta x} \cap A^c} \left|f(Z)(y) - f(Z_{\Delta x}) (y_{\Delta x}) \right|y_{\Delta x,\xi}y_{\xi}d\eta,
\end{equation*}
which is a combination of the integrals of \eqref{eq:first_term_g_1} and \eqref{eq:first_term_g_2} which tend to zero as $\Dx \rightarrow 0$.  This finishes the proof of $g(X_{\Delta x}) \to g(X)$ in $L^1(\R)$.
\end{proof}

To establish convergence at later times in Lagrangian coordinates, we equip the set $\F$ with the following metric, which has been introduced in \cite[Def. 4.6]{alphaHS}. 
\begin{definition}
    Let $d: \F \times \F \rightarrow [0, \infty)$ be defined by 
    \begin{align*}
        d(X, \hat{X}) &= \|y - \hat{y} \|_{\infty} + \|U - \hat{U} \|_{\infty} + \|H_{\xi} - \hat{H}_{\xi}\|_1 + \|y_{\xi} - \hat{y}_{\xi}\|_2 \\ & \qquad + \|U_{\xi} - \hat{U}_{\xi} \|_2 + \|g(X) + y_{\xi}- g(\hat{X}) - \hat{y}_{\xi} \|_2 + \|H_{\xi} - \hat{H}_{\xi} \|_2.
    \end{align*}
\end{definition}
This metric has one major drawback: it separates solutions in Lagrangian coordinates belonging to the same equivalence class. Nevertheless, it is well suited for us, since we are only interested in comparing $X_{\Delta x}(t) = S_t (X_{\Delta x}(0))$ with $X(t) = S_t (X(0))$ and not their respective equivalence classes.

\begin{theorem}\label{thm:convergence_lagrangian_later}
 Given $(u_0, \mu_0, \nu_0) \in \D_0$ and $t\geq 0$, let $X(t) = S_t \circ L \left( (u_0, \mu_0, \nu_0 )\right)$ and $X_{\Delta x}(t) = S_{t} \circ L \circ P_{\Delta x} \left( (u_0, \mu_0, \nu_0) \right)$, then as $\Dx \rightarrow 0$, 
 \begin{subequations}
 \begin{align}
     y_{\Delta x}(t) & \rightarrow y(t) \text{ in } L^{\infty}(\R), \label{eq:convergence_y_later}\\
     U_{\Delta x}(t) & \rightarrow U(t) \text{ in } L^{\infty}(\R), \label{eq:convergence_U_later}\\
     y_{\Delta x, \xi}(t) & \rightarrow y_{\xi}(t) \text{ in } L^2(\R), \label{eq:convergence_y_xi_later} \\
     U_{\Delta x, \xi}(t) &\rightarrow U_{\xi}(t) \text{ in } L^2(\R), \label{eq:convergence_U_xi_later} \\
     H_{\Delta x, \xi}(t) & \rightarrow H_{\xi}(t) \text{ in } L^1(\R)\cap L^2(\R), \label{eq:convergence_H_xi_later}\\ 
     g(X_{\Delta x})(t) & \rightarrow g(X)(t) \text{ in } L^1(\R) \cap L^2(\R). \label{eq:convergence_g_later_1} 
      \end{align}
 \end{subequations}
 Furthermore,
 \begin{subequations}
 \begin{align}
 	\|H_{\Delta x}(t) - H(t) \|_{\infty}&\leq 2\Delta x,  \label{eq:convergence_H_later} \\
     \int_0^t\|V_{\Delta x, \xi}(s) - V_{\xi}(s)\|_{L^1(\R)}\,ds &\leq (1+\alpha)t\|H_{\xi}(0) - H_{\Delta x, \xi}(0) \|_1 \nonumber \\ & \quad + t\|g(X(0)) - g(X_{\Delta x}(0)) \|_1 \nonumber  \\ & \quad + 2\alpha \sqrt{2 \left(1 + \frac{1}{4}t^2 \right)G_{\infty}(0)} \|U_{\xi}(0) - U_{\Delta x, \xi}(0) \|_2.  \label{eq:convergence_V_xi_later} 
 \end{align}
 \end{subequations}
\end{theorem}

\begin{proof}
Let $X(0)=L((u_0, \mu_0,\nu_0))$ and $X_{\Delta x}(0)= L\circ P_{\Delta x}((u_0,\mu_0,\nu_0))$. Combining Lemma \ref{lem:initial_conv_lagrangian}, Lemma \ref{lem:Convergences_derivatives_without_rates}, and Proposition \ref{prop:convergence_g_initially} yields $d(X(0), X_{\Delta x}(0)) \rightarrow 0$. Furthermore, see \cite[Thm. 4.18]{alphaHS}, 
\begin{align*}
    d(X(t), X_{\Delta x}(t)) &\leq \left(3 + \frac{3}{2}t + \frac{1}{2}t^2 + \frac{3}{16}t^3 + \sqrt{F_{\infty}(0)} \left(1 + \frac{1}{4}t + \frac{1}{4}t^2 + \frac{1}{16}t^3 \right) \right)\\ 
    & \qquad \times e^{ \left(2 + \sqrt{F_\infty(0)} \left(\frac{1}{2} + \frac{1}{8}t + \frac{1}{16}t^2 \right) \right)t} d(X(0), X_{\Delta x}(0)). 
\end{align*}
and hence \eqref{eq:convergence_y_later}--\eqref{eq:convergence_H_xi_later} hold.

Since $g(X)(t,\cdot)-g(X_{\Delta x})(t,\cdot)$ and $H(t,\cdot)-H_{\Delta x}(t,\cdot)$ are time-independent, \eqref{eq:convergence_g_later_1} and \eqref{eq:convergence_H_later} follow immediately from Proposition~\ref{prop:convergence_g_initially} and \eqref{eq:initial_approx_H}. 

It remains to show \eqref{eq:convergence_V_xi_later}. Recalling \eqref{def.omegas}, introduce
\begin{equation*}
		\Omega_{n, m}(t) = \Omega_n(X(t)) \cap \Omega_m(X_{\Delta x}(t)), \quad \text{for } n, m \in \{c, d \}. 
\end{equation*}
Using Fubini's theorem, we write
\begin{subequations}
\begin{align}\nonumber
    \int_0^t\int_{\R} \left| V_{\xi}(s, \xi) - V_{\Delta x, \xi}(s, \xi) \right|d\xi ds &\leq \int_{\Omega_{c, c}(0)} \int_0^t \left|V_{\xi}(s, \xi) - V_{\Delta x, \xi}(s, \xi) \right|ds d\xi \\ 
    & \qquad + \int_{\Omega_{d, c}(0)} \int_0^t \left|V_{\xi}(s, \xi) - V_{\Delta x, \xi}(s, \xi) \right|ds d\xi \label{eq:L1_conv_V_xi_2}\\ 
    & \qquad + \int_{\Omega_{c, d}(0)} \int_0^t \left|V_{\xi}(s, \xi) - V_{\Delta x, \xi}(s, \xi) \right|ds d\xi \label{eq:L1_conv_V_xi_3}\\ 
    & \qquad + \int_{\Omega_{d, d}(0)} \int_0^t \left|V_{\xi}(s, \xi) - V_{\Delta x, \xi}(s, \xi) \right|ds d\xi. \label{eq:L1_conv_V_xi_4}
\end{align}
\end{subequations}
Since no wave breaking occurs for $\xi \in \Omega_{c,c}(0)$,
\begin{align}
    \int_{\Omega_{c, c}(0)} \int_0^t |V_{\xi}(s, \xi) &- V_{\Delta x, \xi}(s, \xi) |ds d\xi \nonumber \\ &= \int_{\Omega_{c, c}(0)}\int_0^t\left|g(X)(0, \xi) - g(X_{\Delta x})(0, \xi) \right|ds d\xi. 
    \label{eq:Vxi_term1}
\end{align}

The terms \eqref{eq:L1_conv_V_xi_2} and \eqref{eq:L1_conv_V_xi_3} can be treated similarly, so we only consider \eqref{eq:L1_conv_V_xi_2}. Since any $\xi\in \Omega_d(X_{\Delta x}(0))$ might enter $\Omega_c$ within the time interval $[0,t]$, we have  
\begin{align}
\label{eq:Vxi_term2} \nonumber
    \int_{\Omega_{d, c}(0)} \int_0^t |&V_{\xi}(s, \xi) - V_{\Delta x, \xi}(s, \xi) |ds d\xi 
    \\ \nonumber &= \int_{\Omega_{d, c}(0) \cap \Omega_{d, c}(t)} \int_0^t \left|V_{\xi}(0, \xi) - V_{\Delta x, \xi}(0, \xi) \right|ds d\xi \\  \nonumber
    & \qquad + \int_{\Omega_{d, c}(0) \cap \Omega_{c, c}(t)} \int_0^{\tau(\xi)} \left|V_{\xi}(0, \xi) - V_{\Delta x, \xi}(0, \xi) \right|ds d\xi \\  \nonumber
    & \qquad +  \int_{\Omega_{d, c}(0) \cap \Omega_{c, c}(t)} \int_{\tau(\xi)}^t \left|(1-\alpha)V_{\xi}(0, \xi) - V_{\Delta x, \xi}(0, \xi) \right|ds d\xi  \\ \nonumber
    &\leq \int_{\Omega_{d, c}(0)}\int_0^t \left|H_{\xi}(0, \xi) - H_{\Delta x, \xi}(0, \xi) \right|ds d\xi \\ 
    & \qquad + \int_{\Omega_{d, c}(0)} \int_0^t \left|g(X)(0, \xi) - g(X_{\Delta x})(0, \xi) \right|ds d\xi,
\end{align}
where we used in the last step that $X(0)$ and $X_{\Delta x}(0)$ belong to $\F_0^0$.

For \eqref{eq:L1_conv_V_xi_4} we use a similar decomposition, yielding
\begin{subequations}
\begin{align}
    \int_{\Omega_{d, d}(0)} \int_0^t |V_{\xi}(s, \xi) &- V_{\Delta x, \xi}(s, \xi) |ds d\xi \nonumber \\ &= \int_{\Omega_{d, d}(0) \cap \Omega_{d, d}(t)} \int_0^t\left|V_{\xi}(s, \xi) - V_{\Delta x, \xi}(s, \xi) \right|ds d\xi \notag \\ 
    &  \quad + \int_{\Omega_{d, d}(0) \cap \Omega_{c, d}(t)} \int_0^t\left|V_{\xi}(s, \xi) - V_{\Delta x, \xi}(s, \xi) \right|ds d\xi \label{eq:double_negative_2}\\ 
    &  \quad + \int_{\Omega_{d, d}(0) \cap \Omega_{d, c}(t)} \int_0^t\left|V_{\xi}(s, \xi) - V_{\Delta x, \xi}(s, \xi) \right|ds d\xi \label{eq:double_negative_3}\\ 
    &  \quad + \int_{\Omega_{d, d}(0) \cap \Omega_{c, c}(t)} \!\int_0^t\left|V_{\xi}(s, \xi) - V_{\Delta x, \xi}(s, \xi) \right|ds d\xi. \label{eq:double_negative_4}
\end{align}
\end{subequations}
For the first term we have
\begin{align}
    \int_{\Omega_{d, d}(0) \cap \Omega_{d, d}(t)} \int_0^t|V_{\xi}(s, \xi) &- V_{\Delta x, \xi}(s, \xi) |ds d\xi \nonumber \\
    &=  \int_{\Omega_{d, d}(0) \cap \Omega_{d, d}(t)} \int_0^t \left|H_{\xi}(0, \xi) - H_{\Delta x, \xi}(0, \xi) \right|ds d\xi. 
    \label{eq:Vxi_term3}
\end{align}

 For \eqref{eq:double_negative_2}, recalling \eqref{eq:stability_function_g} and that $X(0)$ and $X_{\Delta x}(0)$ belong to $\F_0^0$, we have 
\begin{align*}
    \int_{\Omega_{d, d}(0) \cap \Omega_{c, d}(t)} \int_0^t&|V_{\xi}(s, \xi) - V_{\Delta x, \xi}(s, \xi) |ds d\xi \nonumber \\
    &= \int_{\Omega_{d, d}(0) \cap \Omega_{c, d}(t)} \int_0^{\tau(\xi)} \left|V_{\xi}(0, \xi) - V_{\Delta x, \xi}(0, \xi) \right|ds d\xi \nonumber \\ 
    & \qquad + \int_{\Omega_{d, d}(0) \cap \Omega_{c, d}(t)} \int_{\tau(\xi)}^t \left|(1-\alpha)V_{\xi}(0, \xi) - V_{\Delta x, \xi}(0, \xi) \right|ds d\xi \nonumber \\ 
    &\leq  \int_{\Omega_{d, d}(0) \cap \Omega_{c, d}(t)} \int_0^t \left|H_{\xi}(0, \xi) - H_{\Delta x, \xi}(0, \xi) \right|ds d\xi  
    \\ 
    & \qquad + \int_{\Omega_{d, d}(0) \cap \Omega_{c, d}(t)} \int_{\tau(\xi)}^t \alpha V_{\Delta x, \xi}(0, \xi)dsd\xi. \nonumber  
\end{align*}
For the last term in the above inequality, observe that $U_\xi(t,\xi)$ and $U_{\Delta x, \xi}(t,\xi)$ are increasing functions, cf. \eqref{eq:integrated_ODE2}, which equal zero at $t= \tau(\xi)$ and $t= \tau_{\Delta x}(\xi)$, respectively, for $\xi \in \Omega_{d, d}(0)$. Furthermore, \eqref{eq:integrated_ODE2} implies 
\begin{align*}
    \int_{\Omega_{d, d}(0) \cap \Omega_{c, d}(t)} \! \int_{\tau(\xi)}^t \alpha &V_{\Delta x, \xi}(0, \xi)dsd\xi \\ 
    = &2 \alpha \! \int_{\Omega_{d, d}(0) \cap \Omega_{c, d}(t)} \left(U_{\Delta x, \xi}(t, \xi) - U_{\Delta x, \xi}(\tau(\xi), \xi) \right)d\xi\\
    \leq & \, 2\alpha \int_{\Omega_{d, d}(0) \cap \Omega_{c, d}(t)} \left(U_{\xi}(\tau(\xi), \xi) - U_{\Delta x, \xi}(\tau(\xi), \xi) \right)d\xi \\ 
    \leq & \, 2 \alpha  \int_{\Omega_{d, d}(0) \cap \Omega_{c, d}(t)} \left|U_{\xi}(0, \xi) - U_{\Delta x, \xi}(0, \xi) \right|d\xi \\  
    & + \alpha \! \int_{\Omega_{d, d}(0) \cap \Omega_{c, d}(t)} \!\int_0^{\tau(\xi)} \left| H_{\xi}(0, \xi) - H_{\Delta x, \xi}(0, \xi) \right|d\xi, 
\end{align*}
and thus
\begin{align}
     \int_{\Omega_{d, d}(0) \cap \Omega_{c, d}(t)} \int_0^t|V_{\xi}&(s, \xi) - V_{\Delta x, \xi}(s, \xi) |ds d\xi \nonumber \\
     &\leq \left(1 + \alpha \right)  \int_{\Omega_{d, d}(0) \cap \Omega_{c, d}(t)} \int_0^t \left|H_{\xi}(0, \xi) - H_{\Delta x, \xi}(0, \xi) \right|ds d\xi    \nonumber \\ 
     & \qquad + 2\alpha \int_{\Omega_{d, d}(0) \cap \Omega_{c, d}(t)} \left|U_{\xi}(0, \xi) - U_{\Delta x, \xi}(0, \xi) \right|d\xi. \label{eq:Vxi_term4}  
\end{align}
The term \eqref{eq:double_negative_3} is bounded similarly.

To estimate \eqref{eq:double_negative_4}, we split $\Omega_{d, d}(0) \cap \Omega_{c, c}(t)$ into two sets,
\begin{equation*}
\Omega_{d, d}(0) \cap \Omega_{c, c}(t)= A(t) \cup A_{\Delta x}(t),
\end{equation*}
where
\begin{equation*}
A(t) := \{\xi: 0<\tau(\xi) \leq \tau_{\Delta x}(\xi) \leq t \}\text{ and } A_{\Delta x}(t):= \{\xi:0< \tau_{\Delta x}(\xi) < \tau(\xi) \leq t \}.
\end{equation*}
We only present the details for the integral over $A(t)$, since the other one can be treated similarly. Following closely the argument leading to \eqref{eq:Vxi_term4}, we get
\begin{align}\nonumber
    \int_{ A(t)}\int_0^t|V_{\xi}(s, \xi) - V_{\Delta x, \xi}(s, \xi) |ds d\xi 
    &\leq\int_{A(t)} \int_0^{t}\left| H_{\xi}(0, \xi) - H_{\Delta x, \xi}(0, \xi) \right|ds d\xi \\ \nonumber
    & \qquad + \int_{ A(t)}\int_{\tau(\xi)}^{\tau_{\Delta x}(\xi)} \alpha V_{\Delta x}(0,\xi)ds d\xi\\ \nonumber
   &\leq \left(1+\alpha \right) \int_{ A(t)} \int_0^t \left|H_{\xi}(0, \xi) - H_{\Delta x, \xi}(0, \xi) \right|ds d\xi \\ 
   & \quad + 2\alpha \int_{  A(t)} \left|U_{\xi}(0, \xi) - U_{\Delta x, \xi}(0, \xi) \right|d\xi. \label{eq:Vxi_term5}
\end{align}

Combining \eqref{eq:Vxi_term1}--\eqref{eq:Vxi_term2}, \eqref{eq:Vxi_term3}, \eqref{eq:Vxi_term4}--\eqref{eq:Vxi_term5}, and analogous estimates for the remaining cases, yields 
    \begin{align} \nonumber
        \int_0^t \int_{\R}|V_{\xi}(s, \xi) -V_{\Delta x, \xi}(s, \xi)|d\xi ds & \leq (1 + \alpha)t \int_{\R} \left| H_{\xi}(0, \xi) - H_{\Delta x, \xi}(0, \xi) \right|d\xi \\ \nonumber &\qquad + t \int_{\R} \left|g(X)(0, \xi) - g(X_{\Delta x})(0, \xi) \right|d\xi \\ & \qquad + 2\alpha  \int_{B(t)} \left|U_{\xi}(0, \xi) - U_{\Delta x, \xi}(0, \xi) \right|d\xi, \label{eq:pre_estimate_V_xi}
    \end{align}    
where 
    \begin{equation*}
        B(t) = \{\xi \in \R: \tau(\xi) \leq t \} \bigcup \{\xi \in \R: \tau_{\Delta x}(\xi) \leq t \}. 
    \end{equation*}
    Furthermore, $\mathrm{meas} \! \left(\{\xi \in \R: \tau(\xi) \leq t \}\right) \leq (1 + \frac{1}{4}t^2)H_{\infty}(0)$, and likewise for the set $\{\xi: \tau_{\Delta x}(\xi) \leq t \}$, by \cite[Cor. 2.4]{alphaHS}. Hence 
\begin{align}
    \mathrm{meas}\!\left(B(t) \right) \leq 2\left(1 + \frac{1}{4}t^2 \right)H_{\infty}(0).
        \label{eq:measure_one_broken}
    \end{align}
Thus, by applying the Cauchy--Schwarz inequality to the last term in \eqref{eq:pre_estimate_V_xi} and inserting \eqref{eq:measure_one_broken}, we obtain \eqref{eq:convergence_V_xi_later}. 
\end{proof}

\begin{corollary} \label{corr:PointwiseConv}
 Let $(u_0, \mu_0, \nu_0) \in \D_0$ and $t \geq 0$. Set $X(t) = S_t \circ L \left((u_0, \mu_0, \nu_0 ) \right)$ and $X_{\Delta x}(t) = S_{t} \circ L \circ P_{\Delta x} \left( (u_0, \mu_0, \nu_0 )\right)$. Then there exists a subsequence $\{\Delta x_m \}_{ m \in \mathbb{N}}$ with $\Delta x_m \rightarrow 0$ as $m \rightarrow \infty$, such that for a.e. $t \in [0, \infty)$ we have
    \begin{align}
        V_{\Delta x_m, \xi}(t,\cdot) & \rightarrow V_{\xi}(t,\cdot) \text{ in } L^1(\R), \notag \\
        V_{\Delta x_m, \infty}(t) & \rightarrow V_{\infty}(t). \label{eq:seq_converges_V_xi}
    \end{align}
\end{corollary}
\begin{proof}
By \eqref{eq:convergence_V_xi_later} and \cite[Cor. 2.32]{RealAnalysisFolland}, there exists a subsequence $\{V_{\Delta x_m, \xi }\}_{m \in \mathbb{N}}$ such that $V_{\Delta x_m, \xi}(t,\cdot) \rightarrow V_{\xi}(t,\cdot)$ in $L^1(\R)$ for a.e. $t \in [0, \infty)$. Let $N \subset [0, \infty)$ be the null set of times for which the convergence does not hold. Thus, for $t \in N^c$ we have 
\begin{align*}
    V_{\Delta x_m, \infty}(t) = \|V_{\Delta x_m, \xi}(t) \|_1 \rightarrow \|V_{\xi}(t) \|_1 = V_{\infty}(t). \hspace{1.5cm} \qedhere
\end{align*}
\end{proof}

Observe that the times for which the convergence in Corollary~\ref{corr:PointwiseConv} fails depend on the particular chosen subsequence. Therefore, there is no natural way to extend this convergence to the whole sequence.

\subsection{Convergence of the \texorpdfstring{$\alpha$}{\textalpha}-dissipative solution in Eulerian coordinates}

Finally, we can examine in what sense convergence in Lagrangian coordinates, given by Theorem \ref{thm:convergence_lagrangian_later} and Corollary \ref{corr:PointwiseConv}, carries over to Eulerian coordinates. 

\begin{lemma}
    Given $(u_0, \mu_0, \nu_0) \in \D_0$, let $(u, \mu, \nu)(t)= T_t \left((u_0, \mu_0, \nu_0 )\right)$ and \phantom{aa} $(u_{\Delta x}, \mu_{\Delta x}, \nu_{\Delta x})(t) = T_t \circ P_{\Delta x} \left( (u_0, \mu_0, \nu_0 )\right)$ for $t \in [0, \infty)$. Then
    \begin{align}
        \|u(t) - u_{\Delta x}(t)\|_{\infty} &\leq \|U(t) - U_{\Delta x}(t) \|_{\infty} + \sqrt{F_\infty(0)} \sqrt{\|y(t) - y_{\Delta x}(t) \|_{\infty}}, \label{eq:convergence_u} \\
         \|G(t) - G_{\Delta x}(t) \|_{1} &\leq 4e^{\frac12 t}G_{\infty}(0) \Delta x + 2G_\infty(0)\|y(t) - y_{\Delta x}(t) \|_\infty \nonumber \\
          &+e^{\frac{1}{4} t}G_\infty(0)^{\frac{3}{2}} \|y_{\xi}(t) - y_{\Delta x, \xi}(t) \|_2 . \label{eq:convergence_G_1} \end{align}
   \end{lemma}

\begin{proof}
Following the same line of reasoning as in the proof of \eqref{eq:initial_approx_U}, yields  \eqref{eq:convergence_u}. 

Before proving \eqref{eq:convergence_G_1}, recall that by Definition~\ref{def:MappingM}, we have  
\begin{equation*}
 G(t,y(t,\xi))=H(t,\xi) \quad \text{ for all } \xi \text{ such that } G(t,\cdot) \text{ is continuous at } y(t,\xi)
 \end{equation*} 
 and likewise for $G_{\Delta x}(t,x)$. Thus, the change of variables $x=y(t,\xi)$ yields
 \begin{subequations}\label{splitting:G:L1}
\begin{align}
    \|G(t) - G_{\Delta x}(t) \|_1 &= \int_{\R} |G(t, y(t, \xi)) - G_{\Delta x}(t, y(t, \xi)) |y_{\xi}(t, \xi)d\xi \nonumber \\ 
    &\leq \int_{\R}|H(t, \xi) - H_{\Delta x}(t, \xi)|y_{\xi}(t, \xi)d\xi \notag  \\ 
    & \qquad + \int_{\R}|H_{\Delta x}(t, \xi) - G_{\Delta x}(t, y_{\Delta x}(t, \xi))|y_{\xi}(t, \xi)d\xi \label{eq:L1_conv_G_part_1_2}\\ 
    & \qquad + \int_{\R}|G_{\Delta x}(t, y(t, \xi)) - G_{\Delta x}(t, y_{\Delta x}(t, \xi))|y_{\xi}(t, \xi)d\xi. \label{eq:L1_conv_G_part_1_3} 
\end{align}
\end{subequations}
Following the steps in the proof of \cite[(2.15)]{alphaHS}, we can establish
\begin{equation}\label{grow:equiv}
e^{-\frac{t}{2}}(y_{\xi}(0,\xi)+H_\xi(0,\xi))\leq y_\xi(t,\xi)+H_\xi(t,\xi)\leq e^{\frac12t}(y_\xi(0,\xi)+H_\xi(0,\xi)),
\end{equation}
for almost every $\xi\in \R$,
which implies that 
\begin{align}\nonumber
 \int_{\R}|&H(t, \xi) - H_{\Delta x}(t, \xi)|y_{\xi}(t, \xi)d\xi \nonumber \\ 
 &\leq e^{\frac12 t}\int_{\R}|H(0,\xi)-H_{\Delta x}(0,\xi)|y_\xi(0,\xi) d\xi \nonumber \\ \nonumber
 & \quad + e^{\frac12 t} \int_\R|H(0,\xi)-H_{\Delta x}(0,\xi)|H_\xi(0,\xi) d\xi\\ \nonumber
 & \leq e^{\frac12 t}\sum_{j=-\infty}^\infty \int_{\hat \xi_{3j}}^{\hat\xi_{3j+3}} \vert G_0(x_{2j+2})-G_0(x_{2j})\vert y_\xi(0,\xi) d\xi + 2e^{\frac12 t} H_\infty(0)\Delta x\\ 
 & =  2e^{\frac12 t} \sum_{j=-\infty}^\infty (G_0(x_{2j+2})-G_0(x_{2j}))\Delta x+2e^{\frac12 t} G_\infty(0) \Delta x = 4e^{\frac12 t} G_\infty(0)\Delta x,\label{eq:L1:conv:conv}
 \end{align}
 where we used \eqref{eq:convergence_H_later} and \eqref{eq:char_coincide_even}. 
 
 For \eqref{eq:L1_conv_G_part_1_2}, note that, by construction, $y_{\Delta x}(t,\xi)$ and $H_{\Delta x}(t,\xi)$ are piecewise linear functions with nodes at the points $\hat \xi_j$ with $j\in \mathbb{Z}$. Furthermore, $H_{\Delta x}(t, \xi) - G_{\Delta x}(t, y_{\Delta x}(t, \xi))>0$ on an interval $I$ if and only if $y_{\Delta x}(t,\xi) $ is constant on this interval. Assuming that such an interval $I$ exists, an upper bound on its length can be established as follows. Observe that 
 \begin{equation*}
 I\subset \mathcal{S}_{\Delta x}(t)=\left \{\xi : \frac{y_{\Delta x, \xi}}{y_{\Delta x, \xi}+H_{\Delta x,\xi}}(t, \xi)=0 \right\}= \left \{\xi : \frac{H_{\Delta x, \xi}}{y_{\Delta x, \xi}+H_{\Delta x, \xi}}(t,\xi)=1 \right\},
 \end{equation*} 
 and hence, due to \eqref{grow:equiv},
 \begin{equation*}
 \mathrm{meas}(I)\leq \mathrm{meas}(\mathcal{S}_{\Delta x}(t))\leq \int_\R\frac{H_{\Delta x, \xi}}{y_{\Delta x, \xi}+H_{\Delta x, \xi}}(t,\xi)d\xi\leq e^{\frac12 t} H_{\Delta x, \infty}(0)=e^{\frac12 t} H_{\infty}(0). 
 \end{equation*}
Thus, it holds that 
\begin{align}
    \int_{\R}|H_{\Delta x}&(t, \xi) - G_{\Delta x}(t, y_{\Delta x}(t, \xi))|y_{\xi}(t, \xi)d\xi \nonumber \\  
    &\leq \int_{\mathcal{S}_{\Delta x}(t)}|H_{\Delta x}(t, \xi) - G_{\Delta x}(t, y_{\Delta x}(t, \xi))||y_{\xi}(t, \xi) - y_{\Delta x}(t, \xi)|d\xi \nonumber \\ 
    &\leq H_{\infty}(0) \mathrm{meas}(\mathcal{S}_{\Delta x}(t))^{\frac{1}{2}}\|y_{\xi}(t) - y_{\Delta x, \xi}(t) \|_2 \nonumber\\ 
    &\leq H_{\infty}(0)^{\frac{3}{2}} e^{\frac{1}{4}t} \|y_{\xi}(t) - y_{\Delta x, \xi}(t) \|_2. \label{eq:temp_second_term_G}
\end{align} 
To estimate \eqref{eq:L1_conv_G_part_1_3}, we exploit that $G_{\Delta x}(t)\in \text{BV}(\R)$, since it is monotonically increasing and bounded. Introducing $\Theta = \|y(t) - y_{\Delta x}(t) \|_{\infty}$, then yields
\begin{align}
\int_{\R}|G_{\Delta x}(t, &y(t, \xi)) - G_{\Delta x}(t, y_{\Delta x}(t, \xi))|y_{\xi}(t, \xi)d\xi \nonumber \\
&= \int_{\R} \left| G_{\Delta x}(t, y(t, \xi)) - G_{\Delta x}\left(t, y(t, \xi) - y(t, \xi) + y_{\Delta x}(t, \xi) \right) \right|y_{\xi}(t, \xi)d\xi \nonumber \\ 
&\leq \int_{\R} \left| G_{\Delta x}(t, y(t, \xi) + \Theta) - G_{\Delta x}(t, y(t, \xi) - \Theta ) \right|y_{\xi}(t, \xi)d\xi \nonumber \\ 
&= \sum_{j \in \mathbb{Z}}\int_{j\Theta}^{(j+1)\Theta} \left|G_{\Delta x}(t, x+\Theta) - G_{\Delta x}(t, x - \Theta) \right|dx \nonumber \\  
&= \int_0^{\Theta}\sum_{j \in \mathbb{Z}}|G_{\Delta x}(t, x + (j+1)\Theta) - G_{\Delta x}(t, x + (j-1)\Theta)|dx \nonumber \\ 
&\leq 2\Theta G_{\Delta x,\infty}(t) = 2G_\infty(0) \|y(t) - y_{\Delta x}(t) \|_{\infty}.
    \label{eq:BV_argument_G}
\end{align}
where we applied Tonelli's theorem to exchange the order of the sum and integral.

Combining \eqref{splitting:G:L1}, \eqref{eq:L1:conv:conv},   \eqref{eq:temp_second_term_G}, and \eqref{eq:BV_argument_G}, we end up with \eqref{eq:convergence_G_1}.  
\end{proof}

\begin{theorem}
Given $(u_0, \mu_0, \nu_0) \in \D_0$, let $(u, \mu, \nu)(t)= T_t \left((u_0, \mu_0, \nu_0)\right)$ and $(u_{\Delta x}, \mu_{\Delta x}, \nu_{\Delta x})(t) = T_t \circ P_{\Delta x} \left((u_0, \mu_0, \nu_0)\right)$ for $t \in [0, \infty)$. Then,  as $\Delta x\to 0$\footnote{We  say that $\eta_{\Delta x} \Longrightarrow \eta$ if $\int_{\R}\phi d\eta_{\Delta x} \rightarrow \int_{\R}\phi d \eta$ for all  $\phi \in C_b(\R):= C(\R) \cap L^{\infty}(\R)$.},
\begin{align}
u_{\Delta x}(t) &\to u(t) \text{ in } L^\infty(\R),\label{eq:convergence_u2}\\
u_{\Delta x, x}(t) &\rightharpoonup u_x(t) \text{ in } L^2(\R), \label{eq:convergence_weakly_ux}\\
G_{\Delta x}(t)& \to G(t) \text{ in } L^1(\R), \label{eq:convergence_G_11}\\
 \nu_{\Delta x}(t) & \Longrightarrow \nu(t) \label{eq:convergence_G_2},\\
  (y_{\Delta x})_{\#}(g(X_{\Delta x}(t))d\xi)  &\Longrightarrow  y_{\#}((g(X(t))d\xi)  \label{eq:convergence_g_weak}.
 \end{align}

  Furthermore, there exists a subsequence $\{\Delta x_m \}_{m \in \mathbb{N}}$ with $\Delta x_m \rightarrow 0$ as $m \rightarrow \infty$, and a null set $N \subset [0, \infty)$ such that for all $t \in N^c$ we have as $\Delta x_m\to 0$
    \begin{align}
        F_{\Delta x_m}(t, x)  &\rightarrow F(t, x) \text{ for every $x$ at which } F(t) \text{ is continuous}, \label{eq:convergence_F_1} \\ 
        F_{\Delta x_m, \infty}(t) & \rightarrow 
          F_{\infty}(t). \label{eq:convergence_F_2}
    \end{align}
    \end{theorem}
    
Note that in the case of conservative solutions, $F_{\Delta x} = G_{\Delta x}$ and hence \eqref{eq:convergence_G_1} implies that we have $L^1(\R)$-convergence of $F_{\Delta x}(t)$ for all $t \geq 0$.

\begin{proof}
Note that $U_{\Delta x}(t) \rightarrow U(t)$ and $y_{\Delta x}(t) \rightarrow y(t)$ in $L^{\infty}(\R)$ as $\Delta x\to 0$ by Theorem \ref{thm:convergence_lagrangian_later}, which combined with \eqref{eq:convergence_u} yields \eqref{eq:convergence_u2}.

 To prove \eqref{eq:convergence_weakly_ux}, observe that for any $\psi \in C_c^{\infty}(\R)$, 
 \begin{align*}
     \left|\int_{\R} \psi(x) \left(u_x(t, x) - u_{\Delta x, x}(t, x) \right)dx \right|  &= \left|\int_{\R}\psi'(x) \left(u(t, x) - u_{\Delta x}(t, x) \right)dx \right| \\ &\leq \|\psi ' \|_{1}\|u(t) - u_{\Delta x}(t) \|_{\infty},
 \end{align*}
 and hence, using \eqref{eq:convergence_u2}, the right hand side tends to zero as $\Delta x\to 0$. Since $C_c^\infty(\R)$ is a dense subset of $L^2(\R)$, we end up with $u_{\Delta x, x}(t) \rightharpoonup u_x(t)$ in $L^2(\R)$ as $\Delta x \rightarrow 0$. 

Next, note that $y_{\Delta x}(t) \rightarrow y(t)$ in $L^{\infty}(\R)$ and $y_{\Delta x,\xi}(t) \rightarrow y_\xi(t)$ in $L^{2}(\R)$ as $\Delta x\to 0$ by Theorem \ref{thm:convergence_lagrangian_later}, and hence \eqref{eq:convergence_G_11} follows from \eqref{eq:convergence_G_1}

To prove \eqref{eq:convergence_G_2}, let $\phi \in C_b(\R)$. Then, by Definition~\ref{def:MappingM}, we have
\begin{subequations}
\begin{align}
    \int_{\R}\phi(x)d\nu(t, x) &- \int_{\R} \phi(x)d\nu_{\Delta x}(t, x) \nonumber \\ 
    &=\int_{\R}\phi(y(t, \xi))H_{\xi}(t, \xi)d\xi  - \int_{\R}\phi(y_{\Delta x}(t, \xi))H_{\Delta x, \xi}(t, \xi)d\xi  \nonumber \\ 
    &=  \int_{\R} \left(\phi(y(t, \xi)) - \phi(y_{\Delta x}(t, \xi)) \right)H_{\xi}(t, \xi)d\xi \label{eq:weak_convergence_nu_1} \\ 
    & \qquad + \int_{\R} \phi(y_{\Delta x}(t, \xi)) \left(H_{\xi}(t, \xi) - H_{\Delta x, \xi}(t, \xi) \right)d\xi . \label{eq:weak_convergence_nu_2}
\end{align}
\label{eq:weak_convergence_nu}
\end{subequations}
Since $y_{\Delta x}(t) \rightarrow y(t)$ in $L^{\infty}(\R)$ by \eqref{eq:convergence_y_later} and $\phi\in C_b(\R)$, it follows that $\phi (y_{\Delta x}(t)) \rightarrow \phi(y(t))$ pointwise a.e.. Furthermore, $|\phi (y(t)) - \phi (y_{\Delta x}(t))|H_{\xi}(t) \leq 2\|\phi \|_{\infty}H_{\xi}(t)$ and hence, by the dominated convergence theorem, \eqref{eq:weak_convergence_nu_1} vanishes as $\Delta x \rightarrow 0$. 

For \eqref{eq:weak_convergence_nu_2} on the other hand, observe that $\phi (y_{\Delta x}(t))\in L^{\infty}(\R)$ and $H_{\Delta x, \xi}(t)\rightarrow H_{\xi}(t)$ in $L^1(\R)$ by \eqref{eq:convergence_H_xi_later}, imply
\begin{align*}
    \left| \int_{\R} \phi(y_{\Delta x}(t, \xi)) \left(H_{\xi}(t, \xi) - H_{\Delta x, \xi}(t, \xi) \right)d\xi \right| &\leq \|\phi \|_{\infty}\|H_{\xi}(t) - H_{\Delta x, \xi}(t) \|_1 \rightarrow 0
\end{align*}
as $\Delta x\to 0$. Thus the left hand side of \eqref{eq:weak_convergence_nu} tends to zero as $\Delta x \rightarrow 0$ for any $\phi\in C_b(\R)$ and we have shown \eqref{eq:convergence_G_2}.

As $g(X_{\Delta x}(t)) \rightarrow g(X(t))$ in $L^1(\R)$ by \eqref{eq:convergence_g_later_1}, the proof of \eqref{eq:convergence_g_weak} is completely analogous to the one of \eqref{eq:convergence_G_2}.

To show the remaining part of the theorem, recall Corollary \ref{corr:PointwiseConv}, which ensures the existence of a subsequence $\{\Delta x_m \}_{m \in \mathbb{N}}$ with $\Delta x_m \rightarrow 0$ as $m \rightarrow \infty$, such that $V_{\Delta x_m, \xi} (t) \rightarrow V_{\xi}(t)$ in $L^1(\R)$ for all $t \in N^c$, where $N\subset[0, \infty)$ is a null set. Therefore, applying a similar argument to the one used for proving  \eqref{eq:convergence_G_2}, we can establish that $\mu_{\Delta x_m}(t) \Longrightarrow \mu(t)$ for all $t\in N^c$. Choosing $\phi = 1 \in C_b(\R)$, we obtain \eqref{eq:convergence_F_2}.  Furthermore, as $C_0(\R) \subset C_b(\R)$, we have $\mu_{\Delta x_m}(t) \rightarrow \mu(t)$ vaguely for all $t \in N^c$ and \eqref{eq:convergence_F_1} holds by \cite[Prop. 7.19]{RealAnalysisFolland}. 
\end{proof}

\begin{remark}
    For each $t\in [0,\infty)$ and $\Delta x>0$,  $F_{\Delta x}(t)\in \text{BV}(\R)$. Thus,  by Helly's selection principle, see \cite[App. II.]{Billingsley}, there exists for every $t \in N$ a subsequence $\{F_{\Delta x_k}(t) \}_{k \in \mathbb{N}}$ such that $F_{\Delta x_k}(t)$ converges pointwise almost everywhere to a function of bounded variation.
\end{remark}

\begin{remark}
The projection operator $P_{\Delta x}$ is constructed with focus on accurately approximating the wave profile $u$, preserving the total energy and ensuring that 
    \begin{align*}
    d\mu_{\Delta x, \mathrm{ac}}
        = u_{\Delta x, x}^2dx.
    \end{align*}
  Therefore, the control over the size of the sets $\{x: u_x(x) \leq- \frac{2}{t} < u_{\Delta x, x}(x) < 0\}$ and $\{x: u_{\Delta x, x}(x) \leq- \frac{2}{t} < u_{x}(x) < 0\}$ is limited and as a consequence it prevents the convergence of the energy density $\mu_{\Delta x}$ for every fixed time. 
\end{remark}

\section{Numerical Experiments}\label{sec:Numex}
In this section we present three examples which highlight different challenges for the numerical algorithm. The first two examples are also considered in \cite{myMaster, NumericalConservative} and display solutions with multipeakon and cusp initial data, respectively. 
In the first example, wave breaking happens twice; initially, and also at a later time, and at each occurrence a finite amount of energy concentrates on a set of Lebesgue measure zero. For the second example on the other hand, the wave breaking times accumulate in the interval $[0, 3]$, but at every time only an infinitesimal amount of energy concentrates. The behavior of this latter example is thoroughly studied in \cite[Sec. 5.3]{myMaster}. In both cases, expressions for the exact solutions in Eulerian coordinates can be found.   

In the third example, we consider a cosine wave profile as initial data. Here, an expression for the Lagrangian solution is known, but the exact expression for $u$ in Eulerian coordinates is not available for comparison. The exact solution possesses accumulating wave breaking times, and, similarly to the cusp initial data, experiences wave breaking continuously in time, but now over the unbounded interval $[\frac{2}{\pi}, \infty)$. Furthermore, in contrast to the cusp example, the rate at which the total energy dissipates is nonlinear, making this an interesting challenge for the algorithm.

We present plots of both the time evolution and of the errors $\sup_{t \in [0,T]}\|u(t) - u_{\Delta x}(t) \|_{\infty}$ and $|F_{\infty}(T) - F_{\Delta x, \infty}(T)|$ for a chosen time $T$. In the first two examples, the errors are computed by comparing the exact and numerical solutions at the gridpoints of a uniform mesh which is a refinement of the finest mesh we use for the numerical approximations. The value of $u_{\Delta x}$ at the gridpoints is found by using \eqref{eq:piecewise_linear_recon}. As we do not have a closed form expression for the exact solution in Eulerian coordinates in the third example, Example \ref{ex:numerical_cosinus}, we instead use
\begin{align}
   \mathcal{A}_{\Delta x}(T):=\sup_{t \in [0, T]}\left(\|U(t) - U_{\Delta x}(t) \|_{\infty} + \sqrt{F_{\infty}(t)} \|y(t) - y_{\Delta x}(t)\|_{\infty}^{\frac{1}{2}} \right),
    \label{eq:quantity_rhs}
\end{align}
to measure the error introduced by the projection operator $P_{\Dx}$. Note that due to \eqref{eq:convergence_u}, this is an upper bound for the $L^\infty$-error of $u_{\Delta x}$ in Eulerian coordinates. 

\begin{example}[Multipeakon initial data]\label{ex:numerical_multipeakon}
Consider the Cauchy problem from Example \ref{ex:Example1} with $\alpha=\frac12$,  but now set $\nu_0 = \mu_0$, i.e.,  
\begin{align*}
    u_0(x) &= \begin{cases}
    0, & x < 0, \\
    x, & 0 \leq x < 1, \\
    2-x, & 1 \leq x <2, \\
    0, & 2 \leq x, \end{cases} \\
    d\mu_0 &= d\nu_0 =  \frac12\delta_{0} + u_{0,x}^2dx.
\end{align*}
The exact solution experiences wave breaking at $t=2$. 

In Figure \ref{fig:ex_1_multipeakon} the numerical solutions $(u_{\Delta x}, F_{\Delta x})$ with $\Delta x = 2.35 \cdot 10^{-1}$ and $\Delta x = 8.02 \cdot 10^{-3}$ are compared to the exact solution $(u, F)$ at $t=0$, at the wave breaking time $t=2$, and at $t=4$. 

Figure \ref{fig:ex1_convergence_order} displays the numerically computed errors. In this case the numerically computed convergence order, which is $\geq 0.99$ for both $u$ and $F$, is greater than the general order of the initial approximation error, see Proposition \ref{prop:prop_projection}. We also computed the errors using \eqref{eq:quantity_rhs}, which led to an order of $0.56$. Thus, the estimate in \eqref{eq:convergence_u} is not optimal for this example.

\begin{figure}
\centering
  \captionsetup{width=0.88\textwidth}
\begin{subfigure}{.99\textwidth}
  \centering
\includegraphics[width=0.95\textwidth]{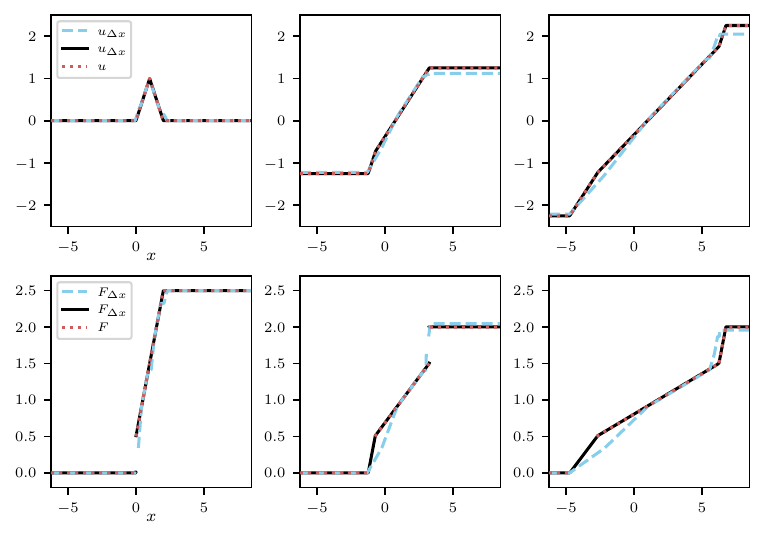}
\end{subfigure}
\caption{Time evolution of $u$ (top row, dashed red line) and $F$ (bottom row, dashed red line), and $u_{\Delta x}$ (top row) and $F_{\Delta x}$ (bottom row) for $\Delta x = 2.35 \cdot 10^{-1}$ (blue dashed line) and $\Delta x = 8.02 \cdot 10^{-3}$ (black solid line) in Example \ref{ex:numerical_multipeakon}. The times from left to right are $t=0$, $2$, $4$, and $\alpha=\frac{1}{2}$.}
\label{fig:ex_1_multipeakon}
\end{figure}

\begin{figure}
    \centering
    \captionsetup{width=0.88\textwidth}
\begin{subfigure}{.99\textwidth}
  \centering
\includegraphics[width=0.83\textwidth]{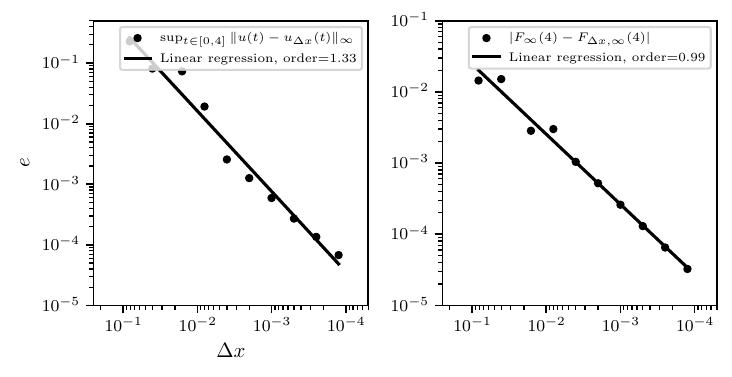}
\end{subfigure}
    \caption{ The errors $\sup_{t \in [0,4]}\|u(t) - u_{\Delta x}(t) \|_{\infty}$ (left) and $|F_{\infty}(T) - F_{\Delta x, \infty}(T)|$ at $T=4$ (right) plotted against the mesh size $\Delta x$ in Example \ref{ex:numerical_multipeakon}.}
\label{fig:ex1_convergence_order}
\end{figure} 

\end{example}

\begin{example}[Cusp initial data]\label{ex:numerical_cusp}
Let $\alpha=\frac25$ and consider the initial data
 \begin{align*}
        u_0(x) &= \begin{cases}
        1, & x < -1, \\
        |x|^{\frac{2}{3}}, & -1 \leq x \leq 1, \\
        1, & 1 < x,
        \end{cases} \\
        d\mu_0 &= d\nu_0 = u_{0, x}^2dx. 
    \end{align*}
The wave breaking times accumulate. For each $t \in [0, 3]$ an infinitesimal amount of energy concentrates.  Furthermore, $u_{0,x}$ is not in $L^{\infty}(\R)$. 

In Figure \ref{fig:ex_2_cusp_solution} $(u_{\Delta x}, F_{\Delta x})$ with $\Delta x = 1.76 \cdot 10^{-1}$ and $\Delta x = 6.01 \cdot 10^{-3}$, and the exact solution, are displayed at $t=0$, $t=\frac{3}{2}$, and $t=3$. Figure \ref{fig:ex2_convergence_order} shows the errors as we refine the mesh. We also computed the errors using \eqref{eq:quantity_rhs}, which led to an order of $0.24$. Thus, in this case the estimate in \eqref{eq:convergence_u} is optimal.

\begin{figure}
\centering
  \captionsetup{width=0.88\textwidth}
\begin{subfigure}{.99\textwidth}
  \centering
\includegraphics[width=0.95\textwidth]{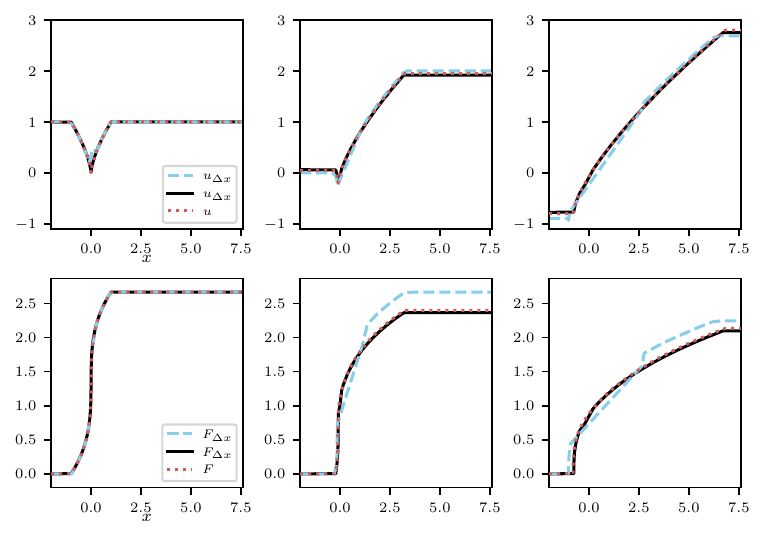}
\end{subfigure}
     \caption{Time evolution of $u$ (top row, dashed red line) and $F$ (bottom row, dashed red line), and $u_{\Delta x}$ (top row) and $F_{\Delta x}$ (bottom row) for $\Delta x = 1.76 \cdot 10^{-1}$ (blue dashed line) and $\Delta x = 6.01 \cdot 10^{-3}$ (black solid line) in Example \ref{ex:numerical_cusp}. The times from left to right are $t=0, \frac{3}{2},3$, and $\alpha=\frac{2}{5}$.}
\label{fig:ex_2_cusp_solution}
\end{figure}

\begin{figure}
    \centering
      \captionsetup{width=0.88\textwidth}
    \includegraphics[width=0.83\textwidth]{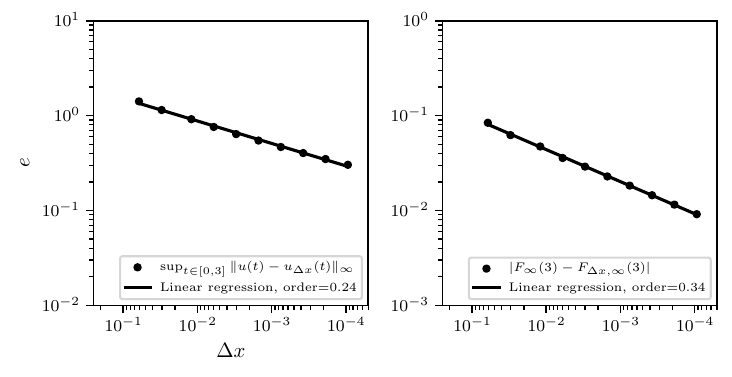}
    \caption{ The errors $\sup_{t \in [0,3]}\|u(t) - u_{\Delta x}(t) \|_{\infty}$ (left) and $|F_{\infty}(T) - F_{\Delta x, \infty}(T)|$ at $T=3$ (right) plotted against the mesh size $\Delta x$ in Example \ref{ex:numerical_cusp}.}
    \label{fig:ex2_convergence_order}
\end{figure}   
\end{example}

As can be observed in Figure \ref{fig:ex2_convergence_order} of Example \ref{ex:numerical_cusp}, the convergence order is lower than in Example \ref{ex:numerical_multipeakon}, suggesting that accumulating wave breaking times deteriorates the convergence rate of the numerical method. However, the upcoming example indicates that accumulation is of less significance than regularity of the initial data.  

\begin{example}[Cosine wave initial data]\label{ex:numerical_cosinus}
Let $\alpha=\frac35$ and consider the initial data
\begin{align*}
    u_0(x) &= \begin{cases}
    1, &x < 0, \\
    \cos(\pi x), & 0 \leq x < 4, \\
    1, & 4 \leq x, \end{cases} \\
    d\mu_0 &= d\nu_0 = u_{0, x}^2 dx = \pi^2 \sin^2(\pi x) \chi_{[0, 4)}(x)dx.
\end{align*}
Note that in contrast to the previous example, the derivative $u_{0, x}$ is Lipschitz continuous.
Moreover, for each $t \in \left(\tfrac{2}{\pi}, \infty\right)$, wave breaking occurs at four distinct isolated points both in Eulerian and Lagrangian coordinates, and happens continuously in the time interval $[\frac{2}{\pi}, \infty)$. One can compute the solution in Lagrangian coordinates exactly, as well as the total energy $F_{\infty}(t)$ for all $t \geq 0$. The Lagrangian solution is then numerically mapped into Eulerian coordinates and compared with two different numerical approximations at the times $t=0$, $\frac{2}{\pi}$, and $\frac{4}{\pi}$ in Figure \ref{fig:ex_3_sinus_simulations}.

Figure \ref{fig:ex3_convergence_order} shows the approximation errors. We observe that  the convergence order is higher than in Example \ref{ex:numerical_cusp}, although the wave breaking times still accumulate. 

\begin{figure}
\centering
 \captionsetup{width=0.88\textwidth}
\begin{subfigure}{.99\textwidth}
  \centering
\includegraphics[width=0.95\textwidth]{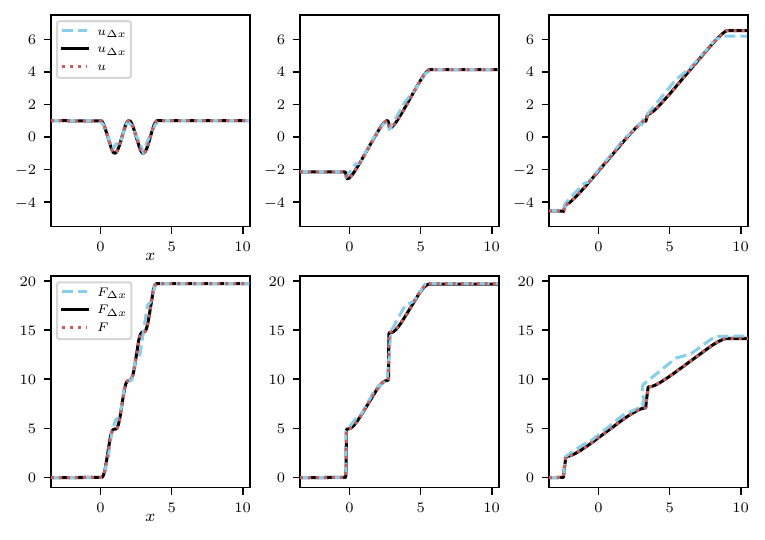}
\end{subfigure}
    \caption{Time evolution of $u$ (top row, dashed red line) and $F$ (bottom row, dashed red line), and $u_{\Delta x}$ (top row) and $F_{\Delta x}$ (bottom row) for $\Delta x = 9.59 \cdot 10^{-2}$ (blue dashed line) and $\Delta x = 5.88 \cdot 10^{-4}$ (black solid line) in Example \ref{ex:numerical_cosinus}. The times from left to right are $t=0, \frac{2}{\pi}, \frac{4}{\pi}$, and $\alpha=\frac{3}{5}$.}
    \label{fig:ex_3_sinus_simulations}
\end{figure}

\begin{figure}
    \centering
      \captionsetup{width=0.88\textwidth}
\includegraphics[width=0.83\textwidth]{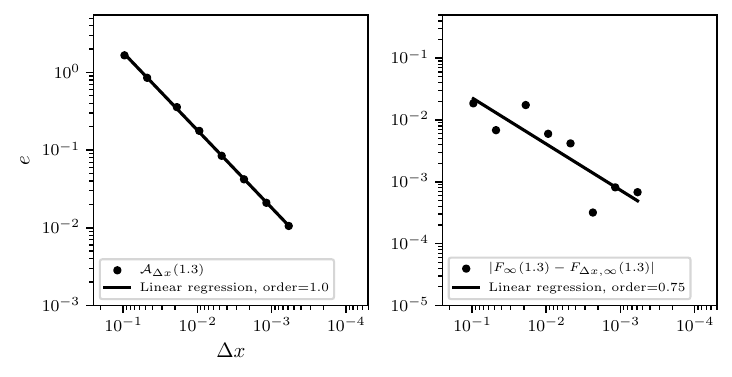}
    \caption{ The errors $\mathcal{A}_{\Delta x}(1.3)$, see \eqref{eq:quantity_rhs}, (left) and $|F_{\infty}(T) - F_{\Delta x, \infty}(T)|$ at $T=1.3$ (right) plotted against the mesh size $\Delta x$ in Example \ref{ex:numerical_cosinus}.}
    \label{fig:ex3_convergence_order}
\end{figure}
\end{example}

\bibliographystyle{plain}

\appendix

\section{Exact solutions to the examples in Section~\ref{sec:Numex}}\label{sec:Appendix}

For completeness, we present the exact solutions, in both Eulerian and Lagrangian coordinates, to Example~\ref{ex:numerical_multipeakon} and \ref{ex:numerical_cusp}. 

\subsection{Multipeakon initial data} \label{subsec:multipeakon}
Let $\alpha\in [0,1]$, $\beta\in \mathbb{R}^+$, and consider the initial data 
\begin{align*}
    u_0(x) &= \begin{cases}
    0, & x \leq 0, \\
    x, & 0 < x \leq 1, \\
    2-x, & 1 < x \leq 2, \\
    0, & 2 < x, \end{cases} \\
    d\mu_0 &= d\nu_0 =  \beta d\delta_{0} + u_{0,x}^2dx.
\end{align*}

We map the initial data to Lagrangian coordinates by applying the mapping $L$ from Definition~\ref{def:MappingL} yielding 
\begin{align*}
    y_0(\xi) &= \begin{cases}
    \xi, & \xi \leq 0, \\
    0, & 0 < \xi \leq \beta , \\
    \frac{1}{2}\xi -\frac12 \beta, & \beta < \xi \leq 4+\beta, \\
    \xi - (2+\beta), & 4+\beta \leq \xi, \end{cases} \\
    U_0(\xi) &= \begin{cases}
    0, & \xi \leq \beta , \\
    \frac{1}{2}\xi -\frac12\beta,  & \beta < \xi \leq 2+\beta, \\
    -\frac12 \xi+\frac12(4+\beta), & 2+\beta < \xi \leq 4+\beta, \\ 
    0, & 4+\beta < \xi, \end{cases} \\
    V_0(\xi) &= H_0(\xi) = \begin{cases}
    0, & \xi \leq 0, \\
    \xi, & 0 < \xi \leq \beta , \\
    \frac{1}{2}\xi+ \frac12\beta, & \beta  < \xi \leq 4+\beta, \\
    2+\beta , & 4+\beta < \xi. \end{cases} 
\end{align*}
Using \eqref{eq:waveBreakingFunction_Lagrangian}, we find
\begin{align*}
    \tau(\xi) &= \begin{cases}
    0, & \xi \in (0, \beta), \\ 
    2, & \xi \in (2+\beta, 4+\beta), \\ 
    \infty, & \text{otherwise}. \end{cases}
\end{align*}
Thus, for $t\in (0, 2)$ no wave breaking occurs, and the solution for $t<2$ is found by solving \eqref{eq:integrated_ODE} with $V(t, \xi) = V_0(\xi)$. In particular, 
\begin{align*}
    y(t, \xi) &= \begin{cases}
    \xi - \frac{1}{8} \left( 2+\beta\right) t^2, & \xi \leq 0, \\
    \frac{1}{4} \xi t^2 - \frac{1}{8} \left (2+\beta\right)t^2, & 0 < \xi \leq \beta, \\
    \frac{1}{8}\left(t + 2\right)^2\xi -\frac12\beta-\frac12\beta t-\frac14 t^2 , & \beta < \xi \leq 2+\beta, \\
    \frac{1}{8} \left(t-2 \right)^2\xi -\frac12 \beta+\frac12\left(4+\beta\right) t-\frac14 t^2, & 2+\beta < \xi \leq 4+\beta, \\
    \xi -(2+\beta)+\frac18 \left(2+\beta\right) t^2, & 4+\beta \leq \xi, \end{cases} \\
    U(t, \xi) &= \begin{cases}
    - \frac{1}{4}\left(2+\beta \right)t, & \xi \leq 0, \\
    \frac{1}{2}\xi t - \frac{1}{4} \left(2+\beta\right)t, & 0 < \xi \leq \beta, \\
    \frac{1}{4}\left(t + 2 \right)\xi -\frac12 \beta-\frac12 t, & \beta < \xi \leq 2+\beta, \\
    \frac{1}{4} \left(t-2 \right)\xi+\frac12 \left(4+\beta\right)-\frac12 t,& 2+\beta < \xi \leq 4+\beta, \\
    \frac{1}{4} \left (2+\beta \right) t, & 4+\beta < \xi, \end{cases} \\
    V(t, \xi) &= H(t, \xi) = V_0(\xi). 
\end{align*}

By Definition~\ref{def:MappingM}, we deduce that for $t\in [0,2)$  
\begin{align*}
    u(t, x) &= \begin{cases}
    - \frac{1}{4} \left(2+\beta \right)t, & x \leq - \frac{1}{8}\left(2+\beta \right) t^2, \\
    \frac{2x}{t}, & - \frac{1}{8}\left(2+\beta \right) t^2 < x \leq -\frac{1}{8}\left(2-\beta \right)t^2, \\
    \frac{4x - \left(2-\beta \right)t}{2 \left(t + 2 \right)}, & -\frac{1}{8}\left(2-\beta\right)t^2 < x \leq 1+t+\frac18 \beta t^2, \\
    \frac{4x-8 - (2+\beta)t}{2(t-2)}, & 1+t+\frac18\beta t^2 < x \leq 2+\frac18 \left(2+\beta\right)t^2, \\
    \frac{1}{4}(2+\beta)t, & 2+\frac{1}{8}\left(2+\beta \right)t^2 < x, \end{cases}
\end{align*}
and 
\begin{align*}
    F(t, x) &=  \begin{cases}
    0, & x \leq - \frac{1}{8}\left(2+\beta \right) t, \\
    \frac{8x+\left(2+\beta\right) t^2}{2t^2}, & - \frac{1}{8}\left(2+\beta \right) t^2 < x \leq -\frac{1}{8}\left(2-\beta \right)t^2, \\
    \frac{8x+8\beta +8\beta t+(2+\beta)t^2}{2 \left(t+2 \right)^2}, & -\frac{1}{8}\left(2-\beta \right)t^2 < x\leq 1+t+\frac18 \beta t^2, \\
    \frac{8x + 8\beta-8\left(2+\beta\right) t+\left(2+\beta\right)t^2}{2(t-2)^2}, & 1+t+\frac18 \beta t^2 < x \leq 2+\frac18\left(2+\beta\right) t^2, \\
    2+\beta, &2+ \frac{1}{8}\left(2+\beta \right)t^2  < x.
    \end{cases}
\end{align*}
Note that $F$ and $G$ coincide prior to wave breaking which occurs at $t=2$. At $t=2$, all the energy confined within the interval $(2+\beta, 4+\beta)$ in Lagrangian coordinates concentrates at the point $x = \frac{1}{2}(6+\beta)$. 

For $t \geq 2$ the Lagrangian energy takes the form
\begin{align*}
    V(t, \xi) &= \begin{cases}
      0,  &\xi \leq 0, \\
      \xi, & 0 < \xi \leq \beta, \\
      \frac{1}{2} \xi + \frac12\beta, & \beta < \xi \leq 2+\beta, \\
      \frac{1}{2}\left(1-\alpha \right)\xi + \frac{1}{2}\left(\beta+2\alpha+\alpha\beta \right), & 2+\beta < \xi \leq  4+\beta, \\
      2+\beta-\alpha, & 4+\beta < \xi, \end{cases}
\end{align*}
and by using
\begin{align*}
    y(t, \xi) &= y(2, \xi) + \int_2^tU(s, \xi)ds, \\ 
    U(t, \xi) &= U(2, \xi) + \frac{1}{2} \int_2^t \left(V(s, \xi) - \frac{1}{2}V_{\infty}(s) \right)ds, 
\end{align*}
we find
\begin{align*}
     y(t, \xi) &= \begin{cases}
    \xi +\frac12\alpha-\frac12 \alpha t-\frac18 \left(2+\beta-\alpha\right)t^2, & \xi \leq 0, \\
    \frac{1}{4}t^2 \xi +\frac12 \alpha-\frac12 \alpha t-\frac18 \left(2+\beta-\alpha\right) t^2, & 0 < \xi \leq \beta, \\
    \frac{1}{8} \left(t+2 \right)^2 \xi -\frac12 (\beta-\alpha)-\frac12 (\beta+\alpha) t-\frac18 (2-\alpha)t^2, & \beta < x \leq 2+\beta, \\
    \frac{1}{8}\left(1 - \alpha \right) \left(t-2 \right)^2 \xi -\frac12 (\beta-3\alpha-\alpha\beta)\\ \quad +\frac12 (4+\beta-3\alpha-\alpha\beta)t 
    -\frac18(2-3\alpha-\alpha\beta) t^2, & 2+\beta < \xi \leq 4+\beta, \\
    \xi -\frac12 (4+2\beta+\alpha)+\frac12 \alpha t+\frac18 \left(2+\beta-\alpha\right)t^2, & 4+\beta < \xi,
    \end{cases} \\
    U(t, \xi) &= \begin{cases}
     - \frac{1}{2}\alpha-\frac14\left(2+\beta-\alpha\right)t, & \xi \leq 0, \\
    \frac{1}{2}\xi t  - \frac{1}{2}\alpha-\frac14 \left(2+\beta-\alpha\right) t, & 0 < \xi \leq \beta, \\
    \frac{1}{4}\left(t+2 \right)\xi -\frac12 (\beta+\alpha)-\frac14 (2-\alpha)t, & \beta < \xi \leq 2+\beta, \\
    \frac{1}{4}\left(1-\alpha \right)\left(t-2 \right)\xi +\frac12(4+\beta-3\alpha-\alpha\beta)\\ \quad -\frac14(2-3\alpha-\alpha\beta)t, & 2+\beta < \xi \leq 4+\beta, \\
     \frac{1}{2}\alpha+\frac14 \left(2+\beta-\alpha\right)t, & 4+\beta < \xi.\end{cases}
\end{align*}

Thus, we find for $t \geq 2$ the following Eulerian solution  
\begin{align*}
    u(t, x) &= \begin{cases}
   -\frac12 \alpha - \frac{1}{4} \left(2+\beta-\alpha \right)t , & x \leq x_1(t), \\
    \frac{4x  -2\alpha+\alpha t}{2t}, & x_1(t) < x \leq x_2(t) \\
    \frac{4x -4\alpha-(2-\beta-2\alpha)t}{2(t+2)}, & x_2(t) < x \leq x_3(t) , \\
    \frac{4x - 8-(2+\beta)t}{2(t-2)}, & x_3(t) < x \leq x_4(t) , \\
    \frac{1}{2}\alpha+\frac14 (2+\beta-\alpha)t, & x_4(t) < x,
    \end{cases} \\
    F(t, x) &= \begin{cases}
    0, & x \leq x_1(t), \\
    \frac{8x-4\alpha+4\alpha t+(2+\beta-\alpha) t^2}{2t^2}, & x_1(t) < x \leq x_2(t), \\
    \frac{8x +4(2\beta-\alpha)+4(2\beta+\alpha)t+(2+\beta-\alpha)t^2}{2\left(t+2 \right)^2}, & x_2(t) < x \leq x_3(t), \\
    \frac{8x +4(2\beta-\alpha)-4(4+2\beta-\alpha) t+(2+\beta-\alpha)t^2}{2 \left(t-2 \right)^2}, & x_3(t) < x \leq x_4(t), \\
    2+\beta-\alpha, & x_4(t) < x,
    \end{cases}
\end{align*}
where we introduced
\begin{align*}
    x_1(t) &= \frac12 \alpha-\frac12 \alpha t -\frac18 (2+\beta-\alpha) t^2, \\
    x_2(t) &= \frac12 \alpha-\frac12 \alpha t-\frac18 (2-\beta-\alpha)t^2, \\
    x_3(t) &= \frac12(2+\alpha) +\frac12 (2-\alpha)t+\frac18 (\beta+\alpha) t^2, \\ 
    x_4(t) &= \frac12 (4-\alpha) +\frac12 \alpha t+\frac18 (2+\beta-\alpha)t^2. 
\end{align*}

\subsection{Cusp initial data}\label{subsec:cusp}
Consider the initial data given in Example \ref{ex:numerical_cusp}. We observe that $u_{0, x}(x) \rightarrow - \infty$ as $x \uparrow 0$. Thus wave breaking occurs at $x=0$ initially. However, there is only an infinitesimal amount of energy concentrating. 

Applying $L$ from Definition~\ref{def:MappingL}, we find the following Lagrangian initial data 
    \begin{align*}
        y_0(\xi) &= \begin{cases}
        \xi, & \xi \leq - 1, \\
        \bar{y}(\xi), & -1 < \xi \leq \frac{11}{3}, \\
        \xi - \frac{8}{3}, & \frac{11}{3} < \xi, \end{cases} \\
        U_0(\xi) &= \begin{cases}
        1, & \xi \leq -1, \\
        \bigl|\bar{y}(\xi) \bigr|^{\frac{2}{3}}, & -1 < \xi \leq \frac{11}{3}, \\
        1, & \frac{11}{3} < \xi, \end{cases} \\
        V_0(\xi) &= \begin{cases}
        0, & \xi \leq -1, \\
        \frac{4}{3}\left(\sgn(\bar{y}(\xi)) \bigl|\bar{y}(\xi) \bigr|^{\frac{1}{3}} +1 \right), & -1 < \xi \leq  \frac{11}{3}, \\
        \frac{8}{3}, & \frac{11}{3} < \xi, \end{cases} 
    \end{align*}
 where $\bar{y}(\xi)$ is implicitly given by 
 \begin{equation}\label{eq:implby}
 \bar y(\xi) + \frac{4}{3}\left(1+\sgn(\bar y(\xi))\bigl |\bar y(\xi) \bigr|^{\frac{1}{3}}\right)=\xi.
 \end{equation}
  In particular, $\bar y$ is continuous and strictly increasing on $[-1, \frac{11}3]$ and 
 \begin{equation*}
 \bar y(-1)=-1,\quad  \bar y\left(\frac43\right)=0,\quad  \text{ and  } \quad \bar y\left(\frac{11}{3}\right)=1.
 \end{equation*}

The wave breaking function, given by  \eqref{eq:waveBreakingFunction_Lagrangian}, reads 
    \begin{align}
        \tau(\xi) &= - 2\frac{y_{ 0, \xi}}{U_{0, \xi}}(\xi) =\begin{cases}
        3 \bigl| \bar{y}(\xi) \bigr|^{\frac{1}{3}}, & \xi \in (-1, \frac43], \\
        \infty, & \text{otherwise.}
        \end{cases}
\label{eq:example_wave_breaking_function}
    \end{align}
Note that to each  $t\in (0,3]$, there exists exactly one $\xi(t)\in (-1, \frac43]$ such that $\tau(\xi)=t$, since $\bar{y}$ is strictly increasing on $[-1, \frac43]$. Furthermore, setting $t=\tau(\xi)$ and rewriting \eqref{eq:example_wave_breaking_function}, we find 
\begin{align}
        y_0\left(\xi(t)\right) &= \bar{y}\left(\xi\left(t \right)\right) = - \left(\frac{t}{3}\right)^3.
        \label{eq:example_wave_breaking_characteristics}
    \end{align}
Thus, combining \eqref{eq:implby} and \eqref{eq:example_wave_breaking_characteristics}, the wave breaking curve $\xi(t)$ in Lagrangian coordinates is given by 
\begin{align}
    \xi(t) &= -\left(\frac{t}{3}\right)^3 + \frac{4}{3}\left(1 - \frac{t}{3} \right).
    \label{eq:example_wave_breaking_curve}
\end{align}

We now compute $V(t, \xi)$ for $0\leq t \leq3$. The energy is unchanged for the characteristics labeled by $\xi < \xi(t)$, as along these no wave breaking has occurred yet. In particular, this means that $V(t, \xi) = V_0(\xi)$ for $\xi \in( -\infty,\xi(t))$. 

For $\xi\in[\xi(t), \frac43]$, we find, using \eqref{eq:integrated_ODE3}, 
\begin{align*}
        V(t, \xi) & = V_0(\xi) - \alpha \int_{\xi(t)}^{\xi}V_{0, \xi}(\eta)d\eta \nonumber \\ &= V_0(\xi) - \alpha \int_{\xi(t)}^{\xi} \frac{4}{9}\bigl |\bar{y}(\eta)\bigr|^{-\frac{2}{3}}\bar{y}_{\xi}(\eta)d\eta   
     \nonumber \\ &= \frac{4}{3}\bigg( - \big(1 - \alpha \big)\bigl |\bar{y}(\xi)\bigr |^{\frac{1}{3}} +1- \frac{1}{3}\alpha t  \bigg) .
\end{align*}

Since $V(t, \frac43)=V_0(\frac43)-\frac49 \alpha t$ and no wave breaking occurs for $\xi \in (\frac43, \infty)$, one has $V(t,\xi)=V_0(\xi)-\frac49\alpha t$ for $\xi\in(\frac43,\infty)$. To summarize, for $0\leq t\leq3$, 
    \begin{align}
        V(t, \xi) &= \begin{cases}
        0, & \xi \leq -1 \\
        \frac{4}{3} \big(-|\bar{y}(\xi)|^{\frac{1}{3}} + 1 \big), & -1 < \xi < \xi(t), \\
         \frac{4}{3} \big( - (1-\alpha) \bigl|\bar{y}(\xi) \bigr |^{\frac{1}{3}} +1- \frac{1}{3}\alpha t \big) , &  \xi(t) \leq \xi \leq \frac{4}{3}, \\
        \frac{4}{3} \big(|\bar{y}(\xi)|^{\frac{1}{3}} +1- \frac{1}{3}\alpha t  \big), & \frac{4}{3} < \xi \leq \frac{11}{3}, \\
        \frac{4}{3}\big(2 - \frac{1}{3}\alpha t \big), & \frac{11}{3} < \xi. 
        \end{cases}
        \label{eq:example_V}
    \end{align}
    
 Next we compute $U(t,\xi)$ and $y(t,\xi)$ for $t\in [0,3]$, using \eqref{eq:integrated_ODE1} and \eqref{eq:integrated_ODE2}. We only show the details for $\xi\in (-1, \frac{4}3]$. For $\xi\in(-1, \xi(t)]$, combining \eqref{eq:integrated_ODE} and \eqref{eq:example_V} yields
   \begin{align}
     U(t, \xi) &= \left(- \bigl |\bar{y}(\xi) \bigr |^{\frac{1}{3}} + \frac{t}{3} \right)^2 - \frac{1}{18} \left(2 - \alpha \right)t^2,
        \label{eq:U_left_breaking_curve}
 \end{align}
and
\begin{align}
    y(t, \xi)
        &= \left( -\bigl|\bar{y}(\xi) \bigr|^{\frac{1}{3}} + \frac{t}{3} \right)^3 - \frac{1}{54} \big(2 - \alpha)t^3. \label{eq:y_left_breaking_curve}
\end{align} 

For $\xi \in [\xi(t), \frac43]$ wave breaking has already occured. Thus, instead of integrating \eqref{eq:integrated_ODE} from $0$ to $t$, we integrate from $\tau(\xi)$ to $t$. Furthermore, as $U$ and $y$ are continuous in time, we get from \eqref{eq:U_left_breaking_curve} and \eqref{eq:y_left_breaking_curve} that 
\begin{equation*}
U(\tau(\xi), \xi) = - \frac{1}{18}(2 - \alpha)\tau^2(\xi),\quad \text{ and }\quad y(\tau(\xi), \xi) = - \frac{1}{54}(2 - \alpha)\tau^3(\xi).
\end{equation*}
Consequently, for $\xi \in [\xi(t), \frac{4}{3}]$, we find
\begin{align*}
    U(t, \xi) & = - \frac{1}{18}(2-\alpha)\tau^2(\xi) - \frac{2}{3} \int_{\tau(\xi)}^t \left( (1 - \alpha) \left |\bar{y}(\xi) \right|^{\frac{1}{3}} + \frac{1}{6}\alpha s \right) ds\\
        &= \frac{1}{9}(1 - \alpha) \left(\tau(\xi) - t \right)^2 - \frac{1}{18}(2 - \alpha)t^2,
\end{align*}
where we used that $\left| \bar{y}(\xi) \right|^{\frac{1}{3}} = \frac{1}{3}\tau(\xi)$ by \eqref{eq:example_wave_breaking_function}. For $y(t, \xi)$, we obtain
\begin{align*}
    y(t, \xi)  &=  -\frac{1}{54}(2-\alpha)\tau^3(\xi) + \int_{\tau(\xi)}^t \left(\frac{1}{9}(1 - \alpha) \left(\tau(\xi) - s\right)^2 - \frac{1}{18}(2 - \alpha)s^2\right)ds \\
        &= -\frac{1}{54}\left(2 - \alpha \right)t^3 - \frac{1}{27}(1 - \alpha) \left(\tau(\xi) - t\right)^3. 
\end{align*}

Following the same lines for all remaining $\xi$, yields for $t \in [0, 3)$,  
 \begin{subequations}
     \begin{align}
        U(t, \xi) &= \begin{cases}
        1-\frac23 t+\frac{1}{18}\alpha t^2, & \xi \leq -1, \\
        \left(-\bigl|\bar{y}(\xi) \bigr|^{\frac{1}{3}} + \frac{t}{3} \right)^2 - \frac{1}{18} \left(2 - \alpha \right)t^2, & -1 < \xi < \xi(t), \\
        \frac{1}{9}(1 - \alpha) \left(3 \bigl |\bar{y}(\xi) \bigr|^{\frac{1}{3}} - t \right)^2 - \frac{1}{18}(2 - \alpha)t^2, & \xi(t) \leq \xi \leq \frac{4}{3}, \\
        \left(\bigl |\bar{y}(\xi) \bigr|^{\frac{1}{3}} + \frac{t}{3} \right)^2 - \frac{1}{18} \left(2 + \alpha \right)t^2, & \frac{4}{3} < \xi \leq \frac{11}{3}, \\
        1+\frac23 t- \frac{1}{18} \alpha t^2 , & \frac{11}{3} < \xi, \end{cases} \label{eq:example_U}\\ 
        y(t, \xi) &= \begin{cases}
        \xi +t-\frac13 t^2+ \frac{1}{54}\alpha t^3 , & \xi \leq -1, \\
        \left(-\bigl|\bar{y}(\xi) \bigr|^{\frac{1}{3}} + \frac{t}{3} \right)^3 - \frac{1}{54} \left(2 - \alpha \right)t^3, & -1 < \xi < \xi(t), \\
        - \frac{1}{27} \left(1 - \alpha \right) \left(3 \bigl |\bar{y}(\xi) \bigr|^{\frac{1}{3}} - t \right)^3 - \frac{1}{54}(2 - \alpha)t^3, & \xi(t) \leq \xi \leq \frac{4}{3}, \\
        \left(\bigl| \bar{y}(\xi) \bigr|^{\frac{1}{3}} + \frac{t}{3} \right)^3 - \frac{1}{54}(2 + \alpha)t^3, & \frac{4}{3} < \xi \leq \frac{11}{3}, \\
        \xi -\frac83+t+\frac13 t^2- \frac{1}{54}\alpha t^3 , & \frac{11}{3} < \xi, \label{eq:example_y} \end{cases}
    \end{align}
 \end{subequations}
 where $\xi(t)$ is given by \eqref{eq:example_wave_breaking_curve}. 

A closer look at the above expressions reveals that it is possible to compute $u(t,x)$ and $F(t,x)$ with the help of Definition~\ref{def:MappingM}. We will only detail this step for $\xi\in (-1, \frac43]$. 

For $\xi\in (-1, \xi(t)]$, we find by rearranging, \eqref{eq:example_y} 
\begin{equation*}
   -\left |\bar{y}(\xi) \right |^{\frac{1}{3}} + \frac{t}{3}  = \left( y(t,\xi) + \frac{1}{54}(2 - \alpha)t^3 \right)^{\frac{1}{3}},
\end{equation*} 
which implies, recalling  \eqref{eq:example_U} and \eqref{eq:example_V},
\begin{align*}
u(t,y(t,\xi))&=U(t,\xi)= \left(y(t,\xi) + \frac{1}{54}(2-\alpha)t^3 \right)^{\frac{2}{3}} - \frac{1}{18}(2-\alpha)t^2,\\
F(t,y(t,\xi))&=V(t,\xi)=\frac{4}{3} \left( \left(y(t,\xi) + \frac{1}{54} (2 - \alpha)t^3 \right)^{\frac{1}{3}} +1- \frac{t}{3} \right).
\end{align*}
Thus,
\begin{align*}
        u(t, x) &= \left(x + \frac{1}{54}(2-\alpha)t^3 \right)^{\frac{2}{3}} - \frac{1}{18}(2-\alpha)t^2, \\
        F(t, x) &= \frac{4}{3} \left( \left(x + \frac{1}{54} (2 - \alpha)t^3 \right)^{\frac{1}{3}} +1- \frac{t}{3}  \right),
    \end{align*}
for $x\in (y(t, -1), y(t, \xi(t))] = (-1+t-\frac13t^2+\frac{1}{54}\alpha t^3, - \frac{1}{54}(2 - \alpha)t^3]$, where we inserted \eqref{eq:example_wave_breaking_characteristics} into \eqref{eq:example_y} to determine the right endpoint of this interval. 

For $\xi \in (\xi(t), \frac{4}{3}]$, we obtain, using \eqref{eq:example_y},
\begin{align*}
     (\tau(\xi) - t) &= -3\left( \frac{1}{\left(1-\alpha \right)} \left(y(t,\xi) + \frac{1}{54}(2 - \alpha)t^3 \right) \right)^{\frac{1}{3}},
\end{align*}
which yields, following the same lines as before,
\begin{align*}
     u(t, x) &= \left(1- \alpha \right)^{\frac{1}{3}}\left(x + \frac{1}{54}(2-\alpha)t^3 \right)^{\frac{2}{3}} - \frac{1}{18}(2 - \alpha)t^2, \\
    F(t, x) &= \frac{4}{3} \left(\big(1 - \alpha)^{\frac{2}{3}} \left(x + \frac{1}{54}(2 - \alpha)t^3  \right)^{\frac{1}{3}}+1 - \frac{t}{3}  \right),
\end{align*}
for $x\in (y(t, \xi(t)), y(t, \frac{4}{3})] = (- \frac{1}{54}(2 - \alpha)t^3, - \frac{1}{54}\alpha t^3]$. 

Considering all $\xi\in \mathbb{R}$, the exact solution in Eulerian coordinates for $t \in [0, 3)$ is given by 
\begin{align*}
        u(t, x) &= \begin{cases}
        1 - \frac{2}{3}t + \frac{1}{18}\alpha t^2, & x \leq x_1(t), \\
        \left(x + \frac{1}{54}(2 - \alpha)t^3 \right)^{\frac{2}{3}} - \frac{1}{18}(2 - \alpha)t^2, & x_1(t) < x \leq x_2(t), \\
       \left(1- \alpha \right)^{\frac{1}{3}}\left(x + \frac{1}{54}(2-\alpha)t^3 \right)^{\frac{2}{3}} - \frac{1}{18}(2 - \alpha)t^2 , &  x_2(t) < x \leq -\frac{1}{54}\alpha t^3, \\
        \left(x + \frac{1}{54}(2 + \alpha)t^3 \right)^{\frac{2}{3}} - \frac{1}{18}(2 + \alpha)t^2, & -\frac{1}{54}\alpha t^3 < x \leq x_3(t), \\
        1 + \frac{2}{3}t - \frac{1}{18}\alpha t^2, & x_3(t) < x, \end{cases} \\
        F(t, x) &= \begin{cases}
        0, & x \leq x_1(t), \\
        \frac{4}{3} \left(  \left(x + \frac{1}{54}(2 - \alpha)t^3 \right)^{\frac{1}{3}} +1-\frac{t}{3}\right), & x_1(t) < x \leq x_2(t), \\
        \frac{4}{3} \left( (1 - \alpha)^{\frac{2}{3}} \left(x + \frac{1}{54}(2 - \alpha)t^3 \right)^{\frac{1}{3}} +1-\frac{t}{3}\right), & x_2(t) < x \leq -\frac{1}{54}\alpha t^3, \\
        \frac{4}{3} \left(  \left(x + \frac{1}{54}(2 + \alpha)t^3 \right)^{\frac{1}{3}} +1-(1+\alpha)\frac{t}{3}\right), & -\frac{1}{54}\alpha t^3 < x \leq x_3(t), \\
        \frac{8}{3} - \frac{4}{9} \alpha t, &  x_3(t)< x, \end{cases} 
    \end{align*}
    where we introduced 
    \begin{align*}
        x_1(t) &=  -1 + t - \frac{1}{3}t^2 + \frac{1}{54}\alpha t^3, \\
        x_2(t) &= -\frac{1}{54}(2-\alpha)t^3, \\
        x_3(t) &= 1 + t + \frac{1}{3}t^2 - \frac{1}{54}\alpha t^3.
    \end{align*}

It is possible to compute the solution for $t> 3$ as well, but this is left to the interested reader. Since no more wave breaking happens for $t>3$, $V(t,\xi)=V(3,\xi)$ for all $\xi\in \mathbb{R}$. Furthermore, $y(3,\xi)$ and $U(3,\xi)$ can be computed using \eqref{eq:example_y} and \eqref{eq:example_U}. It then remains to solve \eqref{eq:integrated_ODE}.

\subsection{Cosine initial data}\label{subsec:cosinus}
Consider the initial data in Example \ref{ex:numerical_cosinus}. We map the data to Lagrangian coordinates by applying $L$ from Definition~\ref{def:MappingL}, resulting in 
\begin{align*}
    y_0(\xi) &= \begin{cases}
        \xi, & \xi \leq 0, \\ 
        \bar{y}(\xi), & 0 < \xi \leq 4 + 2\pi^2, \\ 
        \xi - 2\pi^2, & 4 + 2\pi^2 < \xi, 
    \end{cases} \\ 
    U_0(\xi) &= \begin{cases}
        1, & \xi \leq 0, \\
        \cos(\pi \bar{y}(\xi)), & 0 < \xi \leq 4 + 2\pi^2, \\
        1, & 4 + 2\pi^2 < \xi,  
    \end{cases} \\
    V_0(\xi) &= H_0(\xi) = \begin{cases}
        0, & \xi \leq 0, \\
        \frac{\pi}{4}(2\pi \bar{y}(\xi) - \sin(2\pi \bar{y}(\xi))) & 0 < \xi \leq 4 + 2\pi^2, \\
        2\pi^2, & 4 + 2\pi^2 < \xi, 
    \end{cases}
\end{align*}
where $\bar{y}(\xi)$ is implicitly defined by 
\begin{equation}
    \frac{1}{2} (2 + \pi^2)\bar{y}(\xi) - \frac{\pi}{4}\sin(2\pi \bar{y}(\xi)) = \xi. 
    \label{eq:cos_implicit}
\end{equation}
Thus, $\bar{y}$ is continuous and strictly increasing on $[0, 4 + 2\pi^2]$ and 
\begin{equation*}
    \bar{y} \big(k + \tfrac{k}{2}\pi^2 \big) = k, \hspace{0.2cm} k \in \{0, 1, 2, 3, 4\}.
\end{equation*}
We compute the wave breaking function using \eqref{eq:waveBreakingFunction_Lagrangian}, yielding
\begin{align}
    \tau(\xi) = - 2 \frac{y_{0, \xi}}{U_{0, \xi}}(\xi) = \begin{cases}
    \frac{2}{\pi \sin \left(\pi \bar{y}(\xi) \right)}, & \xi \in \big(0, 1 + \tfrac{1}{2}\pi^2 \big) \cup \big(2 + \pi^2, 3 + \tfrac{3}{2}\pi^2 \big), \\
    \infty, & \text{otherwise}. \end{cases}
    \label{eq:cos_tau}
\end{align}
Consequently, there is no wave breaking before $t=\frac{2}{\pi}$, and we can therefore compute the solution for $t < \frac{2}{\pi}$ by solving \eqref{eq:integrated_ODE} with $V(t, \xi) = V_0(\xi)$. Hence, 
\begin{align}
    y(t, \xi) &= \begin{cases}
    \xi + t - \frac{\pi^2}{4}t^2, & \xi < 0, \\
    \bar{y}(\xi) + \cos(\pi \bar{y}(\xi)) t \\ \quad + \frac{\pi}{16} \left(2\pi \bar{y}(\xi) - 4\pi - \sin (2 \pi \bar{y}(\xi)) \right)t^2, & 0 \leq \xi < 4 + 2\pi^2, \\
    \xi - 2\pi^2 + t + \frac{\pi^2}{4}t^2, & 4 + 2\pi^2 \leq \xi, \end{cases} \nonumber 
    \\ 
    U(t, \xi) &= \begin{cases}
    1 - \frac{\pi^2}{2}t, & \xi < 0, \\
    \cos(\pi \bar{y}(\xi)) \\ \quad + \frac{\pi}{8} \left(2\pi \bar{y}(\xi) - 4\pi - \sin(2\pi \bar{y}(\xi)) \right)t, & 0 \leq \xi < 4 + 2 \pi^2, \\
    1 + \frac{\pi^2}{2}t, & 4 + 2\pi^2 \leq \xi, \end{cases} \nonumber 
    \\
    V(t, \xi) &= H(t, \xi) = V_0(\xi). \label{eq:cos_sol_before}
\end{align}

At $t= \frac{2}{\pi}$ wave breaking occurs at the two points $\xi_1 = \frac{1}{4}(2+\pi^2)$ and $\xi_2= \frac{5}{4}(2 + \pi^2)$. Thereafter, for every $t \in (\frac{2}{\pi}, \infty)$ there are precisely four values of $\xi$ such that $\tau(\xi) = t$. 

The wave breaking curve starting at $\frac{1}{4}(2+\pi^2)$, splits into two parts, one moving to the left ending up at $\xi=0$ as $t \rightarrow \infty$, and one curve moving to the right tending towards $\xi=(1 + \tfrac{1}{2}\pi^2)$ as $t \rightarrow \infty$. Setting $t = \tau(\xi)$, denoting these two curves in Lagrangian coordinates by $\xi_{1, \textrm{l}}(t)$ and $\xi_{1, \textrm{r}}(t)$, respectively, and rewriting \eqref{eq:cos_tau}, leads to 
\begin{align}
    \bar{y} \left(\xi_{1, \textrm{l}}(t) \right) &= \frac{1}{\pi}\arcsin \left( \frac{2}{\pi t} \right), \nonumber \\
    \bar{y}\left(\xi_{1, \textrm{r}}(t)\right) &= 1 - \frac{1}{\pi}\arcsin \left( \frac{2}{\pi t} \right). 
\label{eq:cos_curves_half}
\end{align}
Analogously, by symmetry, the wave breaking curve emanating from $\xi  = \frac{5}{4}(2 + \pi^2)$ splits into two parts. Denote the corresponding Lagrangian curves by $\xi_{2, \textrm{l}}(t)$ and $\xi_{2, \textrm{r}}(t)$, respectively, then
\begin{align*}
    \bar{y} \left(\xi_{2, \textrm{l}}(t)\right) &= 2 + \frac{1}{\pi} \arcsin \left(\frac{2}{\pi t} \right), \\
    \bar{y} \left(\xi_{2, \textrm{r}}(t) \right) &= 3 - \frac{1}{\pi} \arcsin \left(\frac{2}{\pi t} \right). 
\end{align*}
Combining \eqref{eq:cos_curves_half} and \eqref{eq:cos_implicit}, yields an explicit expression for $\xi_{1, \textrm{l}}(t)$, in particular, 
\begin{equation*}
    \xi_{1, \textrm{l}}(t) = \frac{1}{2\pi}\left(2 + \pi^2\right)\arcsin\left(\frac{2}{\pi t}\right) - \frac{\pi}{4}\sin \left(2\arcsin\left(\frac{2}{\pi t} \right) \right), \hspace{0.2cm} t \geq \frac{2}{\pi}.  
\end{equation*}
Explicit expressions for the other Lagrangian wave breaking curves can be obtained completely analogously.

Now we compute $V(t, \xi)$ for $t\geq \frac{2}{\pi}$. The energy is unchanged for characteristics with labels $\xi < \xi_{1, \textrm{l}}(t)$. Thus, $V(t, \xi) = V_0(\xi)$ for $\xi \in (-\infty, \xi_{1, \textrm{l}}(t))$. 

For $\xi \in [\xi_{1, \textrm{l}}(t), \xi_{1, \textrm{r}}(t)]$, using \eqref{eq:cos_curves_half}, we obtain  
\begin{align*}
    V(t, \xi) &= V_0(\xi) - \frac{\pi^2}{2}\alpha \int_{\xi_{1, \mathrm{l}}(t)}^{\xi} \left(1 - \cos(2\pi \bar{y}(\xi) ) \right) \bar{y}_{\xi}(\xi) d\xi \\ &= V_0(\xi) - \frac{\pi^2}{2}\alpha \int_{\bar{y} \left(\xi_{1, \textrm{l}}(t)\right)}^{\bar{y}(\xi)} \left(1 - \cos(2\pi z) \right)dz \\
    &= \frac{\pi}{4} \left(1 - \alpha \right) \left(2\pi \bar{y}(\xi) -\sin \left(2\pi \bar{y}(\xi) \right) \right) \\ & \qquad - \frac{\pi}{4}\alpha \left( \sin \left(2 \arcsin \left(\frac{2}{\pi t} \right) \right) - 2\arcsin \left(\frac{2}{\pi t} \right)  \right) \! . 
\end{align*}
Inserting $\xi = \xi_{1, \mathrm{r}}(t)$ and using \eqref{eq:cos_curves_half}, we find that the total accumulated energy loss at $\xi_{1, \textrm{r}}(t)$ equals
\begin{equation*}
    V_{1, \textrm{d}}(t) = \frac{\pi}{2} \alpha \left( \pi + \sin\left(2 \arcsin \left(\frac{2}{\pi t} \right) \right) - 2\arcsin \left(\frac{2}{\pi t} \right)  \right).
\end{equation*}

For $\xi \in (\xi_{1, \textrm{r}}(t), \xi_{2, \textrm{l}}(t) )$, no wave breaking occurs, but one has to take the total accumulated energy loss up to the point $\xi_{1, \textrm{r}}(t)$ into consideration. Therefore, for $\xi \in (\xi_{1, \textrm{r}}(t), \xi_{2, \textrm{l}}(t))$, we have 
\begin{equation*}
 V(t, \xi) = V_0(\xi) - V_{1, \textrm{d}}(t). 
 \end{equation*}
The remaining cases for $\xi \geq \xi_{2, \textrm{l}}(t)$ follows by symmetry. To summarize, for $t \geq \frac{2}{\pi}$ we have
\begin{align}
\begin{aligned}
    V(t, \xi) &= \begin{cases}
        0, & \xi \leq 0, \\ 
        \frac{\pi}{4} (2\pi \bar{y}(\xi) - \sin(2\pi \bar{y}(\xi))), & 0 < \xi < \xi_{1, \textrm{l}}(t), \\
        \frac{\pi}{4}(1-\alpha)(2\pi \bar{y}(\xi) - \sin(2\pi \bar{y}(\xi))) \\
        \quad - \frac{\pi}{4}\alpha \left( \sin \left(2 \arcsin \left( \frac{2}{\pi t} \right) \right) - 2\arcsin \left(\frac{2}{\pi t} \right) \right), & \xi_{1, \textrm{l}}(t) \leq \xi \leq \xi_{1, \textrm{r}}(t), \\
        \frac{\pi}{4} (2\pi \bar{y}(\xi) - \sin(2\pi \bar{y}(\xi))) - \frac{\pi^2}{2}\alpha \\ \quad - \frac{\pi}{2}\alpha \left( \sin \left(2 \arcsin \left( \frac{2}{\pi t} \right) \right) - 2\arcsin \left(\frac{2}{\pi t} \right)   \right), & \xi_{1, \textrm{r}}(t) < \xi < \xi_{2, \textrm{l}}(t), \\
        \frac{\pi}{4} (1-\alpha)(2\pi \bar{y}(\xi) - \sin(2\pi \bar{y}(\xi))) + \frac{\pi^2}{2}\alpha \\ \quad - \frac{3\pi}{4}\alpha \left( \sin \left(2 \arcsin \left( \frac{2}{\pi t} \right) \right) - 2\arcsin \left( \frac{2}{\pi t} \right)  \right), & \xi_{2, \textrm{l}}(t) \leq \xi \leq \xi_{2, \textrm{r}}(t), \\ 
        \frac{\pi}{4}(2\pi \bar{y}(\xi) - \sin(2\pi \bar{y}(\xi))) - \pi^2 \alpha \\ \quad - \pi \alpha \left(\sin \left(2 \arcsin \left( \frac{2}{\pi t} \right) \right) - 2 \arcsin \left( \frac{2}{\pi t} \right) \right), & \xi_{2, \mathrm{r}}(t) < \xi \leq 4 + 2\pi^2, \\ 
        \pi^2(2-\alpha) 
        \\ \quad - \pi\alpha \left(\sin \left(2 \arcsin \left( \frac{2}{\pi t} \right) \right) - 2 \arcsin \left( \frac{2}{\pi t} \right) \right), & 4 + 2\pi^2 < \xi. 
    \end{cases}
    \raisetag{-15\normalbaselineskip}
\end{aligned}
\label{eq:cos_V_later}
\end{align}
The total Eulerian energy can be read off from \eqref{eq:cos_sol_before} and \eqref{eq:cos_V_later}, in particular,  
\begin{align*}
    F_{\infty}(t) &= \begin{cases}
        2 \pi^2, & 0 \leq t < \frac{2}{\pi}, \\ 
         \pi^2(2-\alpha) - \pi\alpha \left(\sin \left(2 \arcsin \left( \frac{2}{\pi t} \right) \right) - 2 \arcsin \left( \frac{2}{\pi t} \right) \right), & \frac{2}{\pi} \leq t.
     \end{cases}
\end{align*}
It remains to compute $U(t, \xi)$ and $y(t, \xi)$ for $t \geq \frac{2}{\pi}$. We only carry out the details for $U(t, \xi)$ in the interval $(0, \xi_{2, \textrm{l}}(t))$. To compute $U(t, \xi)$ the following definite integrals are needed, 
\begin{equation}
    \int_{\frac{2}{\pi}}^t \arcsin \left(\frac{2}{\pi s} \right)ds = - 1 + t\arcsin \left(\frac{2}{\pi t} \right) + \frac{2}{\pi}\arcosh\left (\frac{\pi}{2}t \right), 
    \label{eq:cos_int1}
\end{equation}
and 
\begin{align}
    \int_{\frac{2}{\pi}}^t \sin \left(2\arcsin \left(\frac{2}{\pi s} \right) \right)ds &= - \frac{4}{\pi} \ln \left (\frac{2}{\pi} \right) - \frac{4}{\pi} \sqrt{1 - \left(\frac{2}{\pi t}\right)^2} \nonumber \\ 
    & \qquad  + \frac{4}{\pi} \ln\left( \sqrt{t^2 - \left(\frac{2}{\pi}\right)^2} +t \right). \label{eq:cos_int2} 
\end{align}
Thus, by combining \eqref{eq:cos_V_later} with \eqref{eq:cos_int1} and \eqref{eq:cos_int2}, we define  
\begin{align}
    b(t) &:= \int_{\frac{2}{\pi}}^tV_{\infty}(s)ds \nonumber \\ & = \pi^2 \left(2-\alpha \right) \left(t - \frac{2}{\pi} \right) +2\pi \alpha  \left(t\arcsin \left(\frac{2}{\pi t} \right) - 1 + \frac{2}{\pi} \arcosh\left(\frac{\pi}{2}t \right) \right)  \nonumber \\ & \qquad + 4\alpha \left(\sqrt{1 - \left(\frac{2}{\pi t}\right)^2} + \ln \left(\frac{2}{\pi}\right) - \ln \left(\sqrt{t^2 - \left(\frac{2}{\pi}\right)^2} + t \right) \right) \!. 
\label{eq:cos_b}
\end{align}
For $\xi \in (0,  \xi_{1, \textrm{l}}(t) ) \cup (\xi_{1, \textrm{r}}(t), \xi_{2, \textrm{l}}(t))$, we will use 
\begin{equation}
    U(t, \xi) = U \left( \frac{2}{\pi}, \xi \right) + \frac{1}{2}\int_{\frac{2}{\pi}}^t \left(V(s, \xi) - \frac{1}{2}V_{\infty}(s) \right)ds.
    \label{eq:cos_rec_u}
\end{equation}
Consequently, for $\xi \in (0, \xi_{1, \textrm{l}}(t))$ we find, equipped with \eqref{eq:cos_b} and \eqref{eq:cos_rec_u}, 
\begin{equation}
    U(t, \xi) = \cos( \pi \bar{y}(\xi)) - \pi - \frac{1}{4}b(t) + \frac{\pi}{8} \left(2 \pi \bar{y}(\xi) - \sin \big(2 \pi \bar{y}(\xi) \big) \right)t. 
    \label{eq:cos_left}
\end{equation}
For $\xi \in (\xi_{1, \textrm{l}}(t),\xi_{1, \textrm{l}}(\frac{2}{\pi}))$ wave breaking has already occurred. Hence, instead of integrating \eqref{eq:integrated_ODE2} from $\frac{2}{\pi}$ to $t$, we integrate from $\tau(\xi)$ to $t$. Moreover, as $U$ is continuous in time, we compute $U(\tau(\xi), \xi)$ from \eqref{eq:cos_left}. Thus
\begin{align*}
    U(t, \xi) &= \cos(\pi \bar{y}(\xi)) - \pi  - \frac{1}{8} \left(b(t) + b \left(\tau(\xi) \right) \right) 
    \\& \quad + \frac{\pi}{8}  \left(2\pi \bar{y}(\xi) - \sin \left(2\pi \bar{y}(\xi) \right) \right) \left(t - \alpha (t-\tau(\xi)) \right) - \frac{\pi^2}{8}(2-\alpha)(t - \tau(\xi)). 
\end{align*}
Next, for $\xi \in (\xi_{1, \textrm{r}}(t), \xi_{2, \textrm{l}}(t))$ using \eqref{eq:cos_rec_u} we obtain
\begin{equation}
    U(t, \xi) = \cos(\pi \bar{y}(\xi)) - \frac{\pi^2}{2}t + \frac{\pi}{8} \left(2\pi \bar{y}(\xi) - \sin \left(2\pi \bar{y}(\xi) \right) \right)t.
    \label{eq:cos_right}
\end{equation}
Finally, as characteristics labeled by $\xi \in (\xi_{1, \textrm{l}}(\frac{2}{\pi}), \xi_{1, \textrm{r}}(t))$ have experienced wave breaking, using \eqref{eq:cos_right} and integrating \eqref{eq:integrated_ODE2} from $\tau(\xi)$ to $t$, we get 
\begin{align*}
    U(t, \xi) &= \cos(\pi \bar{y}(\xi) ) - \frac{1}{8} \big(b(t) - b(\tau(\xi)) \big) - \frac{\pi^2}{8} \big((2-\alpha)t + (2+\alpha)\tau(\xi)\big)
    \\ & \quad + \frac{\pi}{8}  \big(2\pi \bar{y}(\xi) - \sin(2\pi \bar{y}(\xi) ) \big) \big(t - \alpha \big(t - \tau(\xi) \big) \big). 
\end{align*}
Following the same lines of reasoning for the remaining cases, for $t \geq \frac{2}{\pi}$, yields
\begin{align*}
    U(t, \xi) &= \begin{cases}
        1 - \pi - \frac{1}{4}b(t), & \xi \leq 0, 
        \\
        \cos(\pi \bar{y}(\xi)) - \pi - \frac{1}{4}b(t) \\ \quad + \frac{\pi}{8}(2\pi \bar{y}(\xi) - \sin(2\pi \bar{y}(\xi)))t, & 0 < \xi \leq \xi_{1, \textrm{l}}(t), 
        \\ 
        \cos(\pi \bar{y}(\xi)) - \pi - \frac{1}{8}(b(t) + b(\tau(\xi))) 
        \\ 
        \quad + \frac{\pi}{8}(2\pi \bar{y}(\xi) - \sin(2\pi \bar{y}(\xi)))(t - \alpha (t-\tau(\xi))) 
        \\ \quad - \frac{\pi^2}{8}(2-\alpha)(t- \tau(\xi)), & \xi_{1, \textrm{l}}(t) < \xi \leq \xi_{1, \textrm{l}}(\tfrac{2}{\pi}), 
        \\
        \cos(\pi\bar{y}(\xi)) - \frac{1}{8}(b(t) -  b(\tau(\xi))) 
        \\ 
        \quad + \frac{\pi}{8}(2\pi \bar{y}(\xi) - \sin(2\pi \bar{y}(\xi)))(t - \alpha (t-\tau(\xi))) 
        \\ \quad - \frac{\pi^2}{8}((2 - \alpha)t +(2+\alpha)\tau(\xi)), & \xi_{1, \textrm{l}}(\tfrac{2}{\pi}) < \xi \leq \xi_{1, \textrm{r}}(t), 
        \\ 
        \cos(\pi \bar{y}(\xi)) - \frac{\pi^2}{2}t + \frac{\pi}{8}(2\pi \bar{y}(\xi) - \sin(2\pi \bar{y}(\xi)))t, & \xi_{1, \textrm{r}}(t) < \xi \leq  \xi_{2, \textrm{l}}(t), \\
        \cos(\pi \bar{y}(\xi))+ \frac{1}{8}(b(t) - b(\tau(\xi))) 
        \\ 
        \quad + \frac{\pi}{8}(2\pi \bar{y}(\xi) - \sin(2\pi \bar{y}(\xi)))(t - \alpha(t-\tau(\xi))) 
        \\ \quad - \frac{\pi^2}{8}((  6 - 5\alpha)t + (-2 + 5\alpha )\tau(\xi)), 
        & \xi_{2, \textrm{l}}(t) < \xi \leq \xi_{2, \textrm{l}}(\tfrac{2}{\pi}), \\
        \cos(\pi \bar{y}(\xi)) + \pi + \frac{1}{8}(b(t) + b(\tau(\xi))) 
        \\ 
        \quad + \frac{\pi}{8}(2\pi \bar{y}(\xi) - \sin(2\pi \bar{y}(\xi)))(t - \alpha(t-\tau(\xi))) 
        \\ \quad - \frac{\pi^2}{8}( (6-5\alpha)t + (2 + 5\alpha)\tau(\xi)), 
        & \xi_{2, \textrm{l}}(\tfrac{2}{\pi}) < \xi \leq \xi_{2, \textrm{r}}(t), \\
        \cos(\pi \bar{y}(\xi)) + \pi + \frac{1}{4}b(t) \\ \quad + \frac{\pi}{8}(2\pi \bar{y}(\xi) - \sin(2\pi \bar{y}(\xi)))t - \pi^2 t, & \xi_{2, \textrm{r}}(t) < \xi \leq 4 + 2\pi^2, \\
        1 + \pi + \frac{1}{4}b(t), & 4 + 2\pi^2 < \xi, 
    \end{cases}
\end{align*}
where $\xi_{1, \textrm{l}}(\frac{2}{\pi}) =  \frac{1}{4}(2 + \pi^2)$ and  $\xi_{2, \textrm{l}}(\frac{2}{\pi}) = \frac{5}{4}(2+ \pi^2)$. 
Define $B(t)$ by the following
\begin{align*}
    B(t) &:= \int_{\frac{2}{\pi}}^t b(s)ds 
    \\ 
    &=  \frac{\pi^2}{2}(2-\alpha) \left(t - \frac{2}{\pi} \right)^2 + 2\alpha(1 - \pi t) + 4\alpha t \ln \left(\frac{2}{\pi} \right) + 8 \alpha  \sqrt{t^2 - \left(\frac{2}{\pi}\right)^2}
    \\ 
    & \quad + \alpha \left(\pi t^2\arcsin \left( \frac{2}{\pi t} \right) - 2 \sqrt{t^2 - \left( \frac{2}{\pi}\right)^2} + 4 t \arcosh \left(\frac{\pi}{2}t \right) \right) 
    \\
    & \quad  - \alpha \bigg(\frac{8}{\pi} \arctan \left(\frac{\pi}{2}\sqrt{t^2 - \left(\frac{2}{\pi}\right)^2} \right) + 4t\ln \left(\sqrt{t^2 - \left(\frac{2}{\pi}\right)^2} + t \right). 
\end{align*}
Then, by integrating \eqref{eq:integrated_ODE1}, we obtain after some tedious calculations the following expression for $y(t, \xi)$, where we for convenience set $\bar{y} = \bar{y}(\xi)$ and $\tau = \tau(\xi)$, 
\begin{align*}
    y(t, \xi) &= \begin{cases}
        \xi + 1 + (1-\pi)t - \frac{1}{4}B(t), & \xi \leq 0, 
        \\ 
        \bar{y} + 1 + \cos(\pi \bar{y})t - \pi t - \frac{1}{4}B(t)
        \\ 
        \hspace{0.05cm}  + \frac{\pi}{16}(2\pi \bar{y} - \sin(2\pi \bar{y}))t^2, & 0 < \xi < \xi_{1, \textrm{l}}(t), 
        \\
        \bar{y} + 1 + \cos(\pi \bar{y})t - \pi t \\ \hspace{0.05cm} -\frac{1}{8}\left(B(t) + B(\tau) + b(\tau)(t -\tau) \right) 
        \\ 
        \hspace{0.05cm} + \frac{\pi}{16}(2\pi \bar{y} - \sin(2\pi \bar{y}))(t^2 - \alpha(t- \tau)^2) 
        \\ 
        \hspace{0.05cm} - \frac{\pi^2}{16}(2-\alpha)(t- \tau)^2 , & \xi_{1, \textrm{l}}(t) \leq \xi \leq \xi_{1, \textrm{l}}(\tfrac{2}{\pi}), 
        \\ 
        \bar{y} + \cos(\pi \bar{y})t - \frac{1}{8}(B(t) - B(\tau) - b(\tau)(t-\tau))\\ 
        \hspace{0.05cm} + \frac{\pi}{16}(2\pi \bar{y} - \sin(2\pi \bar{y}))(t^2 - \alpha (t- \tau)^2) 
        \\ 
        \hspace{0.05cm} - \frac{\pi^2}{16}((2-\alpha)t^2 +2(2+\alpha)t\tau + (-2 - \alpha) \tau^2), & \xi_{1, \textrm{l}}(\tfrac{2}{\pi}) < \xi \leq \xi_{1, \textrm{r}}(t), \\
        \bar{y} + \cos(\pi \bar{y})t - \frac{\pi^2}{4}t^2 + \frac{\pi}{16}(2\pi \bar{y} - \sin(2\pi \bar{y}))t^2, & \xi_{1, \textrm{r}}(t) < \xi < \xi_{2, \textrm{l}}(t), \\
        \bar{y} + \cos(\pi \bar{y})t + \frac{1}{8}(B(t) - B(\tau) - b(\tau)(t- \tau))
        \\
        \hspace{0.05cm} + \frac{\pi}{16}(2\pi \bar{y} - \sin(2\pi \bar{y}))(t^2 - \alpha (t- \tau)^2) 
       \\ \hspace{0.05cm}- \frac{\pi^2}{16}((6-5\alpha)t^2 +(-4+10\alpha)t\tau + (2-5\alpha)\tau^2), 
        & \xi_{2, \textrm{l}}(t) \leq \xi \leq \xi_{2, \textrm{l}}(\tfrac{2}{\pi}), \\ 
        \bar{y} -1 + \cos(\pi \bar{y})t + \pi t 
        \\ \hspace{0.05cm} + \frac{1}{8}(B(t) + B(\tau) + b(\tau)(t - \tau)) \\ 
        \hspace{0.05cm} + \frac{\pi}{16}(2\pi \bar{y} - \sin(2\pi \bar{y}))(t^2 - \alpha (t- \tau)^2) \\ 
       \hspace{0.05cm} - \frac{\pi^2}{16}((6-5\alpha)t^2 + (4 + 10\alpha) t \tau + (-2- 5\alpha)\tau^2),
         & \xi_{2, \textrm{l}}(\tfrac{2}{\pi}) < \xi \leq \xi_{2, \textrm{r}}(t), \\ 
        \bar{y} - 1 + \cos(\pi \bar{y})t + \pi t + \frac{1}{4}B(t) -\frac{1}{2}\pi^2 t^2 \\ 
        \hspace{0.05cm} + \frac{\pi}{16}(2\pi \bar{y} - \sin(2\pi \bar{y}))t^2, & \xi_{2, \textrm{r}}(t) < \xi \leq 4 + 2\pi^2, \\ 
        \xi - 2\pi^2 - 1 + (1+\pi)t + \frac{1}{4}B(t), & 4 + 2\pi^2 < \xi. 
    \end{cases}
\end{align*}

\end{document}